\documentclass[a4paper,11pt,twoside, eqno]{article}
\usepackage{mathrsfs}
\usepackage{amssymb}
\usepackage{amsmath}
\usepackage{amsthm}
\usepackage{amsfonts}
\usepackage{color}
\usepackage{latexsym}
\usepackage{indentfirst}
\usepackage{amscd}
\usepackage{bbm}
\usepackage{graphicx}
\usepackage[all]{xy}
\usepackage[mathscr]{eucal}
\usepackage{subfigure}

\newcommand{\N}{{\mathbb N}}

\newcommand{\R}{{\mathbb R}}

\allowdisplaybreaks
\newtheorem{theorem}{Theorem}[section]

\newtheorem{corollary}[theorem]{Corollary}

\newtheorem{remark}[theorem]{Remark}
\newtheorem{hypothesis}[theorem]{Hypothesis}
\newtheorem{lemma}[theorem]{Lemma}

\newtheorem{proposition}[theorem]{Proposition}
\newtheorem{claim}[theorem]{Claim}

\numberwithin{equation}{section}

\begin{document}

\title{\bf\Large Morse theory methods for a class of quasi-linear\\ elliptic systems of higher order
\footnotetext{\hspace{-0.35cm} 2020
{\it Mathematics Subject Classification}.
Primary~58E05, 49J52, 49J45.
\endgraf}}
\date{}
\author{Guangcun Lu\footnote{
School of Mathematical Sciences, Beijing Normal University,
Laboratory of Mathematics  and Complex Systems,  Ministry of
  Education,    Beijing 100875, The People's Republic
 of China; \texttt{gclu@bnu.edu.cn}/October 26, 2026.}}

\maketitle

\vspace{-0.7cm}

\begin{center}

\begin{minipage}{13cm}
{\small {\bf Abstract}\quad
We develop the local Morse theory for a class of non-twice continuously differentiable functionals
on Hilbert spaces, including a new generalization of the Gromoll-Meyer's splitting theorem
 and a weaker Marino-Prodi perturbation type result. They are applicable to a wide range of multiple
 integrals with  quasi-linear elliptic Euler equations and systems of higher order.
}
\end{minipage}
\end{center}

\vspace{0.2cm}

\tableofcontents

\vspace{0.2cm}

\section{Introduction }

Since Palais and Smale \cite{Pal, PaSm, Sma} generalized finite-dimensional Morse theory \cite{Mir, Mor}
to nondegenerate $C^2$ functionals on infinite dimensional Hilbert manifolds
and used it to study multiplicity of solutions for semilinear elliptic boundary value problems,
via many people's effort,  such a direction has very successful developments, see a few of nice books
\cite{BaSzWi, Ch, Ch1, MaWi, MoMoPa, PerAO, Skr1, ZouSc} and references therein for details. The Morse theory
for functionals on an infinite dimensional Hilbert manifold
has two main aspects:  Morse relations related critical groups to Betti numbers of
underlying spaces (global), computation of critical groups (local).
Combining use of both is the most effective in applications. The global aspect is well-developed,
for example, $C^1$-smoothness for functionals are sufficient.
The basic tools for the local aspect
mainly consist of Gromoll-Meyer's generalized Morse lemma (or splitting theorem)
 in \cite{GrM}  and the perturbation theorem of Marino and Prodi \cite{MP},
which are stated for  $C^2$ functionals on Hilbert spaces (cf. \cite{Ch, MaWi}).
It is for such reasons that  most of applications of the Morse theory to differential equations are
restricted to  semi-linear elliptic equations  and Hamiltonian systems \cite{Ch, MaWi, MoMoPa}.
Applications to quasi-linear elliptic equations and systems
require a suitable local Morse theory for either non-twice continuously differentiable functionals on Hilbert spaces
or twice continuously differentiable functionals on  Banach  spaces.
There exists significant progress for some special versions of quasi-linear elliptic equations and systems,
e.g. \cite{CaCiMaVa13, Ch83, CiDeVa15, CiDeVa18, De10, Skr1, Tr77, Uh72}, though no satisfactory
local Morse theory in these two cases is developed.

This work is motivated by  studies of quasi-linear elliptic equations and systems of higher order given by
the following multi-dimensional variational problem in (\ref{e:1.3})
under {\textsf{Hypothesis} $\mathfrak{F}_{p,N,m,n}$} on the integrand $F$.
 Since in this situation the functional $\mathfrak{F}$ in
(\ref{e:1.3}) cannot, in general, be of class $C^2$ on its natural domain space
 $W^{m, p}(\Omega, \mathbb{R}^N)$ for $p=2$, the known local Morse theory is helpless.
This requires us to develop the local Morse theory for this class of non-twice continuously differentiable
 functionals on Hilbert spaces, for example, some generalization of the Gromoll-Meyer's splitting theorem
 and some weaker Marino-Prodi perturbation type result.

 Throughout this paper, unless stated otherwise,
we will use the following notations:  For normed linear spaces $X,Y$,
 we denote
 by $X^\ast$ the dual space of $X$, and by $\mathscr{L}(X,Y)$ the space of
 bounded linear operators  from $X$ to $Y$.
We also abbreviate $\mathscr{L}(X):=\mathscr{L}(X,X)$.
Denote by  
$$
B_X(y,r):=\{x\in X \ |\  \|x-y\|_X<r\}\quad\text{and}\quad
\bar{B}_X(y,r):=\{x\in X \ |\  \|x-y\|_X\le r\}
$$
the open and closed balls in $X$ with radius $r$ and centred at $y$, 
respectively. 
The (norm)-closure of a set $S \subset X$ will be denoted by $\overline{S}$ or $Cl(S)$.
Let $m,n\ge 1$  be two  integers,  $\Omega\subset\R^n$  a bounded domain   with  boundary
$\partial\Omega$. Denote the general point of $\Omega$ by $x=(x_1,\cdots,x_n)\in\R^n$
and the element of Lebesgue $n$-measure on $\Omega$ by $dx$.
A {\it multi-index} is an $n$-tuple
$\alpha=(\alpha_1,\cdots,\alpha_n)\in (\mathbb{N}_0)^n$, where $\mathbb{N}_0=\mathbb{N}\cup\{0\}$.
 $|\alpha|:=\alpha_1+\cdots+\alpha_n$ is called the length of $\alpha$.
 Denote by $M(k)$ the number of such
$\alpha$ of length $|\alpha|\le k$, $M_0(k)=M(k)-M(k-1)$, $k=0,\cdots,m$,
where $M(-1)=0$. %(fa 是 $M(-1)=\emptyset$.)
Then $M(0)=M_0(0)$ only consists of
${\bf 0}=(0,\cdots,0)\in (\mathbb{N}_0)^n$.

Let $p\in [2, \infty)$ be a real number,  and let $N\ge 1$, $n>1$ be integers. We make

\noindent{\textsf{Hypothesis} $\mathfrak{F}_{p,N,m,n}$}.\quad
 For each multi-index $\gamma$ as above, let
 \begin{eqnarray*}
 &&p_\gamma\in (2,\infty)\;\hbox{if}\;
 |\gamma|=m-n/p,\qquad p_\gamma=\frac{np}{n-(m-|\gamma|)p}
\;\hbox{if}\; m-n/p<|\gamma|\le m,\\
 &&q_\gamma=1\;\hbox{if}\;|\gamma|<m-n/p,\qquad
 q_\gamma=\frac{p_\gamma}{p_\gamma-1} \;\hbox{if}\;m-n/p\le |\gamma|\le m;
 \end{eqnarray*}
 and for each two multi-indexes $\alpha, \beta$ as above, let
 $p_{\alpha\beta}=p_{\beta\alpha}$ be defined by the conditions
\begin{eqnarray*}
&&p_{\alpha\beta}= 1-\frac{1}{p_\alpha}-\frac{1}{p_\beta}\quad
 \hbox{if}\;|\alpha|=|\beta|=m,\\
&&p_{\alpha\beta}=  1-\frac{1}{p_\alpha}\quad \hbox{if}\;m-n/p\le |\alpha|\le
 m,\; |\beta|<m-n/p,\\
&&p_{\alpha\beta}= 1 \quad \hbox{if}\; |\alpha|, |\beta|<m-n/p, \\
&& 0<p_{\alpha\beta}<1-\frac{1}{p_\alpha}-\frac{1}{p_\beta}\quad\hbox{if}\;|\alpha|,\;|\beta|\ge
 m-n/p,\;|\alpha|+|\beta|<2m.
\end{eqnarray*}
For $M_0(k)=M(k)-M(k-1)$, $k=0,1,\cdots,m$ as above,
we write $\xi\in \prod^m_{k=0}\mathbb{R}^{N\times M_0(k)}$ as
$\xi=(\xi^0,\cdots,\xi^m)$, where
 $ \xi^0=(\xi^1_{\bf 0},\cdots,\xi^N_{\bf 0})^T\in \mathbb{R}^{N}$ and for $k=1,\cdots,m$,
 $\xi^k=\left(\xi^i_\alpha\right)\in \mathbb{R}^{N\times M_0(k)}$, where $1\le i\le N$ and $|\alpha|=k$.
  Denote by  
  $$
  \xi^k_\circ=\{\xi^k_\alpha\,:\,|\alpha|<m-n/p\},\quad k=1,\cdots,N.
  $$
   Let
$\overline\Omega\times\prod^m_{k=0}\mathbb{R}^{N\times M_0(k)}\ni (x,
\xi)\mapsto F(x,\xi)\in\R$
be twice continuously differentiable in $\xi$ for almost all $x$,
measurable in $x$ for all values of $\xi$, and $F(\cdot,\xi)\in L^1(\Omega)$ for $\xi=0$.
Suppose that derivatives of $F$ fulfill  the following properties:\\
{\bf (i)}  For $i=1,\cdots,N$ and $|\alpha|\le m$, functions
 $F^i_\alpha(x,\xi):= F_{\xi^i_\alpha}(x,\xi)$ for $\xi=0$
belong to $L^1(\Omega)$ if $|\alpha|<m-n/p$,
and to $L^{q_\alpha}(\Omega)$ if $m-n/p\le |\alpha|\le m$.\\
{\bf (ii)} There exists a continuous, positive, nondecreasing functions $\mathfrak{g}_1$ such that
for $i,j=1,\cdots,N$ and $|\alpha|, |\beta|\le m$  functions
$\overline\Omega\times\R^{M(m)}\to\R,\; (x, \xi)\mapsto
F^{ij}_{\alpha\beta}(x,\xi):=F_{\xi^i_\alpha\xi^j_\beta}(x,\xi)$
satisfy:
\begin{eqnarray}\label{e:1.1}
 |F^{ij}_{\alpha\beta}(x,\xi)|\le
\mathfrak{g}_1(\sum^N_{k=1}|\xi_\circ^k|)\left(1+
\sum^N_{k=1}\sum_{m-n/p\le |\gamma|\le
m}|\xi^k_\gamma|^{p_\gamma}\right)^{p_{\alpha\beta}}.
\end{eqnarray}
{\bf (iii)} There exists a continuous, positive, nondecreasing functions $\mathfrak{g}_2$ such that
\begin{eqnarray}\label{e:1.2}
\sum^N_{i,j=1}\sum_{|\alpha|=|\beta|=m}F^{ij}_{\alpha\beta}(x,\xi)\eta^i_\alpha\eta^j_\beta\ge
\mathfrak{g}_2(\sum^N_{k=1}|\xi^k_\circ|)\Biggl(1+ \sum^N_{k=1}\sum_{|\gamma|=m}|\xi^k_\gamma|\Biggr)^{p-2}
\sum^N_{i=1}\sum_{|\alpha|= m}(\eta^i_\alpha)^2
\end{eqnarray}
for any $\eta=(\eta^{i}_{\alpha})\in\R^{N\times M_0(m)}$.

{\it Note}: {\bf (a)} If $m\le n/p$ the functions  $\mathfrak{g}_1$ and $\mathfrak{g}_2$ should be understand
as positive constants.\\
 {\bf (b)} For $N=1$ the conditions in {\textsf{Hypothesis} $\mathfrak{F}_{p,N,m,n}$} were introduced in
\cite[\S3.1]{Skr1} (also see \cite[\S1.2]{Skr2} and \cite[p. 110,118]{Skr3});
but it was only required that $p_\gamma\in (0,\infty)$ if $|\gamma|=m-n/p$ there.
We modify it as ``$p_\gamma\in (2,\infty)$ if $|\gamma|=m-n/p$"
so as to coincide with the condition ``$0<p_{\alpha\beta}<1-\frac{1}{p_\alpha}-\frac{1}{p_\beta}$
if $|\alpha|=|\beta|=m-n/p$". This is only needed in case $mp\ge n$.\\
{\bf (c)} The controllable growth condition \cite[p. 40]{Gi}
(also called `common condition of Morrey' or `the natural assumption of Ladyzhenskaya and Ural'tseva'
\cite[p. 38,(I)]{Gi}) is stronger than {\textsf{Hypothesis} $\mathfrak{F}_{2,N,1,n}$}, see Proposition~\ref{prop:A.1};
the Lagrangian function  in De Giorgi's example (cf. \cite[p. 54]{Gi})
 satisfies {\textsf{Hypothesis} $\mathfrak{F}_{2,n,1,n}$}, but does not fulfill
the controllable growth condition on $\Omega=B^n_1(0)=\{x\in\mathbb{R}^n\,|\,|x|<1\}$,  $n\ge 3$.

 Let $\Omega\subset \R^n$  be a bounded domain  such that
 the Sobolev embeddings theorems for the spaces $W^{m, p}(\Omega)$ hold.
 For an element of $W^{m, p}(\Omega, \mathbb{R}^N)$,
$\vec{u}=(u^1,\cdots, u^N):\Omega\to\mathbb{R}^N$,
 we  denote by $D^k\vec{u}$ the set $\{D^\alpha u^i\,:\, |\alpha|=k,\; i=1,\cdots,N\}$
 for  $k=1,\cdots,m$,  and form
 the expression $F(x, \vec{u}(x),\cdots, D^m\vec{u}(x))$, in which
 $\vec{u}(x)$ and $D^\alpha u^i(x)$ take the place of $\xi^0$ and  $\xi^i_\alpha$, respectively.
 Let $V=\vec{w}+V_0\subset W^{m, p}(\Omega, \mathbb{R}^N)$, where $V_0$ is
 a closed subspace containing $W^{m,p}_0(\Omega, \mathbb{R}^N)$.
 Consider the variational integral
\begin{equation}\label{e:1.3}
\mathfrak{F}(\vec{u})=\int_\Omega F(x, \vec{u},\cdots, D^m\vec{u})dx,\quad \vec{u}\in V.
\end{equation}
Call critical points of $\mathfrak{F}$ {\it generalized solutions} of the boundary value problem
corresponding to  $V$:
\begin{equation}\label{e:1.4}
\sum_{|\alpha|\le m}(-1)^{|\alpha|}D^\alpha F^i_\alpha(x, \vec{u},\cdots, D^m\vec{u})=0,\quad
i=1,\cdots,N.
\end{equation}

When $N=1$,  we write  $\xi\in\R^{M(m)}$ as  $\xi=\{\xi_\alpha:\,|\alpha|\le m\}$,
  $\xi_\circ=\{\xi_\alpha:\,|\alpha|<m-n/p\}$ (this is empty if $mp\le n$), and $F_{\alpha\beta}(x,\xi)=:F_{\xi_\alpha\xi_\beta}(x,\xi)$.
As stated in  \cite[\S3.4,Lemma 16, \S5.2]{Skr1} and \cite[p. 118-119]{Skr3},
 under \textsf{Hypothesis} $\mathfrak{F}_{p,1,m,n}$  the functional $\mathfrak{F}$ in
(\ref{e:1.3}) is of class $C^1$; and the (derivative) mapping
$\mathfrak{F}':W^{m,p}_0(\Omega)\to [W^{m,p}_0(\Omega)]^\ast$ is
 Fr\'echet differentiable if $p>2$, but only G\^ateaux-differentiable if $p=2$.
 The latter is best possible.  In fact, it was shown on \cite[Chap.5, Sec. 5.1, Theorem~1]{Skr3}:
 {\it If $p=2$, $m=1$ and $F\in C^2(\overline{\Omega}\times\R^1\times\R^n)$
 has uniformly bounded mixed partial derivatives
 $F_{\xi_i\xi_j}$, $F_{\xi_i u}$ and $F_{uu}$
  (therefore $F$ satisfies \textsf{Hypothesis} $\mathfrak{F}_{2,1,m,n}$),
  then the functional $\mathfrak{F}$ on $W^{1,2}_0(\Omega)$  has Fr\'echet second derivative at zero
  if and only if
$F(x,0,\xi)=\sum^n_{i,j=1}a_{ij}(x)\xi_i\xi_j+ \sum^n_{i=1}b_i(x)\xi_i+ c(x)$.}
 So, generally speaking, under \textsf{Hypothesis} $\mathfrak{F}_{2,1,m,n}$
 the known Morse--Palais lemma cannot be used for  $\mathfrak{F}$. Even so, by improving Smale's method in \cite{Sma},
  Skrypnik \cite[Chapter 5]{Skr1}
  obtained Morse inequalities for $\mathfrak{F}$ on $W^{m,2}_0(\Omega)$
    provided that $\mathfrak{F}$ is coercive  and that each critical point $u$ of $\mathfrak{F}$ is nondegenerate
 in the sense that the G\^ateaux derivative of $\mathfrak{F}$ at $u$
 is an invertible bounded linear self-adjoint operator on $W^{m,2}_0(\Omega)$.
(If $p=\dim\Omega=2$ and $F\in C^{k,\alpha}$ for some $\alpha\in (0,1)$ and an integer $k\ge 3$,
it was proved in \cite[Chapter 7, Th.4.4]{Skr3} that  every critical point $u$ of $\mathfrak{F}$ on $W^{m,2}_0(\Omega)$ sits in $C^{k+m-1,\alpha}(\overline{\Omega})$; in fact $u$ is also analytic in $\Omega$ provided that $F$ is analytic in its arguments.)

For effectively using  Morse theory methods to study critical points
 of $\mathfrak{F}$ on $W^{m,2}(\Omega,\mathbb{R}^N)$,
 it is expected that there exists a corresponding  Gromoll-Meyer's splitting theorem
for this functional. Recently,  the author in \cite[Theorem~1.1]{Lu1} proved a generalization of
Gromoll-Meyer's splitting theorem in \cite{GrM} and used it to study  periodic solutions of Lagrangian systems
on compact manifolds which are strongly convex and has quadratic growth on the fibers. It includes
the case of $\dim\Omega=1$ (and similar one appeared  in some optimal control problems \cite{Va1}). Lu \cite[Theorem~1.1]{Lu1}  was also generalized to a class of continuously directional
differentiable functions on Hilbert spaces in \cite[Theorem~2.1]{Lu2}.
Our design of these splitting theorems is based on a key fact that
 the involved solutions  have higher smoothness, which
is  usually satisfied for many one-dimensional variational
problems. Such an assumption of regularity ensured that the implicit function theorem can be used
  in the proofs of \cite[Theorem~1.1]{Lu1} and \cite[Theorem~2.1]{Lu2}.
If $N=1$, $\dim\Omega=2$ and $F$ is smooth enough, we may prove
under \textsf{Hypothesis} $\mathfrak{F}_{2,1,m,2}$ that \cite[Theorem~2.1]{Lu2}
is applicable for the functional $\mathfrak{F}$ on $W^{m,2}_0(\Omega)$.
However, if $\dim\Omega>2$, for the variational problem in
 (\ref{e:1.3}),  it seems helpless
 because of lack of the priori regularity of critical points; see Section~\ref{sec:compare}
 for details. Thus new ideas and methods are needed.
 We need establish an implicit function theorem for only G\^ateaux differentiable map $\mathcal{F}^\prime$.
 After carefully analyzing this map, we propose the following fundamental assumption and
 arrive at the expected goal.

 \begin{hypothesis}\label{hyp:1.1}
{\rm Let $H$ be a Hilbert space with inner product $(\cdot,\cdot)_H$
and the induced norm $\|\cdot\|$, and let $X$ be a dense linear subspace in $H$.
Let  $V$ be an open neighborhood of the origin $\theta\in H$,
and let $\mathcal{L}\in C^1(V,\mathbb{R})$ satisfy $\mathcal{L}'(\theta)=0$.
Assume that the gradient $\nabla\mathcal{L}$ has a G\^ateaux derivative $B(u)\in \mathscr{L}_s(H)$ at every point
$u\in V\cap X$, and that the map $B:V\cap X\to
\mathscr{L}_s(H)$  has a decomposition
$B=P+Q$, where for each $x\in V\cap X$,  $P(x)\in\mathscr{L}_s(H)$ is  positive definitive and
$Q(x)\in\mathscr{L}_s(H)$ is compact, and they also satisfy the following
properties:\\
{\bf (D1)}  All eigenfunctions of the operator $B(\theta)$ that correspond
to non-positive eigenvalues belong to $X$.\\
{\bf (D2)} For any sequence $(x_k)\subset
V\cap X$ with $\|x_k\|\to 0$, $\|P(x_k)u-P(\theta)u\|\to 0$ for any $u\in H$.\\
{\bf (D3)} The  map $Q:V\cap X\to \mathscr{L}(H)$ is continuous at $\theta$ with respect to the topology
on $H$.\\
{\bf (D4)} For any sequence $(x_k)\subset V\cap X$ with $\|x_k\|\to 0$, there exist
 constants $C_0>0$ and $k_0\in\N$ such that
$(P(x_k)u, u)_H\ge C_0\|u\|^2$ for all $u\in H$ and for all $k\ge k_0$.
}
\end{hypothesis}

The condition (D4) is equivalent to (D4*) in \cite{Lu2} by Lemma~\ref{lem:D*}. Lemma~\ref{lem:S.2.4}
shows that Hypothesis~\ref{hyp:1.1} with $X=H$ is hereditary on closed subspaces.

Under Hypothesis~\ref{hyp:1.1}, if $\theta$ is nondegenerate, i.e., ${\rm Ker}(B(\theta))=\{\theta\}$,
 we prove a new generalization of Morse-Palais Lemma, Theorem~\ref{th:S.1.1}. If Hypothesis~\ref{hyp:1.1}
 holds with $X=H$ we establish a new splitting lemma, Theorem~\ref{th:S.1.2}.
 Strategies of their proofs will be given at the end of Section~\ref{sec:S}.
 Actually, we prove a more general parameterized splitting theorem, Theorem~\ref{th:S.5.3},
which will be used to generalize  many bifurcation theorems for potential operators in \cite{Lu8}.
Comparing with splitting lemmas in \cite{Lu1, Lu2},  the new ones may  largely simplify the arguments for
Lagrangian systems in \cite{Lu1}. However, the former may, sometime, provide  more elaborate results,
for example,  as we have done modifying the proof ideas of them  may yield  the desired splitting
 lemma for the Finsler energy functional on the space of $H^1$-curves  in \cite{Lu5}.
 It is not clear how to complete this with the present one.  In accord with Hypothesis~\ref{hyp:1.1},
 a weaker Marino-Prodi perturbation type result, Theorem~\ref{th:MP.2}, is also presented in
Section~\ref{sec:MP}.

In Section~\ref{sec:Funct}, we first list some fundamental analytic properties of the functional
 $\mathfrak{F}$ under  Hypothesis~$\mathfrak{F}_{p,N,m,n}$. In particular, Corollary~\ref{cor:4.4}
   shows that Hypothesis~$\mathfrak{F}_{2,N,m,n}$ assures  $\mathfrak{F}$ to satisfy  Hypothesis~\ref{hyp:1.1}
on any closed subspace of $W^{m,2}(\Omega, \mathbb{R}^N)$ for a bounded Sobolev domain
$\Omega\subset\mathbb{R}^n$.   Their proofs  are not difficult, but cumbersome, and may  be completed by non-essentially changing that of \cite[Theorem~3.1]{Lu7}.
Then we  are only satisfied to give Morse inequalities and some corollaries.
Finally, we also make compares with previous work and explore applicability of them in Section~\ref{sec:compare}.

Further essential applications may be found in the sequel papers  \cite{Lu8,Lu10}.
We showed in \cite{Lu8} that Theorem~\ref{th:S.3.1} can be effectively
used to generalize some famous bifurcation theorems for potential operators, which leaded to
 many bifurcation results for quasi-linear elliptic Euler equations and systems of higher order.
Using the theory developed in this paper we can also generalize the results in \cite{Lu1}
to a class of Lagrangian systems of higher order with lower smoothness conditions for Lagrangians.

%\noindent{\bf Acknowledgements}. The author is grateful to the anonymous referees for useful remarks.

%We also thank the referee for very detailed feedback for improvement

\section{The splitting lemmas for a class of non-$C^2$ functionals}\label{sec:S}
\setcounter{equation}{0}

\subsection{Statements of main results}\label{sec:S.1}

We always assume that Hypothesis~\ref{hyp:1.1} holds
 without special statements. Then it implies that $\nabla\mathcal{L}$ is of class $(S)_+$ near $\theta$
as proved in \cite[p.2966-2967]{Lu2}. In particular, $\mathcal{L}$
satisfies the (PS) condition near $\theta$.

Let $H=H^+\oplus H^0\oplus H^-$ be the orthogonal decomposition
according to the positive definite, null and negative definite spaces of $B(\theta)$.
Denote by $P^\ast$ the orthogonal projections onto $H^\ast$, $\ast=+,0,-$.
By \cite[Proposition~B.2]{Lu2} Hypothesis~\ref{hyp:1.1}
implies  that there exists a constant $C_0>0$ such that
each $\lambda\in (-\infty, C_0)$ is either not in the spectrum $\sigma(B(\theta))$ or is an
isolated point of $\sigma(B(\theta))$ which is also an eigenvalue of finite multiplicity.
It follows that both $H^0$ and $H^-$ are finitely dimensional, and that
there exists a small $a_0>0$ such that $[-2a_0,
2a_0]\cap\sigma(B(\theta))$ at most contains a point $0$, and hence
\begin{equation}\label{e:S.1.1}
  (B(\theta)u, u)_H\ge 2a_0\|u\|^2\quad\forall u\in H^+,\quad
  (B(\theta)u, u)_H\le -2a_0\|u\|^2\quad\forall u\in H^-.
\end{equation}
Note that (D1) implies $H^-\oplus H^0\subset X$.
$\mu:=\dim H^-$ and $\nu:=\dim H^0$ are called the {\it Morse index} and
{\it nullity} of the critical point $\theta$.
In particular, if $\nu=0$ the critical point $\theta$ is said to be {\it nondegenerate}.
Without special statements, all nondegenerate critical points in this paper are in the sense
of this definition. Moreover, such a critical point must be isolated by (\ref{e:S.3.2}).

Our first result is the following Morse-Palais Lemma, a special case of Theorem~\ref{th:S.3.1}.
 Comparing with that of \cite[Remark~2.2(i)]{Lu2},  the smoothness of $\mathcal{L}$ is
 strengthened, but  other conditions  are suitably  weakened.

\begin{theorem}\label{th:S.1.1}
Under Hypothesis~\ref{hyp:1.1}, if  $\theta$
is nondegenerate, then it is an isolated critical point, and there exist a small $\epsilon>0$,
an open neighborhood $W$ of $\theta$ in
$H$ and an origin-preserving homeomorphism, $\phi: B_{H^+}(\theta,\epsilon) +
B_{H^-}(\theta,\epsilon)\to W$,
 such that
$$
\mathcal{L}\circ\phi(u^++ u^-)=\|u^+\|^2-\|u^-\|^2,\quad \forall (u^+, u^-)\in B_{H^+}(\theta,\epsilon)\times
B_{H^-}(\theta,\epsilon).
$$
Moreover, if $\hat{H}$ is a closed subspace containing $H^-$, and $\hat{H}^+$ is the orthogonal
complement of $H^-$ in $\hat{H}$, i.e., $\hat{H}^+=\hat{H}\cap H^+$, then
$\phi$ restricts to a homeomorphism $\hat{\phi}:(B_{\hat{H}^+}(\theta,\epsilon) + B_{H^-}(\theta,\epsilon))
\to\hat{W}:=W\cap\hat{H}$, and $\mathcal{L}\circ\hat{\phi}(u^++ u^-)=\|u^+\|^2-\|u^-\|^2$
for all $(u^+, u^-)\in B_{\hat{H}^+}(\theta,\epsilon)\times B_{H^-}(\theta,\epsilon)$.
\end{theorem}

Under the assumptions of this theorem, if $X=H$ we can prove
that $\nabla\mathcal{L}$ is locally
 invertible near $\theta$ in Theorem~\ref{th:S.4.4}.
Theorem~\ref{th:S.1.1} is also key for us to prove  Theorem~\ref{th:S.5.3}, whose special case is:

\begin{theorem}[Splitting Theorem]\label{th:S.1.2}
Let Hypothesis~\ref{hyp:1.1}
hold with $X=H$. Suppose $\nu\ne 0$. Then
there exist small positive numbers
$\epsilon, r, s$, a unique continuous map
$\varphi:B_{H^0}(\theta,\epsilon)\to H^+\oplus H^-$ satisfying
\begin{equation}\label{e:S.1.2}
\varphi(\theta)=\theta\quad\hbox{and}\quad (I-P^0)\nabla\mathcal{L}(z+ \varphi(z))=0\quad\forall z\in B_{H^0}(\theta,\epsilon),
 \end{equation}
an open neighborhood $W$ of $\theta$ in $H$ and an origin-preserving homeomorphism
$$
\Phi: B_{H^0}(\theta,\epsilon)\times
\left(B_{H^+}(\theta, r) +
B_{H^-}(\theta, s)\right)\to W
$$
of form $\Phi(z, u^++ u^-)=z+ \varphi(z)+\phi_z(u^++ u^-)$ with
$\phi_z(u^++ u^-)\in H^+\oplus H^-$  such that
$$
\mathcal{ L}\circ\Phi(z, u^++ u^-)=\|u^+\|^2-\|u^-\|^2+ \mathcal{
L}(z+ \varphi(z))
$$
for all $(z, u^+ + u^-)\in B_{H^0}(\theta,\epsilon)\times
\left(B_{H^+}(\theta, r) +
B_{H^-}(\theta, s)\right)$.
Moreover, $\varphi$ is of class $C^{1-0}$, and
we have also:
\begin{enumerate}
\item[\bf (a)] For
each $z\in B_{H^0}(\theta,\epsilon)$, $\Phi(z, \theta)=z+ \varphi(z)$,
$\phi_z(u^++ u^-)\in H^-$ if and only if $u^+=\theta$;

\item[\bf (b)] The functional $B_{H^0}(\theta,\epsilon)\ni z\mapsto
\mathcal{L}^\circ(z):=\mathcal{ L}(z+ \varphi(z))$ is of class $C^1$ and
$D\mathcal{L}^\circ(z)[v]=D\mathcal{L}(z+\varphi(z))[v]$ for all $v\in H^0$.
If $\mathcal{L}$ is of class $C^{2-0}$, so is $\mathcal{L}^\circ$.
\end{enumerate}
\end{theorem}

Since the map $\varphi$ satisfying (\ref{e:S.1.2}) is unique,
as in \cite{Lu1, Lu2} it is possible to prove in some cases that $\varphi$ and
 $\mathcal{L}^\circ$ are of class $C^1$ and $C^2$, respectively.

Theorems~\ref{th:S.1.1},\ref{th:S.1.2}
 cannot be derived from those of \cite{DHK1}.
In fact, according to the conditions (c) and (d) in \cite[Theorem~1.3]{DHK1}
the functional $\mathcal{L}$ in Theorem~\ref{th:S.1.1} should satisfy:
\begin{description}
\item[(${\bf c}^\prime$)] $\exists\;\eta>0, \delta>0$ such that
$|(B(u)(u+z)-B(\theta)(u+z), h)|<\eta\|u+z\|\cdot\|h\|$ for all $u\in B_{H}(\theta,\delta)$,
$z\in H^0$ and $h\in H\setminus\{\theta\}$;
\item[(${\bf d}^\prime$)] $\exists\; \delta>0$ such that
$\bigr(\nabla\mathcal{L}(z+u^+_1+u^-_1)-\nabla\mathcal{L}(z+u^+_2+u^-_2), (u^+_1-u^+_2)+
(u^-_1-u^-_2)\bigl)>0$
for all $(u^+_1, u^-_1), (u^+_2, u^-_2)\in B_{H^+}(\theta,\delta)\times B_{H^-}(\theta,\delta)$
with $u^+_1+u^-_1\ne u^+_2+u^-_2$.
\end{description}

The former implies
$\|B(u)(u+z)-B(\theta)(u+z)\|\le\eta\|u+z\|$ for all $u\in B_{H}(\theta,\delta)$,
$z\in H^0$; and the latter implies, for some $t\in (0,1)$,
$\bigl(B(z+u^+_2+u^-_2+ tu^++tu^-)(u^++u^-), u^++u^-\bigr)>0$
with $u^+=u^+_1-u^+_2$ and $u^-=u^-_1-u^-_2$.
From these it is not hard to see that under our assumptions the conditions
(${\bf c}^\prime$) and (${\bf d}^\prime$) cannot be satisfied in general.

Let ${\bf K}$ always denote an Abel group (without special statements), and
let $H_q(A,B;{\bf K})$ denote the $q$th relative singular homology group of
a pair $(A,B)$ of topological spaces with coefficients in ${\bf K}$.
For each $q\in\N\cup\{0\}$ {\it the $q$th critical group} (with coefficients in ${\bf K}$)
of $\mathcal{L}$ at $\theta$ is defined by
$C_q(\mathcal{L},\theta;{\bf K})=H_q(\mathcal{L}_c\cap U, \mathcal{L}_c\cap U\setminus\{\theta\};{\bf K})$,
where $c=\mathcal{L}(\theta)$, $\mathcal{L}_c=\{\mathcal{L}\le c\}$ and $U$ is a neighborhood of $\theta$ in $H$.
Under the assumptions of Theorem~\ref{th:S.1.1} we have
$C_q(\mathcal{L},\theta;{\bf K})=\delta_{q\mu}{\bf K}$ as usual.
For the degenerate case, though our $\mathcal{L}^{\circ}$ is only of class $C^1$,
 the proofs in \cite[Theorem~8.4]{MaWi} and \cite[Theorem~5.1.17]{Ch1} (or \cite[Theorem~I.5.4]{Ch})
 may be slightly  modify to get  the following shifting theorem, a special case of Theorem~\ref{th:S.5.4}.

\begin{theorem}[Shifting Theorem]\label{th:S.1.3}
Under the assumptions of Theorem~\ref{th:S.1.2},  if $\theta$ is an
isolated critical point  of $\mathcal{ L}$, then
$C_q(\mathcal{L}, \theta;{\bf K})\cong C_{q-\mu}(\mathcal{
L}^{\circ}, \theta; {\bf K})$ for all $q\in\mathbb{N}_0$. Consequently,
${\rm rank}C_q(\mathcal{L}, \theta;{\bf K})$ is finite for all $q\in\mathbb{N}_0$,
and $C_q(\mathcal{L}, \theta;{\bf K})=0$ if $q<\mu$ or $q>\mu+\nu$.
\end{theorem}

As done for $C^2$ functionals in \cite{Ch,Ch1,MaWi,MoMoPa}
some critical point theorems can be derived from Theorem~\ref{th:S.1.3}.
For example, $C_q(\mathcal{L}, \theta;{\bf K})$
is equal to $\delta_{q\mu}{\bf K}$ (resp.
$\delta_{q(\mu+\nu)}{\bf K}$) if $\theta$ is a local minimizer
(resp. maximizer)  of $\mathcal{L}^{\circ}$, and
$C_q(\mathcal{L}, \theta;{\bf K})=0$ for $q\le\mu$ and $q\ge\mu+\nu$
if $\theta$ is neither a local minimizer nor local  maximizer of
$\mathcal{L}^{\circ}$. Similarly, the corresponding generalizations of  Theorems~2.1, 2.1', 2.2, 2.3
 and Corollary~1.3 in \cite[Chapter II]{Ch} can be obtained with
Theorems~\ref{th:S.1.1}, \ref{th:S.1.2} and their equivariant versions in Section~
\ref{sec:S.6}. In particular, as a generalization of
\cite[Theorem~II.1.6]{Ch} (or \cite[Theorem~5.1.20]{Ch1})
we have

\begin{theorem}\label{th:S.1.4}
Let Hypothesis~\ref{hyp:1.1}
hold with $X=H$, and let  $\theta$ be an isolated
critical point of mountain pass type, i.e.,  $C_1(\mathcal{L}, \theta;{\bf K})\ne 0$.
Suppose that $\nu>0$ and $\mu=0$ imply $\nu=1$. Then
$C_q(\mathcal{L}, \theta;{\bf K})=\delta_{q1}{\bf K}$.
\end{theorem}

When $\nu>0$ and $\mu=1$, $C_{0}(\mathcal{
L}^{\circ}, \theta; {\bf K})\ne 0$ by Theorem~\ref{th:S.1.3}.
We can change $\mathcal{L}^\circ$ outside a very small
neighborhood $\theta\in B_{H^0}(\theta,\epsilon)$ to get
a $C^1$ functional on $H^0$ which is coercive (and so satisfies the
(PS)-condition). Then it follows from $C_{0}(\mathcal{
L}^{\circ}, \theta; {\bf K})\ne 0$ and \cite[Proposition~6.95]{MoMoPa}
that $\theta$ is a local minimizer of $\mathcal{L}^\circ$.
As a generalization of Corollary~3.1 in
\cite[page 102]{Ch} we have also:
Under the assumptions of Theorem~\ref{th:S.1.4},
 if the smallest eigenvalue $\lambda_1$ of $B(\theta)=d^2\mathcal{L}(\theta)$
 is simple whenever $\lambda_1=0$, then $\lambda_1\le 0$, and
  ${\rm index}(\nabla \mathcal{L}, \theta)=-1$.
Theorem~5.1 and Corollary~5.1 in
\cite[page 121]{Ch} are also true if ``$f\in C^2(M,\mathbb{R})$"
and ``Fredholm operators $d^2f(x_i)$" are replaced by
``$f\in C^1(M,\mathbb{R})$ and $\nabla f$ is G\^ateaux differentiable"
and ``under some
chart around $p_i$ the functional $f$ has a representation
that satisfies Hypothesis~\ref{hyp:1.1}", respectively.
We can also generalize many critical point theorems in  \cite{Lu2, Lu6} to the setting above,
for example, combining with \cite{JM} a corresponding result to \cite[Theorem~2.10]{Lu2} may be proved
under suitable assumptions. They will be given in other places.\\

\noindent{\bf Strategies of the proof of Theorem~\ref{th:S.1.2} and arrangements in this section}.\quad
Under the assumptions of Theorem~\ref{th:S.1.2}, no known implicit function
theorems or contraction mapping principles can be used to get $\varphi$ in (\ref{e:S.1.2}), which
is rather different from the case in \cite{Lu1,Lu2}.
The methods in \cite{DHK1} provide a possible way to construct such a $\varphi$.
However, as shown below Theorem~\ref{th:S.1.2}, our assumptions cannot guarantee
the conditions (${\bf c}^\prime$) and (${\bf d}^\prime$) above. Fortunately, it is with Lemma~\ref{lem:S.4.1} and Theorem~\ref{th:S.1.1} that we can complete this construction.

In Section~\ref{sec:S.2} we list some lemmas, and prove a more general parameterized version of  Theorem~\ref{th:S.1.1}.
 It is necessary for a key implicit function theorem for a family of potential operators, Theorem~\ref{th:S.4.3},
 which is proved in Section~\ref{sec:S.4};   we also give an inverse function theorem,  Theorem~\ref{th:S.4.4}, there.
In Section~\ref{sec:S.5} we shall prove a parameterized splitting theorem, Theorem~\ref{th:S.5.3},
and a parameterized shifting theorem, Theorem~\ref{th:S.5.4};
 Theorems~\ref{th:S.1.2},~\ref{th:S.1.3} are special cases of
them, respectively. The equivariant case  is considered in Section~\ref{sec:S.6}.

\subsection{Lemmas and a parameterized version of Theorem~\ref{th:S.1.1}}\label{sec:S.2}

Under Hypothesis~\ref{hyp:1.1} we have
the following two lemmas as proved in \cite{Lu1,Lu2}.

\begin{lemma}\label{lem:S.2.1}
 There exists a function $\omega:V\cap X\to [0, \infty)$  such that $\omega(x)\to 0$ as $x\in V\cap X$ and $\|x\|\to
0$, and that for any $x\in V\cap X$,  $u\in H^0\oplus H^-$ and $v\in H$,
$$
|(B(x)u, v)_H- (B(\theta)u, v)_H |\le \omega(x) \|u\|\cdot\|v\|.
$$
\end{lemma}

\begin{lemma}\label{lem:S.2.2}
There exists a  small neighborhood $U\subset V$ of $\theta$ in $H$
and a number $a_1\in (0, 2a_0]$ such that for any $x\in U\cap X$,
\begin{enumerate}
\item[{\rm (i)}] $(B(x)u, u)_H\ge a_1\|u\|^2\;\forall u\in H^+$;
\item[{\rm (ii)}] $|(B(x)u,v)_H|\le\omega(x)\|u\|\cdot\|v\|\;\forall u\in H^+, \forall v\in
H^-\oplus H^0$;
\item[{\rm (iii)}] $(B(x)u,u)_H\le-a_0\|u\|^2\;\forall u\in H^-$.
\end{enumerate}
\end{lemma}

\begin{lemma}\label{lem:D*}
 (D4) is equivalent to the condition (D4*) in \cite{Lu2}:
\begin{itemize}
\item[\rm (D4*)] There exist positive constants $\eta_0>0$ and  $C'_0>0$ such that
$$
(P(x)u, u)\ge C'_0\|u\|^2\quad\forall u\in H,\;\forall x\in
B_H(\theta,\eta_0)\cap X.
$$
\end{itemize}
\end{lemma}

 Indeed, since each $P(x)$ is a positive definite bounded
linear operator, its spectral set is a bounded closed subset in $(0,\infty)$, and
$\sigma(\sqrt{P(x)})=\{\sqrt{\lambda}\,|\, \lambda\in\sigma(P(x))\}$.
It follows that (D4) implies (D4*).

The following result is easily verified, see \cite{Lu7}.

\begin{lemma}\label{lem:S.2.4}
Suppose that {\rm Hypothesis~\ref{hyp:1.1}} with $X=H$ is satisfied. Then for
any closed subspace  $\hat{H}\subset H$, $(\hat{H}, \hat{V}, \hat{\mathcal{L}})$
satisfies {\rm Hypothesis~\ref{hyp:1.1}} with $X=H$, where
 $\hat{V}:=V\cap\hat{H}$ and $\hat{\mathcal{L}}:=\mathcal{L}|_{\hat{V}}$.
\end{lemma}

For later applications in \cite{Lu8}, we shall prove the following more general version of Theorem~\ref{th:S.1.1}.

\begin{theorem}\label{th:S.3.1}
Under Hypothesis~\ref{hyp:1.1}, let  $\mathcal{G}\in C^1(V,\mathbb{R})$
satisfy: {\rm i)} $\mathcal{G}'(\theta)=\theta$,
{\rm ii)} the gradient $\nabla\mathcal{G}$ has
G\^ateaux derivative $\mathcal{G}''(u)\in \mathscr{L}_s(H)$ at any $u\in V$, and $\mathcal{G}'': V\to \mathscr{L}_s(H)$
 is continuous at $\theta$.
Suppose that the critical point  $\theta$  of $\mathcal{L}$ is a nondegenerate.
 Then there exist $\rho>0$, $\epsilon>0$, a family of open neighborhoods of $\theta$ in
$H$, $\{W_\lambda\,|\, |\lambda|\le\rho\}$
and a family of origin-preserving homeomorphisms, $\phi_\lambda: B_{H^+}(\theta,\epsilon) +
B_{H^-}(\theta,\epsilon)\to W_\lambda$, $|\lambda|\le\rho$,
 such that
$$
(\mathcal{L}+\lambda\mathcal{G})\circ\phi_\lambda(u^++ u^-)=\|u^+\|^2-\|u^-\|^2,
\quad\forall (u^+, u^-)\in B_{H^+}(\theta,\epsilon)\times
B_{H^-}(\theta,\epsilon).
$$
Moreover, $[-\rho,\rho]\times (B_{H^+}(\theta,\epsilon) +
B_{H^-}(\theta,\epsilon))\ni (\lambda, u)\mapsto \phi_\lambda(u)\in H$
is continuous, and $\theta$ is an isolated critical point of each $\mathcal{L}+\lambda\mathcal{G}$. Finally, if $\hat{H}$ is a closed subspace containing $H^-$, and $\hat{H}^+$ is the orthogonal
complement of $H^-$ in $\hat{H}$, i.e., $\hat{H}^+=\hat{H}\cap H^+$, then each
$\phi_\lambda$ restricts to a homeomorphism $\hat{\phi}_\lambda:(B_{\hat{H}^+}(\theta,\epsilon) + B_{H^-}(\theta,\epsilon))
\to\hat{W}_\lambda:=W_\lambda\cap\hat{H}$, and
$(\mathcal{L}+\lambda\mathcal{G})\circ\hat{\phi}_\lambda(u^++ u^-)=\|u^+\|^2-\|u^-\|^2$
for all $(u^+, u^-)\in B_{\hat{H}^+}(\theta,\epsilon)\times B_{H^-}(\theta,\epsilon)$.
\end{theorem}

\begin{proof}[\bf Proof]
Take a small $\epsilon>0$ so that $\bar
B_{H^+}(\theta,\epsilon)\oplus \bar
B_{H^-}(\theta,\epsilon)$ is contained in the open
neighborhood $U$ in Lemma~\ref{lem:S.2.2}. As in Step 3 of the proof of \cite[Theorem~1.1]{Lu1}
(or the proof of \cite[Lemma~3.5]{Lu2}), it follows from
the mean value theorem and Lemma~\ref{lem:S.2.2} that
\begin{eqnarray}\label{e:S.3.1}
&&D\mathcal{L}(u^++u^-_2)[u^-_2-u^-_1]
- D\mathcal{L}(u^++u^-_1)[u^-_2-u^-_1]\le -a_0\|u^-_2-u^-_1\|^2,\\
&&D\mathcal{L}(u^++ u^-)[u^+-u^-]
\ge  a_1\|u^+\|^2+ a_0\|u^-\|^2\label{e:S.3.2}
\end{eqnarray}
for all $u^+\in \bar B_{H^+}(\theta,\epsilon)$ and $u^-_i\in\bar
B_{H^-}(\theta,\epsilon)$, $i=1, 2$. (See \cite{Lu7} for details).

 Since $\mathcal{G}'': V\to \mathscr{L}_s(H)$ is
continuous at $\theta$, as in the proofs
of (\ref{e:S.3.1}) and (\ref{e:S.3.2}) in \cite{Lu7}
we may shrink $\epsilon>0$ and find $\rho>0$ such that
\begin{eqnarray*}
&&|\lambda|\cdot|D\mathcal{G}(u^++u^-_2)[u^-_2-u^-_1]
- D\mathcal{G}(u^++u^-_1)[u^-_2-u^-_1]| \le
\frac{a_0}{2}\|u^-_2-u^-_1\|^2,\\
&&|\lambda D\mathcal{G}(u^++ u^-)[u^+-u^-]|
\le  \frac{a_1}{2}\|u^+\|^2+ \frac{a_0}{2}\|u^-\|^2
\end{eqnarray*}
for all $\lambda\in [-\rho,\rho]$, $u^+\in \bar
B_{H^+}(\theta,\epsilon)$ and $u^-, u^-_i\in\bar
B_{H^-}(\theta,\epsilon)$, $i=1, 2$.
The first inequality and (\ref{e:S.3.1})  lead to
\begin{eqnarray}\label{e:S.3.3}
&&D(\mathcal{L}+\lambda\mathcal{G})(u^++u^-_2)[u^-_2-u^-_1]
- D(\mathcal{L}+\lambda\mathcal{G})(u^++u^-_1)[u^-_2-u^-_1]\nonumber\\
&& \le -\frac{a_0}{2}\|u^-_2-u^-_1\|^2,\quad\forall (\lambda, u^+, u^-)\in[-\rho,\rho]\times\bar B_{H^+}(\theta,\epsilon)\times\bar
B_{H^-}(\theta,\epsilon).
\end{eqnarray}
The latter and (\ref{e:S.3.2})  yield for all $(\lambda, u^+, u^-)\in[-\rho,\rho]\times\bar B_{H^+}(\theta,\epsilon)\times\bar
B_{H^-}(\theta,\epsilon)$,
\begin{eqnarray}\label{e:S.3.4}
D(\mathcal{L}+\lambda\mathcal{G})(u^++ u^-)[u^+-u^-]
\ge  \frac{a_1}{2}\|u^+\|^2+ \frac{a_0}{2}\|u^-\|^2.
\end{eqnarray}
In particular, this implies that $\theta$ is an isolated critical point of each $\mathcal{L}+\lambda\mathcal{G}$ and that
$$
D(\mathcal{L}+\lambda\mathcal{G})(u^+)[u^+]
\ge  \frac{a_1}{2}\|u^+\|^2> p(\|u^+\|),\quad\forall (\lambda, u^+)\in [-\rho,\rho]\times\bar
B_{H^+}(\theta,\epsilon)\setminus\{\theta\},
$$
where $p:(0, \varepsilon]\to (0, \infty)$ is a non-decreasing
function given by $p(t)=\frac{a_1}{4}t^2$.
 This, (\ref{e:S.3.3}) and (\ref{e:S.3.4})
show that the conditions of \cite[Theorem~A.1]{Lu2} are satisfied.
The first two conclusions follow immediately.

For the final claim, note that (\ref{e:S.3.3}) and (\ref{e:S.3.4}) naturally hold
for all $u^+\in\bar{B}_{\hat{H}^+}(\theta, \epsilon)$ and $u^-, u^-_i\in \bar{B}_{H^-}(\theta, \epsilon)$,
$i=1,2$. Carefully checking the proof of \cite[Theorem~A.1]{Lu2} the conclusion is easily obtained.
(Note that this claim seems unable to be directly derived from Lemma~\ref{lem:S.2.4}.)
\end{proof}

\subsection{An implicit function theorem for a family of potential operators}\label{sec:S.4}

Under Hypothesis~\ref{hyp:1.1}, we shall prove an implicit function theorem, Theorem~\ref{th:S.4.3},
which implies the first claim in Theorem~\ref{th:S.1.2}, and  an inverse function theorem, Theorem~\ref{th:S.4.4}.

Take $\epsilon>0$, $r>0$ and $s>0$ so small that the closures of both
$$
\mathcal{Q}_{r,s}:=B_{H^+}(\theta,r)\oplus B_{H^-}(\theta,s)\quad\hbox{and}\quad
B_{H^0}(\theta,\epsilon)\oplus\mathcal{Q}_{r,s}
$$
are contained in the neighborhood $U$ in Lemma~\ref{lem:S.2.2}.
Since $H^0\subset X$,  $X\cap \mathcal{Q}_{r,s}$ is also dense in $\mathcal{Q}_{r,s}$.
Let $P^\bot=I-P^0=P^++P^-$. By Lemma~\ref{lem:S.2.2} we obtain
$a_0'>0, a_1'>0$ such that\footnote{Several corrections have been made from this point to the end of the proof of Lemma~\ref{lem:S.4.1}. }
\begin{eqnarray}
(P^\bot\nabla\mathcal{L}(z+u), u^+)_H-(P^\bot\nabla\mathcal{L}(z), u^+)_H
\ge a_1'\|u^+\|^2-a_0'[\omega(z+u)]^2\|u^-\|^2,\label{e:S.4.1}\\
(P^\bot\nabla\mathcal{L}(z+u), u^-)_H-(P^\bot\nabla\mathcal{L}(z), u^-)_H
\le -a_1'\|u^-\|^2+a_0'[\omega(z+u)]^2\|u^+\|^2\label{e:S.4.2}
\end{eqnarray}
for all $u\in \overline{\mathcal{Q}_{r,s}}$ (the closure of
$\mathcal{Q}_{r,s}$ in $H$)
and $z\in\bar B_{H^0}(\theta,\epsilon)$.
For completeness, we prove (\ref{e:S.4.1}); the proof of (\ref{e:S.4.2}) is analogous.
%For the sake of completeness, we give the proof of (\ref{e:S.4.1});
%a similar argument yields (\ref{e:S.4.2}).
For a given $u\in X\cap \mathcal{Q}_{r,s}$,  the mean value theorem
yields a  $t\in (0,1)$ such that
$$
(P^\bot\nabla\mathcal{L}(z+u), u^+)_H-(P^\bot\nabla\mathcal{L}(z), u^+)_H=
(B(z+tu)u, u^+)_H.
$$
Note that  $tu\in X\cap \mathcal{Q}_{r,s}$ because the latter is
convex. By  Lemma~\ref{lem:S.2.2} we derive
\begin{align*}
( B(z+tu)u, u^+)_H&=( B(z+tu)u^+, u^+)_H+( B(z+tu)u^-, u^+)_H\\
&\ge a_1\|u^+\|^2-
\omega(z+tu)\|u^-\|\cdot\|u^+\|\\
&\ge \frac{a_1}{2}\|u^+\|^2-\frac{1}{2a_1}[(\omega(z+tu)]^2\|u^-\|^2
\end{align*}
since
$$
\omega(z+tu)\|u^-\|\cdot\|u^+\|\le \frac{a_1}{2}\|u^+\|^2+\frac{1}{2a_1}(\omega(z+tu)\|u^-\|)^2.
$$
This proves (\ref{e:S.4.1})
for $a_1'=\frac{a_1}{2}$ and $a_0'=\frac{1}{2a_1}$. 

%%%%%%%%%%%%%%%%%%%%%%%%%%%%%%%%%%%%%%%%%%%%%%%%%%%%%%%%%%%%%%%%%%
%A similar argument yields (2.8). Indeed, we have  a  $t\in (0,1)$ such that
%$$
%(P^\bot\nabla\mathcal{L}(z+u), u^-)_H-(P^\bot\nabla\mathcal{L}(z), u^-)_H=
%(B(z+tu)u, u^-)_H.
%$$
%It follows from Lemma~2.6 that
%\begin{align*}
%( B(z+tu)u, u^-)_H&=( B(z+tu)u^-, u^-)_H+( B(z+tu)u^+, u^-)_H\\
%&\le -a_1\|u^-\|^2+
%\omega(z+tu)\|u^+\|\cdot\|u^-\|\\
%&\le -\frac{a_1}{2}\|u^-\|^2+\frac{1}{2a_1}[(\omega(z+tu)]^2\|u^-\|^2
%\end{align*}
%since
%$$
%\omega(z+tu)\|u^+\|\cdot\|u^-\|\le \frac{a_1}{2}\|u^-\|^2+\frac{1}{2a_1}(\omega(z+tu)\|u^+\|)^2.
%$$
%This proves (2.8) for $a_1'=\frac{a_1}{2}$ and $a_0'=\frac{1}{2a_1}$. 
%%%%%%%%%%%%%%%%%%%%%%%%%%%%%%%%%%%%%%%%%%%%%%%%%%%%%%%%%%%%%%%%

 Since $\omega(z+u)\to 0$ as  $\|z+u\|\to 0$ (with $z+u\in X$), by shrinking
$r>0,s>0$ and $\epsilon>0$ we can require
that $[\omega(z+u)]^2<\frac{a_1'}{2a_0'}$ for all $(z,u)\in \bar B_{H^0}(\theta,\epsilon)\times(\overline{\mathcal{Q}_{r,s}}\cap X)$.
This,  (\ref{e:S.4.1}) and (\ref{e:S.4.2}) lead to, respectively,
\begin{align*}
&(P^\bot\nabla\mathcal{L}(z+u)-P^\bot\nabla\mathcal{L}(z), u^+)_H\ge a_1'\|u^+\|^2-\frac{a_1'}{2}\|u^-\|^2\quad\forall
(z,u)\in \bar B_{H^0}(\theta,\epsilon)\times(\overline{\mathcal{Q}_{r,s}}\cap X),\\
&(P^\bot\nabla\mathcal{L}(z+u)-P^\bot\nabla\mathcal{L}(z), u^-)_H\le -a_1'\|u^-\|^2+\frac{a_1'}{2}\|u^+\|^2\quad\forall
(z,u)\in \bar B_{H^0}(\theta,\epsilon)\times(\overline{\mathcal{Q}_{r,s}}\cap X).
\end{align*}
Those two inequalities also hold true since $\overline{\mathcal{Q}_{r,s}}\cap X$ is dense in $\overline{\mathcal{Q}_{r,s}}$.
Hence for all $u\in \overline{\mathcal{Q}_{r,s}}$, $z_j\in\bar B_{H^0}(\theta,\epsilon)$, $j=1,2$,
and $t\in [0,1]$, we have
\begin{align}
&\bigl(tP^\bot(\nabla\mathcal{L}(z_1+u)-\nabla\mathcal{L}(z_1))+ (1-t)P^\bot(\nabla\mathcal{L}(z_2+u)-\nabla\mathcal{L}(z_2)), u^+\bigr)_H
\nonumber\\
&\ge a_1'\|u^+\|^2-\frac{a_1'}{2}\|u^-\|^2,\quad\label{e:S.4.6}\\
&\bigl(tP^\bot(\nabla\mathcal{L}(z_1+u)-\nabla\mathcal{L}(z_1))+ (1-t)P^\bot(\nabla\mathcal{L}(z_2+u)-\nabla\mathcal{L}(z_2)), u^-\bigr)_H
\nonumber\\
&\le -a_1'\|u^-\|^2+\frac{a_1'}{2}\|u^+\|^2.\quad\label{e:S.4.7}
\end{align}

\begin{lemma}\label{lem:S.4.1}
Let $\Omega=[0,1]\times \bar B_{H^0}(\theta,\epsilon)\times \bar B_{H^0}(\theta,\epsilon)\times
\partial\overline{\mathcal{Q}_{r,s}}$. Then,
for fixed $r>0, s>0$ and sufficiently small $\epsilon>0$,
$$
\inf\{\|tP^\bot\nabla\mathcal{L}(z_1+u)+ (1-t)P^\bot\nabla\mathcal{L}(z_2+u)\|\,\mid\,
 (t,z_1,z_2,u)\in\Omega \}>0.
$$
\end{lemma}

\begin{proof}[\bf Proof]
Since $\partial\overline{\mathcal{Q}_{r,s}}=[(\partial B_{H^+}(\theta,r))\oplus \bar B_{H^-}(\theta,s)]\cup
[ \bar B_{H^+}(\theta,r)\oplus (\partial B_{H^-}(\theta,s))]$, we have
 $\Omega=\Lambda_1\cup\Lambda_2$, where
$\Lambda_1=[0,1]\times \bar B_{H^0}(\theta,\epsilon)\times \bar B_{H^0}(\theta,\epsilon)\times
(\partial B_{H^+}(\theta,r))\oplus \bar B_{H^-}(\theta,s)$ and
$\Lambda_2=[0,1]\times \bar B_{H^0}(\theta,\epsilon)\times \bar B_{H^0}(\theta,\epsilon)\times
B_{H^+}(\theta,r)\oplus (\partial\bar B_{H^-}(\theta,s))$. Firstly, let us prove
\begin{equation}\label{e:S.4.8}
\inf\{\|tP^\bot\nabla\mathcal{L}(z_1+u)+ (1-t)P^\bot\nabla\mathcal{L}(z_2+u)\|\,|\,
 (t,z_1,z_2,u)\in\Lambda_1 \}>0.
\end{equation}
By a contradiction, suppose that there exist sequences $(t_n)\subset [0,1]$ and
$$
(z_n),\,(z_n')\subset \bar B_{H^0}(\theta,\epsilon),\quad (u_n)\subset
(\partial B_{H^+}(\theta,r))\oplus \bar B_{H^-}(\theta,s)
$$
 such that $\|t_nP^\bot\nabla\mathcal{L}(z_n+u_n)+(1-t_n)P^\bot\nabla\mathcal{L}(z_n'+u_n)\|\to 0$.
We can assume
\begin{eqnarray}
&&(t_nP^\bot\nabla\mathcal{L}(z_n+u_n)+(1-t_n)P^\bot\nabla\mathcal{L}(z_n'+u_n), u^+_n)_H
\le \frac{a_1'r^2}{4},\quad\forall n\in\mathbb{N},\quad\label{e:S.4.10}\\
&&(t_nP^\bot\nabla\mathcal{L}(z_n+u_n)+(1-t_n)P^\bot\nabla\mathcal{L}(z_n'+u_n), u^-_n)_H
\ge -\frac{a_1'r^2}{4},\quad\forall n\in\mathbb{N}.\quad\label{e:S.4.11}
\end{eqnarray}
Note that $u^+_n\in \partial B_{H^+}(\theta,r)$ and $u^-_n\in\bar B_{H^-}(\theta,s)$.
So (\ref{e:S.4.10}) and (\ref{e:S.4.6}) lead to
\begin{align*}
\frac{a_1'}{4}r^2&\ge (t_nP^\bot\nabla\mathcal{L}(z_n+u_n)+(1-t_n)P^\bot\nabla\mathcal{L}(z_n'+u_n), u^+_n)_H\\
&=\bigl(t_nP^\bot(\nabla\mathcal{L}(z_n+u_n)-\nabla\mathcal{L}(z_n))+ (1-t_n)P^\bot(\nabla\mathcal{L}(z_n'+u_n)-\nabla\mathcal{L}(z_n')), u^+_n\bigr)_H\\
&\quad+ (\beta_n, u_n^+)_H\\
&\ge  a_1'r^2-\frac{a_1'}{2}\|u^-_n\|^2+ (\beta_n, u_n^+)_H,
\end{align*}
where $\beta_n=t_nP^\bot\nabla\mathcal{L}(z_n)+(1-t_n)P^\bot\nabla\mathcal{L}(z_n')$. 
Put
$$
\eta(\epsilon):=\sup\{\|P^\bot\nabla\mathcal{L}(z)\|\,\mid\, z\in H^0,\, \|z\|\le\epsilon\},
$$
which tends to $0$ as $\epsilon\to 0$ by Hypothesis~\ref{hyp:1.1}.
Clearly, $\|\beta_n\|\le \eta(\epsilon)$ for all $n\in\mathbb{N}$, and
\begin{align*}
\frac{a_1'}{4}r^2&\ge (t_nP^\bot\nabla\mathcal{L}(z_n+u_n)+(1-t_n)P^\bot\nabla\mathcal{L}(z_n'+u_n), u^+_n)_H\\
&\ge  a_1'r^2-\frac{a_1'}{2}\|u^-_n\|^2-\eta(\epsilon)r,
\end{align*}
and therefore
\begin{equation}\label{e:S.4.12}
\|u_n^-\|^2\ge\frac{3}{2}r^2
-\frac{\eta(\epsilon)r}{a_1'},\quad\forall n\in\mathbb{N}.
\end{equation}
Similarly, from (\ref{e:S.4.7}) and (\ref{e:S.4.11})  we conclude that
\begin{align*}
-\frac{a_1'r^2}{4}&\le
(t_nP^\bot\nabla\mathcal{L}(z_n+u_n)+(1-t_n)P^\bot\nabla\mathcal{L}(z_n'+u_n), u^-_n)_H\\
&=\bigl(t_nP^\bot(\nabla\mathcal{L}(z_n+u_n)-\nabla\mathcal{L}(z_n))+ (1-t_n)P^\bot(\nabla\mathcal{L}(z_n'+u_n)-\nabla\mathcal{L}(z_n')), u^-_n\bigr)_H\\
&\quad+ (\beta_n, u_n^-)_H\\
&\le -a_1'\|u^-_n\|^2+\frac{a_1'r^2}{2}+\eta(\epsilon)\|u_n^-\|,
\end{align*}
and hence
$$
\|u_n^-\|^2\le\frac{3}{4}r^2
+\frac{\eta(\epsilon)s}{a_1'},\quad\forall n\in\mathbb{N}.
$$
Combining this with (\ref{e:S.4.12})  yields
$$
\frac{3}{2}r^2
-\frac{\eta(\epsilon)r}{a_1'}\le \frac{3}{4}r^2
+\frac{\eta(\epsilon)s}{a_1'},
$$ 
which gives a contradiction when $\epsilon>0$ is taken sufficiently small, because 
 $\eta(\epsilon)\to 0$ as $\epsilon\to 0$. 
The proof of (\ref{e:S.4.8}) is complete.

 Next, we only need to prove
$$
\inf\{\|tP^\bot\nabla\mathcal{L}(z_1+u)+ (1-t)P^\bot\nabla\mathcal{L}(z_2+u)\|\,|\,
 (t,z_1,z_2,u)\in\Lambda_2 \}>0
$$
again. As above, suppose that  there  exist sequences $(t_n)\subset [0,1]$ and
$$
 (z_n),\,(z_n')\subset \bar B_{H^0}(\theta,\epsilon),\quad
(v_n)\subset
B_{H^+}(\theta,r)\oplus (\partial B_{H^-}(\theta,s))
$$
 such that $\|t_nP^\bot\nabla\mathcal{L}(z_n+v_n)+(1-t_n)P^\bot\nabla\mathcal{L}(z_n'+v_n)\|\to 0$.
 As above we can assume
\begin{eqnarray}
&&(t_nP^\bot\nabla\mathcal{L}(z_n+v_n)+(1-t_n)P^\bot\nabla\mathcal{L}(z_n'+v_n), v^+_n)_H
\le \frac{a_1's^2}{4},\quad\forall n\in\mathbb{N},\quad\label{e:S.4.14}\\
&&(t_nP^\bot\nabla\mathcal{L}(z_n+v_n)+(1-t_n)P^\bot\nabla\mathcal{L}(z_n'+v_n), v^-_n)_H
\ge -\frac{a_1's^2}{4},\quad\forall n\in\mathbb{N}.\quad\label{e:S.4.15}
\end{eqnarray}
Note that $v_n^+\in B_{H^+}(\theta,r)$ and $v_n^-\in\partial B_{H^-}(\theta,s)$ for all $n\in\mathbb{N}$.
Then (\ref{e:S.4.7}) and (\ref{e:S.4.15}) imply
\begin{align*}
-\frac{a_1's^2}{4}&\le (t_nP^\bot\nabla\mathcal{L}(z_n+v_n)+(1-t_n)P^\bot\nabla\mathcal{L}(z_n'+v_n), v_n^-)_H\\
&=\bigl(t_nP^\bot(\nabla\mathcal{L}(z_n+v_n)-\nabla\mathcal{L}(z_n))+ (1-t_n)P^\bot(\nabla\mathcal{L}(z_n'+v_n)-\nabla\mathcal{L}(z_n')), v^-_n\bigr)_H\\
&\quad+ (\beta_n, v_n^-)_H\\
&\le -a_1's^2+\frac{a_1'}{2}\|v_n^+\|^2+ \eta(\epsilon)s,
\end{align*}
where $\beta_n=t_nP^\bot\nabla\mathcal{L}(z_n)+(1-t_n)P^\bot\nabla\mathcal{L}(z_n')$.
It follows that
\begin{equation}\label{e:S.4.16}
\|v_n^+\|^2\ge
\frac{3}{2}s^2-\frac{2\eta(\epsilon)}{a_1'}s,\quad\forall n\in\mathbb{N}.
\end{equation}
By the same argument, (\ref{e:S.4.6}) and (\ref{e:S.4.14}) yield
\begin{align*}
\frac{a_1's^2}{4}&\ge (t_nP^\bot\nabla\mathcal{L}(z_n+v_n)+(1-t_n)P^\bot\nabla\mathcal{L}(z_n'+v_n), v_n^+)_H\\
&=\bigl(t_nP^\bot(\nabla\mathcal{L}(z_n+v_n)-\nabla\mathcal{L}(z_n))+ (1-t_n)P^\bot(\nabla\mathcal{L}(z_n'+v_n)-\nabla\mathcal{L}(z_n')), v^+_n\bigr)_H\\
&\quad+ (\beta_n, v_n^+)_H\\
%&\ge a_1'\|v_n^+\|^2-\frac{a_1'}{2}s^2+ (\beta_n, v_n^+)_H\\
&\ge a_1'\|v_n^+\|^2-\frac{a_1'}{2}s^2- \eta(\epsilon)r
\end{align*}
and so
$$
\|v_n^+\|^2\le\frac{3}{4}s^2+ \frac{\eta(\epsilon)}{a_1'}r,\;\forall n\in\mathbb{N}.
$$
Together with (\ref{e:S.4.16}), this leads to
$$
\frac{3}{4}s^2\le\frac{\eta(\epsilon)}{a_1'}(2s+r),
$$
which gives a contradiction for sufficiently small $\epsilon>0$.
 The desired claim is proved.
\end{proof}

Since (D4) is equivalent to (D4*) by Lemma~\ref{lem:D*}, it was proved in \cite[p. 2966--2967]{Lu2} that
$\nabla\mathcal{L}$ is of class $(S)_+$ under the conditions (S), (F), (C) and (D) in
\cite{Lu2}. In particular, this is also true under the assumptions of Theorem~\ref{th:S.1.1}
(without requirement $H^0=\{\theta\}$).

In the following we always assume that
 $r>0,s>0$ and $\epsilon>0$ are as in Lemma~\ref{lem:S.4.1}.

\begin{lemma}\label{lem:S.4.2}
For each $z\in B_{H^0}(\theta, \epsilon)$, the map
$$
f_z:\overline{\mathcal{Q}_{r,s}}\to H^+\oplus H^-,\; u\mapsto P^\bot\nabla\mathcal{L}(z+u),
$$
is of class $(S)_+$. Moreover, for any two points $z_0, z_1\in B_{H^0}(\theta, \epsilon)$ the map
$$
\mathscr{H}:[0,1]\times \overline{\mathcal{Q}_{r,s}}\to H^+\oplus H^-,\;(t,u)\mapsto
 (1-t)P^\bot\nabla\mathcal{L}(z_0+u)+ tP^\bot\nabla\mathcal{L}(z_1+u)
  $$
  is a homotopy of class $(S)_+$ (cf. \cite[Definition~4.40]{MoMoPa}).
\end{lemma}

\begin{proof}[\bf Proof]
By \cite[Proposition~4.41]{MoMoPa} we only need to prove
 the first claim.
Let a sequence $(u_j)\subset \overline{\mathcal{Q}_{r,s}}$
weakly converge to $u$ in $H^+\oplus H^-$, and satisfy
$\overline{\lim}(P^\bot\nabla\mathcal{L}(z+u_j), u_j-u)_H\le 0$.
It suffices to prove $u_j\to u$ in $H^+\oplus H^-$.
Note that $u_j\rightharpoonup u$ in $H$ because $\overline{\mathcal{Q}_{r,s}}\subset H^+\oplus H^-$.
So is $z+u_j\rightharpoonup z+u$ in $H$.
Moreover, $u_j-u\in H^+\oplus H^-$ implies
\begin{eqnarray*}
(P^\bot\nabla\mathcal{L}(z+u_j), u_j-u)_H
=(\nabla\mathcal{L}(z+u_j), u_j-u)_H
=(\nabla\mathcal{L}(z+u_j), (z+u_j)-(z+u))_H.
\end{eqnarray*}
It follows that $\overline{\lim}(\nabla\mathcal{L}(z+u_j), (z+u_j)-(z+u))_H\le 0$.
But $\nabla\mathcal{L}$ is of class $(S)_+$ near $\theta\in H$, we have
$z+u_j\to z+u$ and so $u_j\to u$.
\end{proof}

Let $\deg$ denote the Browder-Skrypnik degree for  demicontinuous $(S)_+$-maps
(\cite{Bro0, Bro}, \cite{Skr1, Skr2, Skr3}), see \cite[\S4.3]{MoMoPa} for a nice exposition.
By Lemma~\ref{lem:S.4.1} $\deg(f_0, \mathcal{Q}_{r,s}, \theta)$ is well-defined and
using the Poincar\'e-Hopf theorem (cf. \cite[Theorem~1.2]{CiDe}) we have
 \begin{equation}\label{e:S.4.17}
\deg(f_0, \mathcal{Q}_{r,s}, \theta)=\sum^\infty_{q=0}(-1)^q{\rm rank}C_q(f_0,\theta;G).
 \end{equation}
Note that $\mathcal{L}|_{\mathcal{Q}_{r,s}}$ satisfies the conditions of
Theorem~\ref{th:S.1.1} at $\theta\in H^+\oplus H^-$. It follows that $C_q(f_0,\theta;G)=\delta_{\mu q}G$,
where $\mu=\dim H^-$. Hence (\ref{e:S.4.17}) becomes
 \begin{equation}\label{e:S.4.18}
\deg(f_0, \mathcal{Q}_{r,s}, \theta)=(-1)^{\mu}.
 \end{equation}
For each $z\in B_{H^0}(\theta,\epsilon)$,
we derive from Lemma~\ref{lem:S.4.1} that
\begin{eqnarray*}
\inf\{\|f_z(u)\|\,|\, u\in\partial\overline{\mathcal{Q}_{r,s}}\}>0,\quad
\inf\{\|tf_z(u)+(1-t)f_0(u)\|\,|\,t\in [0,1],\; u\in\partial\overline{\mathcal{Q}_{r,s}}\}>0.
\end{eqnarray*}
The former implies  that $\deg(f_z, \mathcal{Q}_{r,s}, \theta)$ is well-defined,
 the latter and  Lemma~\ref{lem:S.4.2} lead to
 \begin{equation}\label{e:S.4.19}
 \deg(f_z, \mathcal{Q}_{r,s}, \theta)=\deg(f_0, \mathcal{Q}_{r,s}, \theta)=(-1)^{\mu}.
 \end{equation}
  So there exists a point $u_z\in  \mathcal{Q}_{r,s}$ such that
  \begin{equation}\label{e:S.4.20}
 P^\bot\nabla\mathcal{L}(z+ u_z)=f_z(u_z)=\theta.
 \end{equation}

\begin{theorem}[Parameterized Implicit Function Theorem]\label{th:S.4.3}
Under the assumptions of Theorem~\ref{th:S.1.2},
suppose further that  $\mathcal{G}_1,\cdots,  \mathcal{G}_n\in C^1(V,\mathbb{R})$
satisfy
\begin{description}
\item[(i)] $\mathcal{G}'_j(\theta)=\theta$, $j=1,\cdots,n$;
\item[(ii)] for each $j=1,\cdots,n$, the gradient $\nabla\mathcal{G}_j$ has  G\^ateaux derivative $\mathcal{G}''_j(u)\in \mathscr{L}_s(H)$ at any $u\in V$, and $\mathcal{G}''_j: V\to \mathscr{L}_s(H)$ is continuous at $\theta$.
\end{description}
Then by shrinking $r>0,s>0$ and $\epsilon>0$ in Lemma~\ref{lem:S.4.1} (if necessary) we have $\delta>0$
and a unique continuous map
 \begin{equation}\label{e:S.4.21}
\psi:[-\delta, \delta]^n\times B_H(\theta,\epsilon)\cap H^0\to \mathcal{Q}_{r,s}\subset
(H^0)^\bot
 \end{equation}
 such that for all $(\vec{\lambda}, z)\in [-\delta, \delta]^n\times B_{H}(\theta,\epsilon)\cap H^0$
 with $\vec{\lambda}=(\lambda_1,\cdots,\lambda_n)$,
$\psi(\vec{\lambda},\theta)=\theta$ and
\begin{equation}\label{e:S.4.22}
 P^\bot\nabla\mathcal{L}(z+ \psi(\vec{\lambda}, z))+
 \sum^n_{j=1}\lambda_j P^\bot\nabla\mathcal{G}_j(z+ \psi(\vec{\lambda}, z))=\theta,
 \end{equation}
 where  $P^\bot$ is  as in (\ref{e:S.4.20}). This $\psi$ also satisfies
 \begin{equation}\label{e:S.4.22.1}
 \|\psi(\vec{\lambda}, z_1)-\psi(\vec{\lambda}, z_2)\|\le 3\|z_1-z_2\|,\quad
 \forall (\vec{\lambda},z)\in [-\delta, \delta]^n\times B_H(\theta,\epsilon)\cap H^0.
 \end{equation}
 Moreover, if $G$ is a compact Lie group acting on $H$ orthogonally,
$V$, $\mathcal{L}$ and all $\mathcal{G}_j$ are $G$-invariant (and hence $H^0$, $(H^0)^\bot$
are $G$-invariant subspaces, and $\nabla\mathcal{L}$, $\nabla\mathcal{G}_j$ are $G$-equivariant),
then  $\psi$ is equivariant with respect to $z$, i.e.,
$\psi(\vec{\lambda}, g\cdot z)=g\cdot\psi(\vec{\lambda},z)$ for $(\vec{\lambda},z)\in
[-\delta, \delta]^n\times B_H(\theta,\epsilon)\cap H^0$ and $g\in G$.
\end{theorem}

\begin{proof}[\bf Proof]
 {\bf Step 1}.\quad 
 There exist numbers $\rho_1, \delta\in (0,1)$ satisfying: (i) $B_H(\theta, 2\rho_1)\subset V$,
 (ii) if  $\vec{\lambda}_k=(\lambda_{k,1},\cdots,\lambda_{k,n})\in [-\delta,\delta]^n$
 converges to  $\vec{\lambda}_0=(\lambda_{0,1},\cdots,\lambda_{0,n})\in [-\delta,\delta]^n$,
 $u_k\in B_H(\theta, 2\rho_1)$ weakly converges to
  $u_0\in B_H(\theta, 2\rho_1)$,  and it also holds that
\begin{eqnarray}\label{e:S.4.23}
\overline{\lim}(\nabla\mathcal{L}(u_k)+\sum^n_{j=1}\lambda_j\nabla\mathcal{G}_j(u_k)  , u_k-u_0)_H\le 0,
\end{eqnarray}
then $u_k\to u_0$. In particular, for each $\vec{\lambda}\in[-\delta, \delta]^n$, the map
$$
{B}_H(\theta, 2\rho_1)\times [0,1]\to H^+\oplus H^-,\;
(t, u)\mapsto P^\bot\nabla\mathcal{L}(u)+ \sum^n_{j=1}t\lambda_j P^\bot\nabla\mathcal{G}_j(u)
  $$
  is a homotopy of class $(S)_+$ (cf. \cite[Definition~4.40]{MoMoPa}).

In fact, by \cite[(5.8)]{Lu2} we have $\rho_1>0$ and $C_0'>0$ such that $B_H(\theta, 2\rho_1)\subset V$
and
\begin{eqnarray}\label{e:S.4.24}
(\nabla\mathcal{L}(u), u-u')_H&\ge&\frac{C_0'}{2}\|u-u'\|^2+ (\nabla\mathcal{L}(u'), u-u')_H\nonumber\\
&+&(Q(\theta)(u-u'), u-u')_H,\quad\forall u,u'\in B_H(\theta, 2\rho_1).
 \end{eqnarray}
Similarly, for each fixed $j\in\{1,\cdots,n\}$, we have $\tau=\tau(u,u')\in (0,1)$ such that
\begin{eqnarray*}
&&(\nabla\mathcal{G}_j(u), u-u')_H=(\nabla\mathcal{G}_j(u)-\nabla\mathcal{G}_j(u'), u-u')_H+ (\nabla\mathcal{G}_j(u'), u-u')_H\\
&=&(\mathcal{G}''_j(\tau u+ (1-\tau)u')(u-u'), u-u')_H+ (\nabla\mathcal{G}_j(u'), u-u')_H\\
&=&([\mathcal{G}''_j(\tau u+ (1-\tau)u')-\mathcal{G}''_j(\theta)](u-u'), u-u')_H+ (\nabla\mathcal{G}_j(u'), u-u')_H\\
&&+ (\mathcal{G}''_j(\theta)(u-u'), u-u')_H,\quad\forall u, u'\in B_H(\theta, 2\rho_1).
 \end{eqnarray*}
Since $V\ni v\mapsto \mathcal{G}''_j(v)\in\mathscr{L}_s(H)$ is continuous at $\theta$,
we may shrink $\rho_1>0$ so that
\begin{eqnarray}\label{e:S.4.24.1}
\|\mathcal{G}''_j(v)-\mathcal{G}''_j(\theta)\|\le\frac{C_0'}{8n},
\quad\forall v\in B_H(\theta, 2\rho_1),\;j=1,\cdots,n.
\end{eqnarray}
It follows that for all $u, u'\in B_H(\theta, 2\rho_1)$ and $j=1,\cdots,n$,
\begin{eqnarray*}
|(\nabla\mathcal{G}_j(u), u-u')_H|\le \frac{C_0'}{8n}\|u-u'\|^2 + |(\nabla\mathcal{G}_j(u'), u-u')_H|
+ |(\mathcal{G}_j''(\theta)(u-u'), u-u')_H|.
 \end{eqnarray*}
Take $\delta\in (0, 1)$ so that
$\delta\sum^n_{j=1}\|\mathcal{G}_j''(\theta)\|<C_0'/8$.
These and (\ref{e:S.4.24}) imply that
\begin{eqnarray*}
&&(\nabla\mathcal{L}(u), u-u')_H+ \sum^n_{j=1}\lambda_j(\nabla\mathcal{G}_j(u), u-u')_H\nonumber\\
&\ge&\frac{C_0'}{4}\|u-u'\|^2+ (\nabla\mathcal{L}(u'), u-u')_H+ (Q(\theta)(u-u'), u-u')_H\nonumber\\
&-& \sum^n_{j=1}|(\nabla\mathcal{G}_j(u'), u-u')_H|,\quad\forall
\vec{\lambda}=(\lambda_1,\cdots,\lambda_n)\in [-\delta, \delta]^n.
 \end{eqnarray*}
Replacing $u, u'$ and $\lambda_j$ by $u_k, u_0$ and $\lambda_{k,j}$ in the inequality, we derive from
(\ref{e:S.4.23}) that $u_k\to u_0$
since  (D3)  implies that
$(\nabla\mathcal{L}(u_0), u_k-u_0)_H\to 0$, $(Q(\theta)(u_k-u_0), u_k-u_0)_H\to 0$ and
$(\nabla\mathcal{G}_j(u_0), u_k-u_0)_H\to 0$.\\

  {\it Note}: The above proof shows that
  the family $\{\mathcal{L}_{\vec{\lambda}}:=\mathcal{L}+ \sum^n_{j=1}\lambda_j\mathcal{G}_j\,|\, \vec{\lambda}\in [-\delta,\delta]^n\}$
  satisfies the (PS) condition on $\bar{B}_H(\theta,\varepsilon)$ for any $\varepsilon<2\rho_1$, that is,
  if  $\vec{\lambda}_k\in [-\delta,\delta]^n$ converges to
  $\vec{\lambda}_0\in [-\delta,\delta]^n$,
 and $u_k\in \bar{B}_H(\theta,\varepsilon)$ satisfies
 $\nabla\mathcal{L}_{\vec{\lambda}_k}(u_k)\to\theta$ and $\sup_k|\mathcal{L}_{\vec{\lambda}_k}(u_k)|<\infty$,
 then $(u_k)$ has a converging subsequence
 $u_{k_i}\to u_0\in \bar{B}_H(\theta,\varepsilon)$ with
 $\nabla\mathcal{L}_{\vec{\lambda}_0}(u_0)=\theta$.\\

\noindent{\bf Step 2}.\quad
Since $\nabla\mathcal{L}$ and $\nabla\mathcal{G}_1,\cdots,\nabla\mathcal{G}_n$ are all locally bounded,
for $r>0,s>0$ and $\epsilon>0$ in Lemma~\ref{lem:S.4.1}, by shrinking them we can assume that
$\bar B_{H^0}(\theta,\epsilon)\times\overline{\mathcal{Q}_{r,s}}
\subset{B}_H(\theta, 2\rho_1)$ and
 \begin{eqnarray}\label{e:S.4.24.2}
 \sup\{\|\nabla\mathcal{L}_{\vec{\lambda}}(z,u)\|\,|\,
 (\vec{\lambda},z,u)\in [-1, 1]^n\times \bar{B}_{H^0}(\theta,\epsilon)\oplus\overline{\mathcal{Q}_{r,s}}\}<\infty.
 \end{eqnarray}
 Then by Lemma~\ref{lem:S.4.1} we may shrink  $\delta\in (0, 1)$ so that
$$
\inf\|tP^\bot(\nabla\mathcal{L}+ \sum^n_{j=1}\lambda_j\nabla\mathcal{G}_j)(z_1+u)+ (1-t)P^\bot(\nabla\mathcal{L}+ \sum^n_{j=1}\lambda_j\nabla\mathcal{G}_j)(z_2+u)\|>0,
 $$
where the infimum is taken for all $(t,z_1,z_2,u)\in[0,1]\times \bar B_{H^0}(\theta,\epsilon)\times \bar B_{H^0}(\theta,\epsilon)\times
\partial\overline{\mathcal{Q}_{r,s}}$ and $(\lambda_1,\cdots,\lambda_n)\in [-\delta,\delta]^n$.
 This implies that
for each $(\vec{\lambda}, z)\in [-\delta,\delta]^n\times B_{H^0}(\theta, \epsilon)$, the map
$$
f_{\vec{\lambda}, z}:\overline{\mathcal{Q}_{r,s}}\ni u\mapsto P^\bot\nabla\mathcal{L}(z+u)+
\sum^n_{j=1}\lambda_j P^\bot\nabla\mathcal{G}_j(z+u)\in H^+\oplus H^-
$$
has a well-defined Browder-Skrypnik degree ${\rm deg}(f_{\vec{\lambda}, z}, \mathcal{Q}_{r,s},\theta)$ and
\begin{eqnarray}\label{e:S.4.25}
{\rm deg}(f_{\vec{\lambda}, z}, \mathcal{Q}_{r,s},\theta)={\rm deg}(f_{\vec{0}, 0}, \mathcal{Q}_{r,s},\theta)
={\rm deg}(f_{0}, \mathcal{Q}_{r,s},\theta)=(-1)^{\mu},
 \end{eqnarray}
where $f_0$ is as in (\ref{e:S.4.18}). Hence for each $(\vec{\lambda}, z)\in [-\delta,\delta]^n\times B_{H^0}(\theta, \epsilon)$
there exists a point $u_{\vec{\lambda},z}\in  \mathcal{Q}_{r,s}$ such that
  \begin{equation}\label{e:S.4.26}
 P^\bot\nabla\mathcal{L}(z+ u_{\vec{\lambda},z})+ \sum^n_{j=1}\lambda_j P^\bot\nabla\mathcal{G}_j(z+ u_{\vec{\lambda},z})=f_{\vec{\lambda},z}(u_{\lambda,z})=\theta.
 \end{equation}
By shrinking the above  $\epsilon>0, r>0$ and $s>0$
 (if necessary), $\omega$ and $a_0, a_1$ in Lemma~\ref{lem:S.2.2} can satisfy
 \begin{equation}\label{e:S.4.27}
\omega(z+u)<\min\{a_0,a_1\}/2,\quad\forall
(z,u)\in\bar B_{H^0}(\theta, \epsilon)\times \overline{\mathcal{Q}_{r,s}}.
\end{equation}

 \noindent{\bf Step 3}.\quad {\it
If $\delta\in (0,1)$ is sufficiently small, then  $u_{\vec{\lambda},z}$ is a unique zero point of $f_{\vec{\lambda},z}$ in $\mathcal{Q}_{r,s}$.}
 In fact, suppose that there exists another different $u_{\vec{\lambda},z}'\in \mathcal{Q}_{r,s}$
  satisfying (\ref{e:S.4.26}). Consider the decomposition
 $u_{\vec{\lambda},z}-u_{\vec{\lambda},z}'=(u_{\vec{\lambda},z}-u_{\vec{\lambda},z}')^++ (u_{\vec{\lambda},z}-
 u_{\vec{\lambda},z}')^-$. We may prove the conclusion in three cases:\\
$\bullet$ $\|(u_{\vec{\lambda},z}-u_{\vec{\lambda},z}')^+\|>\|(u_{\vec{\lambda},z}-u_{\vec{\lambda},z}')^-\|$,\\
$\bullet$ $\|(u_{\vec{\lambda},z}-u_{\vec{\lambda},z}')^+\|=\|(u_{\vec{\lambda},z}-u_{\vec{\lambda},z}')^-\|$,\\
$\bullet$ $\|(u_{\vec{\lambda},z}-u_{\vec{\lambda},z}')^+\|<\|(u_{\vec{\lambda},z}-u_{\vec{\lambda},z}')^-\|$.

Let us write $\mathcal{L}_{\vec{\lambda}}=\mathcal{L}+\sum^n_{j=1}\lambda_j\mathcal{G}_j$ for conveniences.
Then (\ref{e:S.4.26}) implies
\begin{eqnarray}\label{e:S.4.28}
0&=&(P^\bot\nabla\mathcal{L}_{\vec{\lambda}}(z+u_{\vec{\lambda},z})-P^\bot\nabla\mathcal{L}_{\vec{\lambda}}(z+
u_{\vec{\lambda},z}'), (u_{\vec{\lambda},z}-u_{\vec{\lambda},z}')^+)_H\nonumber\\
&=&(P^\bot\nabla\mathcal{L}(z+u_{\vec{\lambda},z})-P^\bot\nabla\mathcal{L}(z+u_{\vec{\lambda},z}'), (u_{\vec{\lambda},z}-u_{\vec{\lambda},z}')^+)_H
\nonumber\\
&&+\sum^n_{j=1}\lambda_j(P^\bot\nabla\mathcal{G}_j(z+u_{\vec{\lambda},z})-P^\bot\nabla\mathcal{G}_j(z+u_{\vec{\lambda},z}'), (u_{\vec{\lambda},z}-u_{\vec{\lambda},z}')^+)_H.
\end{eqnarray}
For simplicity we write $u_{\vec{\lambda},z}$ and $u_{\vec{\lambda},z}'$ as $u_z$ and $u_z'$, respectively.
In the first two cases,  we may use the mean value theorem to get
 $\tau\in (0,1)$ such that
 \begin{eqnarray}\label{e:S.4.29}
&&(P^\bot\nabla\mathcal{L}(z+u_z)-P^\bot\nabla\mathcal{L}(z+u_z'), (u_z-u_z')^+)_H\nonumber\\
&=&(\nabla\mathcal{L}(z+u_z)-\nabla\mathcal{L}(z+u_z'), (u_z-u_z')^+)_H\nonumber\\
&=&(B(z+ \tau u_z+(1-\tau )u_z')(u_z-u_z'), (u_z-u_z')^+)_H\nonumber\\
&=&(B(z+ \tau u_z+(1-\tau )u_z')(u_z-u_z')^+, (u_z-u_z')^+)_H\nonumber\\
&+&(B(z+ \tau u_z+(1-\tau)u_z')(u_z-u_z')^-, (u_z-u_z')^+)_H\nonumber\\
&\ge& a_1\|(u_z-u_z')^+\|^2-\omega(z+ \tau u_z+(1-t)u_z')\|(u_z-u_z')^-\|\cdot\|(u_z-u_z')^+\|\nonumber\\
&\ge& a_1\|(u_z-u_z')^+\|^2-\frac{a_1}{4}[\|(u_z-u_z')^-\|^2+\|(u_z-u_z')^+\|^2]\nonumber\\
&\ge& a_1\|(u_z-u_z')^+\|^2-\frac{a_1}{2}\|(u_z-u_z')^+\|^2\nonumber\\
&=&\frac{a_1}{2}\|(u_z-u_z')^+\|^2,
\end{eqnarray}
where the first inequality comes from  Lemma~\ref{lem:S.2.2}(i)-(ii),
the second is derived from (\ref{e:S.4.27}) and the inequality
$2|ab|\le |a|^2+|b|^2$, and
the third  is because
$\|(u_z-u_z')^-\|\le \|(u_z-u_z')^+\|$.
It follows from  (\ref{e:S.4.28})--(\ref{e:S.4.29}) that
\begin{eqnarray}\label{e:S.4.30}
0&=&(P^\bot\nabla\mathcal{L}_{\vec{\lambda}}(z+u_{\vec{\lambda},z})-P^\bot\nabla\mathcal{L}_{\vec{\lambda}}(z+
u_{\vec{\lambda},z}'), (u_{\vec{\lambda},z}-u_{\vec{\lambda},z}')^+)_H\nonumber\\
&\ge&\sum^n_{j=1}\lambda_j(\mathcal{G}''_j(z+ \tau u_{\vec{\lambda},z}+(1-\tau)u_{\vec{\lambda},z}')(u_{\vec{\lambda},z}-u_{\vec{\lambda},z}'),
(u_{\vec{\lambda},z}-u_{\vec{\lambda},z}')^+)_H\nonumber\\      &&+\frac{a_1}{2}\|(u_{\vec{\lambda},z}-u_{\vec{\lambda},z}')^+\|^2.
\end{eqnarray}
By (\ref{e:S.4.24.1}) we have a constant $M>0$ such that
\begin{eqnarray}\label{e:S.4.30.1}
\sup\{\|\mathcal{G}''_j(z+w)\|\,|\,(z,w)\in \bar B_{H^0}(\theta,\epsilon)\times\overline{\mathcal{Q}_{r,s}},\,j=1,\cdots,n\}<M.
\end{eqnarray}
 From this and the inequality $\|(u_z-u_z')^-\|\le \|(u_z-u_z')^+\|$ we deduce
\begin{eqnarray}\label{e:S.4.31}
&&\sum^n_{j=1}|\lambda_j(\mathcal{G}''_j(z+ \tau u_{\vec{\lambda},z}+(1-\tau)u_{\vec{\lambda},z}')(u_{\vec{\lambda},z}-u_{\vec{\lambda},z}'), (u_{\vec{\lambda},z}-u_{\vec{\lambda},z}')^+)_H|\nonumber\\
&\le& n\delta M\|(u_{\vec{\lambda},z}-u_{\vec{\lambda},z}')\|\cdot\|(u_{\vec{\lambda},z}-u_{\vec{\lambda},z}')^+\|\nonumber\\
&\le&n\delta M[\|(u_{\vec{\lambda},z}-u_{\vec{\lambda},z}')^+\|^2+
\|(u_{\vec{\lambda},z}-u_{\vec{\lambda},z}')^-\|\cdot\|(u_{\vec{\lambda},z}-u_{\vec{\lambda},z}')^+\|]\nonumber\\
&\le &2n\delta M\|(u_{\vec{\lambda},z}-u_{\vec{\lambda},z}')^+\|^2.
\end{eqnarray}
Let us shrink $\delta>0$ in Step 2 so that $\delta<\frac{a_1}{8nM}$. Then (\ref{e:S.4.30}) and (\ref{e:S.4.31}) lead to
\begin{eqnarray*}
0=(P^\bot\nabla\mathcal{L}_{\vec{\lambda}}(z+u_{\vec{\lambda},z})-P^\bot\nabla\mathcal{L}_{\vec{\lambda}}(z+u_{\vec{\lambda},z}'), (u_{\vec{\lambda},z}-u_{\vec{\lambda},z}')^+)_H
\ge\frac{a_1}{4}\|(u_{\vec{\lambda},z}-u_{\vec{\lambda},z}')^+\|^2.
\end{eqnarray*}
This contradicts  $(u_{\vec{\lambda},z}-u_{\vec{\lambda},z}')^+\ne\theta$.

Similarly, for the third case,
as in (\ref{e:S.4.29}) we may use Lemma~\ref{lem:S.2.2}(ii)-(iii) to obtain
\begin{eqnarray*}
&&(P^\bot\nabla\mathcal{L}(z+u_z)-P^\bot\nabla\mathcal{L}(z+u_z'), (u_z-u_z')^-)_H\\
&=&(\nabla\mathcal{L}(z+u_z)-\nabla\mathcal{L}(z+u_z'), (u_z-u_z')^-)_H\\
&=&(B(z+ tu_z+(1-t)u_z')(u_z-u_z'), (u_z-u_z')^-)_H\\
&=&(B(z+ tu_z+(1-t)u_z')(u_z-u_z')^-, (u_z-u_z')^-)_H\\
&+&(B(z+ tu_z+(1-t)u_z')(u_z-u_z')^+, (u_z-u_z')^-)_H\\
&\le& -a_0\|(u_z-u_z')^-\|^2+\omega(z+ tu_z+(1-t)u_z')\|(u_z-u_z')^-\|\cdot\|(u_z-u_z')^+\|\\
&\le& -a_0\|(u_z-u_z')^-\|^2+\frac{a_0}{4}[\|(u_z-u_z')^-\|^2+\|(u_z-u_z')^+\|^2]\\
&\le& -a_0\|(u_z-u_z')^-\|^2+\frac{a_0}{2}\|(u_z-u_z')^-\|^2\\
&=&-\frac{a_0}{2}\|(u_z-u_z')^-\|^2,
\end{eqnarray*}
and hence
\begin{eqnarray}\label{e:S.4.32}
0&=&(P^\bot\nabla\mathcal{L}_{\vec{\lambda}}(z+u_{\vec{\lambda},z})-P^\bot\nabla\mathcal{L}_{\vec{\lambda}}(z+
u_{\vec{\lambda},z}'), (u_{\vec{\lambda},z}-u_{\vec{\lambda},z}')^-)_H\nonumber\\
&\le&\sum^n_{j=1}\lambda_j(\mathcal{G}''_j(z+ \tau u_{\vec{\lambda},z}+(1-\tau)u_{\vec{\lambda},z}')(u_{\vec{\lambda},z}-u_{\vec{\lambda},z}'),
(u_{\vec{\lambda},z}-u_{\vec{\lambda},z}')^-)_H\nonumber\\
&&-\frac{a_0}{2}\|(u_{\vec{\lambda},z}-u_{\vec{\lambda},z}')^+\|^2.
\end{eqnarray}
As in (\ref{e:S.4.31}) we may deduce
\begin{eqnarray*}
&&\sum^n_{j=1}|\lambda_j(\mathcal{G}''_j(z+ \tau u_{\vec{\lambda},z}+(1-\tau)u_{\vec{\lambda},z}')(u_{\vec{\lambda},z}-u_{\vec{\lambda},z}'), (u_{\vec{\lambda},z}-u_{\vec{\lambda},z}')^-)_H|\nonumber\\
&\le& n\delta M\|(u_{\vec{\lambda},z}-u_{\vec{\lambda},z}')\|\cdot\|(u_{\vec{\lambda},z}-u_{\vec{\lambda},z}')^-\|\nonumber\\
&\le&n\delta M[\|(u_{\vec{\lambda},z}-u_{\vec{\lambda},z}')^-\|^2+
\|(u_{\vec{\lambda},z}-u_{\vec{\lambda},z}')^-\|\cdot\|(u_{\vec{\lambda},z}-u_{\vec{\lambda},z}')^+\|]\nonumber\\
&\le &2n\delta M\|(u_{\vec{\lambda},z}-u_{\vec{\lambda},z}')^-\|^2.
\end{eqnarray*}
So if the above $\delta>0$ is also shrunk  so that $\delta<\frac{a_0}{8nM}$,  we may derive from this and (\ref{e:S.4.32}) that
\begin{eqnarray*}
0=(P^\bot\nabla\mathcal{L}_{\vec{\lambda}}(z+u_{\vec{\lambda},z})-P^\bot\nabla\mathcal{L}_{\vec{\lambda}}(z+u_{\vec{\lambda},z}'), (u_{\vec{\lambda},z}-u_{\vec{\lambda},z}')^-)_H
\le-\frac{a_0}{4}\|(u_{\vec{\lambda},z}-u_{\vec{\lambda},z}')^-\|^2,
\end{eqnarray*}
which also leads to a contradiction.
As a consequence, we have a well-defined  map
\begin{eqnarray}\label{e:S.4.33}
\psi: [-\delta,\delta]^n\times B_{H^0}(\theta,\epsilon)\to \mathcal{Q}_{r,s},\;(\lambda, z)\mapsto u_{\vec{\lambda},z}.
\end{eqnarray}

\noindent{\bf Step 4}.\quad{\it  $\psi$ is continuous.}
Let sequences $(\vec{\lambda}_k)\in [-\delta,\delta]^n$ and $(z_k)\subset B_{H^0}(\theta,\epsilon)$
 converge to
$\vec{\lambda}_0\in [-\delta,\delta]^n$ and
$z_0\in B_{H^0}(\theta,\epsilon)$, respectively.
We want to prove that $\psi(\vec{\lambda}_k, z_k)\to\psi(\vec{\lambda}_0, z_0)$.
Since $\psi(\vec{\lambda}_k, z_k)\in \mathcal{Q}_{r,s}$, $k=1,2,\cdots$,
we can suppose $\psi(\vec{\lambda}_k, z_k)\rightharpoonup u_0\in\overline{\mathcal{Q}_{r,s}}$ in $H$.
Noting  $\psi(\vec{\lambda}_k, z_k)-u_0\in H^+\oplus H^-$, by (\ref{e:S.4.26}) we have
\begin{eqnarray*}
\bigr(\nabla\mathcal{L}_{\vec{\lambda}_k}(z_k+\psi(\vec{\lambda}_k, z_k)), \psi(\vec{\lambda}_k,
z_k)-u_0\bigr)
=\bigr(P^\bot\nabla\mathcal{L}_{\vec{\lambda}_k}(z_k+\psi(\vec{\lambda}_k, z_k)), \psi(\vec{\lambda}_k,
z_k)-u_0\bigr)=0.
\end{eqnarray*}
It follows from this and (\ref{e:S.4.24.2}) that
\begin{eqnarray*}
&&|\bigr(\nabla\mathcal{L}_{\vec{\lambda}_k}(z_k+\psi(\vec{\lambda}_k, z_k)), (z_k+\psi(\vec{\lambda}_k,
z_k))-(z_0+u_0)\bigr)|\\
&=&|\bigl(\nabla\mathcal{L}_{\vec{\lambda}_k}(z_n+\psi(\vec{\lambda}_k, z_k)), z_k-z_0\bigr)|
\le\|\nabla\mathcal{L}_{\vec{\lambda}_k}(z_k+\psi(\vec{\lambda}_k, z_k))\|\cdot\|z_k-z_0\|\to 0.
\end{eqnarray*}
As in the proof of Step~1, we may derive from this
that $z_k+\psi(\vec{\lambda}_k,
z_k)\to z_0+u_0$ and so $\psi(\vec{\lambda}_k, z_k)\to u_0$.
Obverse that (\ref{e:S.4.26}) implies
 $P^\bot\nabla\mathcal{L}_{\vec{\lambda}_k}(z_k+\psi(\vec{\lambda}_k, z_k))=0$, $k=1,2,\cdots$.
 The $C^1$-smoothness of  $\mathcal{L}$ and all $\mathcal{G}_j$ leads to
 $P^\bot\nabla\mathcal{L}_{\vec{\lambda}_0}(z_0+u_0)=0$.
  By Step~3 we arrive at
  $\psi(\vec{\lambda}_0, z_0)=u_0$ and hence $\psi$ is continuous at $(\vec{\lambda}_0, z_0)$.\\

\noindent{\bf Step 5}.\quad For any $(\vec{\lambda}, z_i)\in [-\delta, \delta]^n\times B_{H}(\theta,\epsilon)\cap H^0$, $i=1,2$,
by the definition of $\psi$, we have
\begin{eqnarray*}
 P^\bot\nabla\mathcal{L}(z_i+ \psi(\vec{\lambda}, z_i))+
 \sum^n_{j=1}\lambda_j P^\bot\nabla\mathcal{G}_j(z_i+ \psi(\vec{\lambda}, z_i))=\theta,\quad i=1,2,
 \end{eqnarray*}
and hence for $\Xi=z_1-z_2+ \psi(\vec{\lambda}, z_1)-\psi(\vec{\lambda}, z_2)=\Xi^0+\Xi^++\Xi^-$ we derive
\begin{eqnarray}\label{e:S.4.34}
0&=&(P^\bot\nabla\mathcal{L}(z_1+ \psi(\vec{\lambda}, z_1))-P^\bot\nabla\mathcal{L}(z_2+ \psi(\vec{\lambda}, z_2)),\Xi^+)_H
\nonumber\\
&+& \sum^n_{j=1}\lambda_j (P^\bot\nabla\mathcal{G}_j(z_1+ \psi(\vec{\lambda}, z_1))-
P^\bot\nabla\mathcal{G}_j(z_2+ \psi(\vec{\lambda}, z_2)),\Xi^+)_H.
 \end{eqnarray}
As in the proof of (\ref{e:S.4.29}) we obtain  $\tau\in (0,1)$ such that
 \begin{eqnarray}\label{e:S.4.35}
&&(P^\bot\nabla\mathcal{L}(z_1+ \psi(\vec{\lambda}, z_1))-P^\bot\nabla\mathcal{L}(z_2+ \psi(\vec{\lambda}, z_2)),\Xi^+)_H\nonumber\\
&=&(B(\tau z_1+ \tau\psi(\vec{\lambda}, z_1)+ (1-\tau)z_2+ (1-\tau)\psi(\vec{\lambda}, z_2))\Xi^+,\Xi^+)_H\nonumber\\
&&+(B(\tau z_1+ \tau\psi(\vec{\lambda}, z_1)+ (1-\tau)z_2+ (1-\tau)\psi(\vec{\lambda}, z_2))(\Xi^0+\Xi^-),\Xi^+)_H\nonumber\\
&\ge& a_1\|\Xi^+\|^2-\frac{a_1}{4}[\|\Xi^-+\Xi^0\|^2+\|\Xi^+\|^2]\nonumber\\
&=&\frac{3a_1}{4}\|\Xi^+\|^2-\frac{a_1}{4}\|\Xi^0\|^2-\frac{a_1}{4}\|\Xi^-\|^2.
 \end{eqnarray}
 Let us further shrink $\delta>0$ in Step~3 so that $\delta<\frac{\min\{a_0,a_1\}}{16nM}$.
As in (\ref{e:S.4.31}) we may deduce
\begin{eqnarray}\label{e:S.4.36}
&&\sum^n_{j=1}\lambda_j (P^\bot\nabla\mathcal{G}_j(z_1+ \psi(\vec{\lambda}, z_1))-
P^\bot\nabla\mathcal{G}_j(z_2+ \psi(\vec{\lambda}, z_2)),\Xi^+)_H\\
&\le&\sum^n_{j=1}|\lambda_j(\mathcal{G}''_j(\tau z_1+(1-\tau)z_2+ \tau\psi(\vec{\lambda}, z_1)+(1-\tau)\psi(\vec{\lambda}, z_2))\Xi, \Xi^+)_H|\nonumber\\
&\le& n\delta M\|\Xi\|\cdot\|\Xi^+\|\le 2n\delta M[\|\Xi\|^2+\|\Xi^+\|^2]\nonumber\\
&\le&\frac{a_1}{8}[\|\Xi^-\|^2+\|\Xi^0\|^2+2\|\Xi^+\|^2].
\end{eqnarray}
This, (\ref{e:S.4.34}) and (\ref{e:S.4.35}) lead to
$$
0\ge \frac{3a_1}{4}\|\Xi^+\|^2-\frac{a_1}{4}\|\Xi^0\|^2-
\frac{a_1}{4}\|\Xi^-\|^2-\frac{a_1}{8}[\|\Xi^-\|^2+\|\Xi^0\|^2+2\|\Xi^+\|^2]
$$
and so
\begin{eqnarray}\label{e:S.4.37}
0\ge 4\|\Xi^+\|^2-3\|\Xi^0\|^2-3\|\Xi^-\|^2.
\end{eqnarray}

Similarly, replacing $\Xi^+$ by $\Xi^-$ in (\ref{e:S.4.35}) and (\ref{e:S.4.36}) we derive
 \begin{eqnarray*}
&&(P^\bot\nabla\mathcal{L}(z_1+ \psi(\vec{\lambda}, z_1))-P^\bot\nabla\mathcal{L}(z_2+ \psi(\vec{\lambda}, z_2)),\Xi^-)_H\\
&\le& -\frac{3a_0}{4}\|\Xi^-\|^2+\frac{a_0}{4}\|\Xi^0\|^2+ \frac{a_0}{4}\|\Xi^+\|^2,\\
 &&\sum^n_{j=1}\lambda_j (P^\bot\nabla\mathcal{G}_j(z_1+ \psi(\vec{\lambda}, z_1))-
P^\bot\nabla\mathcal{G}_j(z_2+ \psi(\vec{\lambda}, z_2)),\Xi^-)_H\nonumber\\
&\le&\frac{a_0}{8}[\|\Xi^+\|^2+\|\Xi^0\|^2+2\|\Xi^-\|^2].
\end{eqnarray*}
As above these two inequalities and the equality
\begin{eqnarray*}
0&=&(P^\bot\nabla\mathcal{L}(z_1+ \psi(\vec{\lambda}, z_1))-P^\bot\nabla\mathcal{L}(z_2+ \psi(\vec{\lambda}, z_2)),\Xi^-)_H
\nonumber\\
&+& \sum^n_{j=1}\lambda_j (P^\bot\nabla\mathcal{G}_j(z_1+ \psi(\vec{\lambda}, z_1))-
P^\bot\nabla\mathcal{G}_j(z_2+ \psi(\vec{\lambda}, z_2)),\Xi^-)_H
 \end{eqnarray*}
yield: $0\ge 4\|\Xi^-\|^2-3\|\Xi^0\|^2-3\|\Xi^+\|^2$. Combing with (\ref{e:S.4.37}) we obtain
\begin{eqnarray*}
\|\Xi^++\Xi^-\|^2=\|\Xi^+\|^2+\|\Xi^-\|^2\le 6\|\Xi^0\|^2.
\end{eqnarray*}
Note that $\Xi^0=z_1-z_2$ and $\Xi^++\Xi^-=\psi(\vec{\lambda}, z_1)-\psi(\vec{\lambda}, z_2)$.
The desired claim is proved.\\

\noindent{\bf Step 6}.\quad The uniqueness of $\psi$ implies that it is equivariant with respect to $z$.
\end{proof}

As a by-product we have also the following result though it is not used in this paper.

\begin{theorem}[Inverse Function Theorem]\label{th:S.4.4}
If the assumptions of Theorem~\ref{th:S.1.1} hold with $X=H$,
then $\nabla\mathcal{L}$ is a homeomorphism near $\theta$.
\end{theorem}

\begin{proof}[\bf Proof]
We can assume that $\nabla\mathcal{L}$ is of class $(S)_+$ in
 $\overline{\mathcal{Q}_{r,s}}$.
Since $H^0=\{\theta\}$ and $\nabla\mathcal{L}=f_0$,
\begin{eqnarray}\label{e:S.4.38}
\deg(\nabla\mathcal{L}, \mathcal{Q}_{r,s}, \theta)=(-1)^{\mu}
\end{eqnarray}
by (\ref{e:S.4.18}).
Moreover, $\varrho:=\inf\{\|\nabla\mathcal{L}(u)\|\,|\, u\in\partial\overline{\mathcal{Q}_{r,s}}\}>0$
 by Lemma~\ref{lem:S.4.1}. For any given $v\in B_H(\theta,\varrho)$, let us define
 $\mathscr{H}:[0,1]\times\overline{\mathcal{Q}_{r,s}}\to H,\;(t,u)\mapsto \nabla\mathcal{L}(u)-tv$.
 Then
 $$\|\mathscr{H}(t,u)\|=\|\nabla\mathcal{L}(u)-tv\|\ge \|\nabla\mathcal{L}(u)\|-\|v\|\ge\varrho-\|v\|>0,\;
\forall (t,u)\in [0,1]\times\partial\overline{\mathcal{Q}_{r,s}}.
$$
Let  $t_n\to t$ in $[0,1]$, $(u_n)\subset\mathcal{Q}_{r,s}$ converge weakly to $u$ in $H$,
and $\lim\sup_{n\to\infty}(\mathscr{H}(t_n,u_n), u_n-u)_H\le 0$.
Then
$(\nabla\mathcal{L}(u_n),u_n-u)_H=(\mathscr{H}(t_n,u_n),u_n-u)_H+ t_n(v, u_n-u)_H$
leads to
$$
\lim\sup_{n\to\infty}(\nabla\mathcal{L}(u_n),u_n-u)_H\le 0.
$$
It follows that $u_n\to u$ in $H$
because $\nabla\mathcal{L}$ is of class $(S)_+$ in $\overline{\mathcal{Q}_{r,s}}$.
Hence  $\mathscr{H}$ is a homotopy of class $(S)_+$, and thus (\ref{e:S.4.38}) gives
$\deg(\nabla\mathcal{L}-v, \mathcal{Q}_{r,s}, \theta)
=\deg(\nabla\mathcal{L}, \mathcal{Q}_{r,s}, \theta)=(-1)^{\mu}$.
This implies $\nabla\mathcal{L}(\xi_v)=v$ for some $\xi_v\in \mathcal{Q}_{r,s}$.
By Step 3 in the proof of Theorem~\ref{th:S.4.3} (taking $\vec{\lambda}=\vec{0}$)
it is easily seen that the equation $\nabla\mathcal{L}(u)=v$ has a unique solution in $\overline{\mathcal{Q}_{r,s}}$,
and hence $\xi_v$ is unique. Then
we get a map $B_H(\theta,\varrho)\ni v \mapsto\xi_v\in \mathcal{Q}_{r,s}$ to satisfy
$\nabla\mathcal{L}(\xi_v)=v$ for all $v\in B_H(\theta,\varrho)$.
We claim that this map is continuous. Arguing by contradiction,
assume that there exists a sequence $v_n\to v$ in $B_H (\theta,\varrho)$, such that
$\xi_{v_n}\rightharpoonup \xi^\ast$ in $H$ and $\|\xi_{v_n}-\xi_v\|\ge\epsilon_0$
for some $\epsilon_0>0$ and all $n=1,2,\cdots$. Note that
\begin{eqnarray*}
(\nabla\mathcal{L}(\xi_{v_n}), \xi_{v_n}-\xi^\ast)_H&=&(v_n,\xi_{v_n}-\xi^\ast)_H
=(v_n-v,\xi_{v_n}-\xi^\ast)_H+(v,\xi_{v_n}-\xi^\ast)_H\to 0.
\end{eqnarray*}
We derive that $\xi_{v_n}\to\xi^\ast$ in $H$, and so
$\nabla\mathcal{L}(\xi_{v_n})=v_n$ can lead to $\nabla\mathcal{L}(\xi^\ast)=v$.
The uniqueness of solutions implies $\xi^\ast=\xi_v$.
This prove the claim. Hence $\nabla\mathcal{L}$
is a homeomorphism from an open neighborhood $\{\xi_v\,|\,v\in B_H(\theta,\varrho)\}$
of $\theta$ in $\mathcal{Q}_{r,s}$ onto $B_H(\theta,\varrho)$.
\end{proof}

Theorem~\ref{th:S.4.4} cannot be derived from the invariance of domain theorem
(5.4.1) of Berger \cite{Ber} or \cite[Theorem~2.5]{Fe}. Recently, Ekeland proved
an weaker inverse function theorem, \cite[Theorem~2]{Ek}.
Since we cannot insure that $B(u)$ has a right-inverse
$L(u)$ which is uniformly bounded in a neighborhood of $\theta$,
 Theorem~\ref{th:S.4.4} cannot be derived from \cite[Theorem~2]{Ek}
 either.

\subsection{Parameterized splitting and shifting theorems}\label{sec:S.5}

 To shorten the proof of the main theorem, we shall write parts of it into two propositions.

\begin{proposition}\label{prop:S.5.1}
Under the assumptions of Theorem~\ref{th:S.4.3},
 for each $(\vec{\lambda},z)\in [-\delta,\delta]^n\times B_{H^0}(\theta,\epsilon)$,
let $\psi_{\vec{\lambda}}(z)=\psi(\vec{\lambda}, z)$ be given by (\ref{e:S.4.21}).
Then it satisfies
 \begin{eqnarray*}
\mathcal{L}_{\vec{\lambda}}(z+\psi_\lambda(z))&=&\min\{\mathcal{L}_{\vec{\lambda}}(z+ u)\,|\, u\in B_H(\theta, r)\cap H^+\}
\quad\hbox{if}\;H^-=\{\theta\},\\
\mathcal{L}_{\vec{\lambda}}(z+\psi_{\vec{\lambda}}(z))&=&\min\{\mathcal{L}_{\vec{\lambda}}(z+ u+ P^-\psi_{\vec{\lambda}}(z))\,|\, u\in B_H(\theta, r)\cap H^+\}\\
&=&\max\{\mathcal{L}_{\vec{\lambda}}(z+ P^+\psi_{\vec{\lambda}}(z)+ v)\,|\, v\in B_H(\theta, s)\cap H^-\}
\quad\hbox{if}\;H^-\ne\{\theta\}.
\end{eqnarray*}
\end{proposition}

\begin{proof}[\bf Proof]\quad
{\bf Case} $H^-=\{\theta\}$.\quad
Then (\ref{e:S.4.22}) becomes
$P^+\nabla\mathcal{L}_{\vec{\lambda}}(z+\psi_{\vec{\lambda}}(z))=0\;\forall z\in B_{H^0}(\theta,\epsilon)$
since $\mathcal{Q}_{r,s}=B_H(\theta, r)\cap H^+$.
This and the integral mean value theorem give for each $u\in \mathcal{Q}_{r,s}$,
\begin{eqnarray*}
&&\mathcal{L}_{\vec{\lambda}}(z+ u)-\mathcal{L}_{\vec{\lambda}}(z+\psi_{\vec{\lambda}}(z))\\
&=&\int^1_0(\nabla\mathcal{L}_{\vec{\lambda}}(z+\psi_{\vec{\lambda}}(z)+ \tau(u-\psi_{\vec{\lambda}}(z))), u-\psi_{\vec{\lambda}}(z))_Hd\tau\\
&=&\int^1_0(P^+\nabla\mathcal{L}_{\vec{\lambda}}(z+\psi_{\vec{\lambda}}(z)+ \tau(u-\psi_{\vec{\lambda}}(z))), u-\psi_{\vec{\lambda}}(z))_Hd\tau\\
&=&\int^1_0(P^+\nabla\mathcal{L}_{\vec{\lambda}}(z+\psi_{\vec{\lambda}}(z)+ \tau(u-\psi_{\vec{\lambda}}(z)))-P^+\nabla\mathcal{L}_{\vec{\lambda}}(z+\psi_{\vec{\lambda}}(z)), u-\psi_{\vec{\lambda}}(z))_Hd\tau\\
&=&\int^1_0(\nabla\mathcal{L}_{\vec{\lambda}}(z+\psi_{\vec{\lambda}}(z)+ \tau(u-\psi_{\vec{\lambda}}(z)))-\nabla\mathcal{L}_{\vec{\lambda}}(z+\psi_{\vec{\lambda}}(z)), u-\psi_{\vec{\lambda}}(z))_Hd\tau\\
&=&\int^1_0\tau d\tau\int^1_0\big(B(z+\psi_{\vec{\lambda}}(z)+ \rho\tau(u-\psi_{\vec{\lambda}}(z)))(u-\psi_{\vec{\lambda}}(z)), u-\psi_{\vec{\lambda}}(z)\bigr)_Hd\rho\\
&&+\sum^n_{j=1}\lambda_j\int^1_0\tau d\tau\int^1_0\big(\mathcal{G}''_j(z+\psi_{\vec{\lambda}}(z)+ \rho\tau(u-\psi_{\vec{\lambda}}(z)))(u-\psi_{\vec{\lambda}}(z)), u-\psi_{\vec{\lambda}}(z)\bigr)_Hd\rho\\
&\ge&\frac{a_1}{2}\|u-\psi_{\vec{\lambda}}(z)\|^2\\
&&+\sum^n_{j=1}\lambda_j\int^1_0\tau d\tau\int^1_0\big(\mathcal{G}''_j(z+\psi_{\vec{\lambda}}(z)+ \rho\tau(u-\psi_{\vec{\lambda}}(z)))(u-\psi_{\vec{\lambda}}(z)), u-\psi_{\vec{\lambda}}(z)\bigr)_Hd\rho,
\end{eqnarray*}
where the final inequality comes from Lemma~\ref{lem:S.2.2}(i).
For the final sum, as in (\ref{e:S.4.31}) we have
 \begin{eqnarray*}
&&\left|\sum^n_{j=1}\lambda_j\int^1_0\tau d\tau\int^1_0\big(\mathcal{G}''_j(z+\psi_{\vec{\lambda}}(z)+ \rho\tau(u-\psi_{\vec{\lambda}}(z)))(u-\psi_{\vec{\lambda}}(z)), u-\psi_{\vec{\lambda}}(z)\bigr)_Hd\rho\right|\\
&\le& 2n\delta M\|u-\psi_{\vec{\lambda}}(z))\|^2\le \frac{a_1}{4}\|u-\psi_{\vec{\lambda}}(z))\|^2.
\end{eqnarray*}
These lead to
\begin{eqnarray}\label{e:S.5.1}
\mathcal{L}_{\vec{\lambda}}(z+ u)-\mathcal{L}_{\vec{\lambda}}(z+\psi_{\vec{\lambda}}(z))
\ge\frac{a_1}{4}\|u-\psi_{\vec{\lambda}}(z)\|^2,
\end{eqnarray}
which implies the desired conclusion.

 {\bf Case} $H^-\ne\{\theta\}$.\quad
  For each $u\in B_H(\theta, r)\cap H^+$ we have $u+P^-\psi_{\vec{\lambda}}(z)\in \mathcal{Q}_{r,s}$.
  As above we can use (\ref{e:S.4.22}) to derive
  \begin{eqnarray}\label{e:S.5.2}
\mathcal{L}_{\vec{\lambda}}(z+ u+ P^-\psi_{\vec{\lambda}}(z))-\mathcal{L}_{\vec{\lambda}}(z+ \psi_{\vec{\lambda}}(z))
\ge\frac{a_1}{4}\|u-P^+\psi_{\vec{\lambda}}(z)\|^2,
\end{eqnarray}
and therefore  the second equality. Similarly,  for
each $v\in B_H(\theta, r)\cap H^-$ we have $P^+\psi_{\vec{\lambda}}(z)+v\in \mathcal{Q}_{r,s}$,
and use (\ref{e:S.4.22}) and  Lemma~\ref{lem:S.2.2}(ii)-(iii) to deduce
\begin{eqnarray*}
&&\mathcal{L}_{\vec{\lambda}}(z+ P^+\psi_{\vec{\lambda}}(z)+v)-\mathcal{L}_{\vec{\lambda}}(z+\psi_{\vec{\lambda}}(z))\\
&=&\int^1_0(\nabla\mathcal{L}_{\vec{\lambda}}(z+ \psi_{\vec{\lambda}}(z)+ t(v-P^+\psi_{\vec{\lambda}}(z))), v-P^-\psi_{\vec{\lambda}}(z))_Hdt\\
&=&\int^1_0(\nabla\mathcal{L}_{\vec{\lambda}}(z+ \psi_{\vec{\lambda}}(z)+ t(v-P^-\psi_{\vec{\lambda}}(z)))- \nabla\mathcal{L}_{\vec{\lambda}}(z+ \psi_{\vec{\lambda}}(z)), v-P^-\psi_{\vec{\lambda}}(z))_Hdt\\
&=&\int^1_0t\int^1_0(B(z+ \psi_{\vec{\lambda}}(z)+ \tau t(v-P^-\psi_{\vec{\lambda}}(z)))(v-P^-\psi_{\vec{\lambda}}(z)), v-P^-\psi_{\vec{\lambda}}(z))_Hdtd\tau\\
&+&\sum^n_{j=1}\lambda_j\int^1_0t dt\int^1_0\big(\mathcal{G}''_j(z+ \psi_{\vec{\lambda}}(z)+ \tau t(v-P^-\psi_{\vec{\lambda}}(z)))(v-P^-\psi_{\vec{\lambda}}(z)), v-P^-\psi_{\vec{\lambda}}(z)\bigr)_Hd\tau\\
&\le&-\frac{a_0}{2}\|v-P^-\psi_{\vec{\lambda}}(z)\|^2\\
&+&\sum^n_{j=1}\lambda_j\int^1_0t dt\int^1_0\big(\mathcal{G}''_j(z+ \psi_{\vec{\lambda}}(z)+ \tau t(v-P^-\psi_{\vec{\lambda}}(z)))(v-P^-\psi_{\vec{\lambda}}(z)), v-P^-\psi_{\vec{\lambda}}(z)\bigr)_Hd\tau\\
&\le&-\frac{a_0}{4}\|v-P^-\psi_{\vec{\lambda}}(z)\|^2,
\end{eqnarray*}
 and hence  the third equality.
\end{proof}

\begin{proposition}\label{prop:S.5.2}
Under the assumptions of Theorem~\ref{th:S.4.3},
 for each $(\vec{\lambda},z)\in [-\delta,\delta]^n\times B_{H^0}(\theta,\epsilon)$,
let $\psi_{\vec{\lambda}}(z)=\psi(\vec{\lambda}, z)$ be given by (\ref{e:S.4.21}).
 Then
 \begin{equation}\label{e:S.5.3}
 \mathcal{L}^\circ_{\vec{\lambda}}(z):=\mathcal{L}_{\vec{\lambda}}(z+\psi_{\vec{\lambda}}(z))
 =\mathcal{L}(z+\psi(\vec{\lambda}, z))+ \sum^n_{j=1}\lambda_j\mathcal{G}_j(z+\psi(\vec{\lambda}, z))
  \end{equation}
  defines a $C^1$ functional on $B_H(\theta, \epsilon)\cap H^0$, and its differential is given by
\begin{eqnarray}\label{e:S.5.4}
D\mathcal{L}_{\vec{\lambda}}^\circ(z)[h]=D\mathcal{L}(z+\psi(\vec{\lambda},z))[h]+
\sum^n_{j=1}\lambda_j D\mathcal{G}_j(z+\psi(\vec{\lambda},z))[h],\quad\forall h\in H^0.
\end{eqnarray}
(Clearly, this implies that $[-\delta,\delta]^n\ni\vec{\lambda}\mapsto\mathcal{L}^\circ_{\vec{\lambda}}\in C^1(\bar{B}_H(\theta, \epsilon)\cap H^0)$
is continuous by shrinking $\epsilon>0$ since $\dim H^0<\infty$).
\end{proposition}

 \begin{proof}[\bf Proof]
  {\bf Case} $H^-\ne\{\theta\}$.\quad
  For fixed $z\in B_H(\theta, \epsilon)\cap H^0$, $h\in H^0$, and
$t\in\mathbb{R}$ with sufficiently small $|t|$, the last two equalities in
Proposition~\ref{prop:S.5.1} imply
 \begin{eqnarray}\label{e:S.5.5}
&&\mathcal{L}_{\vec{\lambda}}(z+th+ P^+\psi_{\vec{\lambda}}(z+th)+P^-\psi_{\vec{\lambda}}(z))-\mathcal{L}_{\vec{\lambda}}(z+P^+\psi_{\vec{\lambda}}(z+th)+ P^-\psi_{\vec{\lambda}}(z))\nonumber\\
&\le& \mathcal{L}_{\vec{\lambda}}(z+th+\psi_{\vec{\lambda}}(z+th))-\mathcal{L}_{\vec{\lambda}}(z+\psi_{\vec{\lambda}}(z))\nonumber\\
&\le&\mathcal{L}_{\vec{\lambda}}(z+th+ P^+\psi_{\vec{\lambda}}(z)+ P^-\psi_{\vec{\lambda}}(z+th))-
\mathcal{L}_{\vec{\lambda}}(z+ P^+\psi_{\vec{\lambda}}(z)+ P^-\psi_{\vec{\lambda}}(z+th)).\nonumber\\
\end{eqnarray}
Since $\mathcal{L}_{\vec{\lambda}}$ is $C^1$  and  $\psi_{\vec{\lambda}}$ is continuous we deduce,
\begin{eqnarray}\label{e:S.5.6}
&&\lim_{t\to 0}\frac{\mathcal{L}_{\vec{\lambda}}(z+th+P^+\psi_{\vec{\lambda}}(z+th)+P^-\psi_{\vec{\lambda}}(z))-
\mathcal{L}_{\vec{\lambda}}(z+
P^+\psi_{\vec{\lambda}}(z+th)+P^-\psi_{\vec{\lambda}}(z))}{t}\nonumber\\
&=& \lim_{t\to 0}\int^1_0D\mathcal{L}_{\vec{\lambda}}(z+sth+P^+\psi_{\vec{\lambda}}(z+th)+P^-\psi_{\vec{\lambda}}(z))[h] ds\nonumber\\
&=&D\mathcal{L}_{\vec{\lambda}}(z+\psi_{\vec{\lambda}}(z))[h].
\end{eqnarray}
Here the last equality follows from the Lebesgue's Dominated Convergence Theorem
since
$$
\{D\mathcal{L}_{\vec{\lambda}}(z+sth+P^+\psi_{\vec{\lambda}}(z+th)+P^-\psi_{\vec{\lambda}}(z))[h]\,|\, 0\le s\le 1,\;|t|\le 1\}
$$
is bounded by the compactness of
 $\{z+ sth+P^+\psi_{\vec{\lambda}}(z+th)+P^-\psi_{\vec{\lambda}}(z)\,|\, 0\le s\le 1,\; |t|\le 1\}$.

Similarly, we have
\begin{eqnarray}\label{e:S.5.7}
&&\lim_{t\to 0}\frac{\mathcal{L}_{\vec{\lambda}}(z+th+ P^+\psi_{\vec{\lambda}}(z)+P^-\psi_{\vec{\lambda}}(z+th))-\mathcal{L}_{\vec{\lambda}}(z+P^+\psi_{\vec{\lambda}}(z)+
P^-\psi_{\vec{\lambda}}(z+th))}{t}\nonumber\\
&=&D\mathcal{L}_{\vec{\lambda}}(z+\psi_{\vec{\lambda}}(z))[h].
\end{eqnarray}
Using the Sandwich Theorem we conclude from (\ref{e:S.5.5}), (\ref{e:S.5.6}) and (\ref{e:S.5.7}) that
 \begin{eqnarray*}
\lim_{t\to 0}\frac{\mathcal{L}_{\vec{\lambda}}(z+th+\psi_{\vec{\lambda}}(z+th))-\mathcal{L}_{\vec{\lambda}}(z+\psi_{\vec{\lambda}}(z))}{t}=
D\mathcal{L}_{\vec{\lambda}}(z+\psi_{\vec{\lambda}}(z))[h],\quad\forall h\in H^0.
\end{eqnarray*}
 That is, $\mathcal{L}_{\vec{\lambda}}^\circ$ is G\^ateaux differentiable  and $D\mathcal{L}_{\vec{\lambda}}^\circ(z)=D\mathcal{L}_{\vec{\lambda}}(z+\psi_{\vec{\lambda}}(z))|_{H^0}$.
 The latter implies that  $\mathcal{L}_{\vec{\lambda}}^\circ$ is of class $C^1$
 because both $D\mathcal{L}_{\vec{\lambda}}$ and $\psi_{\vec{\lambda}}$ are continuous.\\

\noindent{\bf Case} $H^-=\{\theta\}$.\quad
For fixed $z\in B_H(\theta, \epsilon)\cap H^0$ and $h\in H^0$, and
$t\in\mathbb{R}$ with sufficiently small $|t|$, the first equality in Proposition~\ref{prop:S.5.1} implies
\begin{eqnarray}\label{e:S.5.8}
&&\mathcal{L}_{\vec{\lambda}}(z+th+\psi_{\vec{\lambda}}(z+th))-\mathcal{L}_{\vec{\lambda}}(z+\psi_{\vec{\lambda}}(z+th))\nonumber\\
&\le& \mathcal{L}_{\vec{\lambda}}(z+th+\psi_{\vec{\lambda}}(z+th))-\mathcal{L}_{\vec{\lambda}}(z+\psi_{\vec{\lambda}}(z))\nonumber\\
&\le&\mathcal{L}_{\vec{\lambda}}(z+th+\psi_{\vec{\lambda}}(z))-\mathcal{L}_{\vec{\lambda}}(z+\psi_{\vec{\lambda}}(z)).
\end{eqnarray}
By  the continuity of  $\nabla\mathcal{L}_{\vec{\lambda}}$ and
$\psi_{\vec{\lambda}}$ we obtain
\begin{eqnarray}\label{e:S.5.9}
&&\lim_{t\to 0}\frac{\mathcal{L}_{\vec{\lambda}}(z+th+\psi_{\vec{\lambda}}(z+th))-
\mathcal{L}_{\vec{\lambda}}(z+\psi_{\vec{\lambda}}(z+th))}{t}\nonumber\\
&=& \lim_{t\to 0}\int^1_0D\mathcal{L}_{\vec{\lambda}}(z+sth+\psi_{\vec{\lambda}}(z+th))[h]ds\nonumber\\
&=&D\mathcal{L}_{\vec{\lambda}}(z+\psi_{\vec{\lambda}}(z))[h].
\end{eqnarray}
(As above this follows from the Lebesgue's Dominated Convergence Theorem because
$\{z+ sth+\psi_{\vec{\lambda}}(z+th)\,|\, 0\le s\le 1,\; 0\le t\le 1\}$
is compact
and thus $\{D\mathcal{L}_{\vec{\lambda}}(z+sth+\psi_{\vec{\lambda}}(z+th))[h]\,|\, 0\le s\le 1,\;|t|\le 1\}$
is bounded). Similarly, we may prove
\begin{eqnarray}\label{e:S.5.10}
\lim_{t\to 0}\frac{\mathcal{L}_{\vec{\lambda}}(z+th+\psi_{\vec{\lambda}}(z))-
\mathcal{L}_{\vec{\lambda}}(z+\psi_{\vec{\lambda}}(z))}{t}=
D\mathcal{L}_{\vec{\lambda}}(z+\psi_{\vec{\lambda}}(z))[h],
\end{eqnarray}
and thus
\begin{eqnarray*}
\lim_{t\to 0}\frac{\mathcal{L}_{\vec{\lambda}}(z+th+\psi_{\vec{\lambda}}(z+th))-
\mathcal{L}_{\vec{\lambda}}(z+\psi_{\vec{\lambda}}(z))}{t}=
D\mathcal{L}_{\vec{\lambda}}(z+\psi_{\vec{\lambda}}(z))[h]
\end{eqnarray*}
by (\ref{e:S.5.8}), (\ref{e:S.5.9}) and (\ref{e:S.5.10}).
The desired claim follows immediately.
 \end{proof}

\begin{theorem}[Parameterized Splitting Theorem]\label{th:S.5.3}
Under the assumptions of Theorem~\ref{th:S.4.3},
by shrinking $\delta>0$, $\epsilon>0$ and $r>0, s>0$, we obtain
an open neighborhood $W$ of $\theta$ in $H$ and an origin-preserving
homeomorphism
\begin{eqnarray}\label{e:S.5.11}
&&[-\delta, \delta]^n\times B_{H^0}(\theta,\epsilon)\times
\left(B_{H^+}(\theta, r) +
B_{H^-}(\theta, s)\right)\to [-\delta, \delta]^n\times W,\nonumber\\
&&(\vec{\lambda}, z, u^++u^-)\mapsto (\vec{\lambda},\Phi_{\vec{\lambda}}(z, u^++u^-))
\end{eqnarray}
such that
\begin{equation}\label{e:S.5.12}
\mathcal{ L}_{\vec{\lambda}}\circ\Phi_{\vec{\lambda}}(z, u^++ u^-)=\|u^+\|^2-\|u^-\|^2+ \mathcal{
L}_{\vec{\lambda}}(z+ \psi(\vec{\lambda}, z))
\end{equation}
for all $(\vec{\lambda}, z, u^+ + u^-)\in [-\delta,\delta]^n\times B_{H^0}(\theta,\epsilon)\times
\left(B_{H^+}(\theta, r) +
B_{H^-}(\theta, s)\right)$, where $\psi$ is given by (\ref{e:S.4.21}).
 The functional $\mathcal{L}_{\vec{\lambda}}^\circ: B_H(\theta, \epsilon)\cap H^0\to \mathbb{R}$
   given by (\ref{e:S.5.3}) is of class $C^1$, and its differential is given by (\ref{e:S.5.4}).
Moreover, (i) if $\mathcal{L}$ and $\mathcal{G}_j$, $j=1,\cdots,n$, are of class $C^{2-0}$,
then so is $\mathcal{L}_{\vec{\lambda}}^\circ$ for each $\vec{\lambda}\in [-\delta, \delta]^n$;
(ii) if a compact Lie group $G$  acts on $H$ orthogonally, and
$V$, $\mathcal{L}$ and $\mathcal{G}$ are $G$-invariant (and hence $H^0$, $(H^0)^\bot$
are $G$-invariant subspaces), then for each $\vec{\lambda}\in [-\delta, \delta]^n$, $\psi(\vec{\lambda}, \cdot)$
 and $\Phi_{\vec{\lambda}}(\cdot,\cdot)$  are $G$-equivariant, and $\mathcal{L}^\circ_{\vec{\lambda}}(z)=\mathcal{
L}_{\vec{\lambda}}(z+ \psi(\vec{\lambda}, z))$ is $G$-invariant.
\end{theorem}

If the corresponding conditions with \cite[Theorem~1.1]{Lu1} or \cite[Remark~3.2]{Lu2} are also satisfied,
  we can prove: $\psi(\vec{\lambda}, \cdot)$ is of class $C^1$,
  $\mathcal{L}^\circ_{\vec{\lambda}}$ is of class $C^2$, and
   \begin{eqnarray}
   &&D\mathcal{L}^\circ_{\vec{\lambda}}(z)[u]=(\nabla\mathcal{L}_{\vec{\lambda}}(z+ \psi(\vec{\lambda}, z)), u)_H,\label{e:S.5.12.1}\\
   &&d^2\mathcal{L}^\circ_{\vec{\lambda}}(z)[u,v]=\bigl(\mathcal{L}''_{\vec{\lambda}}(z+ \psi(\vec{\lambda}, z))
   (u+ D_z\psi(\vec{\lambda}, z)[u]),v\bigr)_H\label{e:S.5.12.2}
  \end{eqnarray}
  for all $z\in B_H(\theta, \epsilon)\cap H^0$ and $u,v\in H^0$. Note that $\psi(\vec{\lambda}, \theta)=\theta$.
  We have
 \begin{eqnarray}\label{e:S.5.12.3}
  d^2\mathcal{L}^\circ_{\vec{\lambda}}(\theta)[z_1,z_2]&=&(\mathcal{L}''_{\vec{\lambda}}(\theta)(z_1+D_z\psi(\vec{\lambda}, \theta)[z_1]),z_2)_H\nonumber\\
  &=&-\sum^n_{j=1}\lambda_j(\mathcal{G}''_j(\theta)(z_1+D_z\psi(\vec{\lambda}, \theta)[z_1]),z_2)_H,\quad\forall z_1, z_2\in H^0,
  \end{eqnarray}
and $d^2\mathcal{L}^\circ_{\vec{0}}(\theta)=0$
  by $D_z\psi(\vec{0}, \theta)=\theta$. Moreover, if $\mathcal{L}''(\theta)\mathcal{G}''_j(\theta)=\mathcal{G}''_j(\theta)\mathcal{L}''(\theta)$
 for $j=1,\cdots,n$, then
\begin{eqnarray}\label{e:S.5.12.4}
  d^2\mathcal{L}^\circ_{\vec{\lambda}}(\theta)[z_1,z_2]
 =-\sum^n_{j=1}\lambda_j(\mathcal{G}''_j(\theta)z_1, z_2)_H,\quad\forall z_1, z_2\in H^0.
  \end{eqnarray}

  \begin{claim}\label{cl:S.5.3*}
 In this situation, if $\theta\in H$ is a nondegenerate critical point of $\mathcal{L}_{\vec{\lambda}}$
 then $\theta\in H^0$ is such a critical point of $\mathcal{L}^\circ_{\vec{\lambda}}$ too.
\end{claim}

In fact, suppose that  $z_1\in H^0$ satisfies $d^2\mathcal{L}^\circ_{\vec{\lambda}}(\theta)[z_1,z_2]=0\;
\forall z_2\in H^0$. Then (\ref{e:S.5.12.3}) implies
$$
(P^0\mathcal{L}''_{\vec{\lambda}}(\theta)(z_1+ D_z\psi(\vec{\lambda}, \theta)[z_1]),u)_H=(P^0\mathcal{L}''_{\vec{\lambda}}(\theta)(z_1+ D_z\psi(\vec{\lambda},\theta)[z_1]), P^0u)_H
=0\quad\forall u\in H.
$$
Hence $P^0\mathcal{L}''_{\vec{\lambda}}(\theta)(z_1+ D_z\psi(\vec{\lambda}, \theta)[z_1])=\theta$.
Moreover, since $(I-P^0)\nabla\mathcal{L}_{\vec{\lambda}}(z+ \psi(\vec{\lambda}, z))=\theta$
for all $z\in B_H(\theta, \epsilon)\cap H^0$. Differentiating this equality  at $z=\theta$ we get
$(I-P^0)\mathcal{L}''_{\vec{\lambda}}(\theta)(z_1+ D_z\psi(\vec{\lambda}, \theta)[z_1])=\theta$
for all $z_1\in H^0$. It follows that $\mathcal{L}''_{\vec{\lambda}}(\theta)(z_1+ D_z\psi(\vec{\lambda}, \theta)[z_1])=\theta$
and hence $z_1+ D_z\psi(\vec{\lambda}, \theta)[z_1]=\theta$. Note that $z_1\in H^0$ and
$D_z\psi(\vec{\lambda}, \theta)[z_1]\in (H^0)^\bot$. We arrive at $z_1=\theta$.\\

\begin{proof}[\bf Proof of Theorem~\ref{th:S.5.3}] \quad
Let $N=H^0$, and for each $\vec{\lambda}\in [-\delta,\delta]^n$
 we define a  map
 \begin{equation}\label{e:S.5.13}
 F_{\vec{\lambda}}:B_{N}(\theta, \epsilon)\times \mathcal{Q}_{r,s}\to\R,\;(z, u)\mapsto\mathcal{L}_{\vec{\lambda}}(z+ \psi(\vec{\lambda},z)+ u)-\mathcal{L}_{\vec{\lambda}}(z+ \psi(\vec{\lambda},z)).
\end{equation}
 Then $D_2F_{\vec{\lambda}}(z,u)[v]=(P^\bot\nabla{\mathcal L}_\lambda(z+ \psi(\vec{\lambda},z)+ u), v)_H$
for  $z\in \bar B_{N}(\theta, \epsilon)$,  $u\in \mathcal{Q}_{r,s}$ and $v\in N^\bot$.
 Moreover it holds that
\begin{eqnarray}\label{e:S.5.14}
F_{\vec{\lambda}}(z, \theta)=0\quad\hbox{and}\quad D_2F_{\vec{\lambda}}(z,
\theta)[v]=0\quad\;\forall v\in N^\bot.
\end{eqnarray}
Since  $B_{N}(\theta,\epsilon)\oplus\mathcal{Q}_{r,s}$ has the closure
contained in the neighborhood $U$ in Lemma~\ref{lem:S.2.2}, and
 $\psi(\vec{\lambda},\theta)=\theta$, we can shrink $\nu>0$,
$\epsilon>0$, $r>0$ and $s>0$ so small that
 \begin{equation}\label{e:S.5.15}
 z+ \psi(\vec{\lambda},z)+
u^++ u^-\in U,\quad\forall (\vec{\lambda}, z, u^+ + u^-)\in [-\delta,\delta]^n\times\bar B_{N}(\theta,\epsilon)\times \overline{\mathcal{Q}_{r,s}}.
\end{equation}

Let us verify that each $F_{\vec{\lambda}}$ satisfies conditions (ii)-(iv) in \cite[Theorem~A.1]{Lu2}.

\noindent{\bf Step 1}. For $\vec{\lambda}\in [-\delta,\delta]^n$, $z\in\bar B_{N}(\theta,\epsilon)$, $u^+\in
\bar B_{H^+}(\theta, r)$ and $u^-_1, u^-_2\in\bar
B_{H^-}(\theta,\epsilon)$,  we have
\begin{eqnarray}\label{e:S.5.16}
&&D_2F_{\vec{\lambda}}(z, u^+ + u^-_2)[u^-_2-u^-_1]-D_2F_{\vec{\lambda}}(z, u^++ u^-_1)[u^-_2-u^-_1]\nonumber\\
&=&(\nabla\mathcal{L}_{\vec{\lambda}}(z+ \psi_{\vec{\lambda}}(z)+ u^++u^-_2), u^-_2-u^-_1)_H\nonumber\\
&&- (\nabla\mathcal{L}_{\vec{\lambda}}(z+ \psi_{\vec{\lambda}}(z)+ u^++u^-_1), u^-_2-u^-_1)_H.
\end{eqnarray}
Since  the function
$u\mapsto (\nabla\mathcal{L}_{\vec{\lambda}}(z+ \psi_{\vec{\lambda}}(z)+ u^++u), u^-_2-u^-_1)_H$
is G\^ateaux differentiable,  the mean value theorem yields $t\in (0, 1)$ such that
\begin{eqnarray}\label{e:S.5.17}
&&(\nabla\mathcal{L}_{\vec{\lambda}}(z+ \psi_{\vec{\lambda}}(z)+ u^++u^-_2), u^-_2-u^-_1)_H
- (\nabla\mathcal{L}_{\vec{\lambda}}(z+ \psi_{\vec{\lambda}}(z)+ u^++u^-_1), u^-_2-u^-_1)_H\nonumber\\
&=&\left(B(z+ \psi_{\vec{\lambda}}(z)+ u^++ u^-_1+ t(u^-_2-u^-_1))(u^-_2-u^-_1),
u^-_2-u^-_1\right)_H\nonumber\\
&&+\sum^n_{j=1}\lambda_j\left(\mathcal{G}''_j(z+ \psi_{\vec{\lambda}}(z)+ u^++ u^-_1+ t(u^-_2-u^-_1))(u^-_2-u^-_1),
u^-_2-u^-_1\right)_H\nonumber\\
&\le& \sum^n_{j=1}\lambda_j\left(\mathcal{G}''_j(z+ \psi_{\vec{\lambda}}(z)+ u^++ u^-_1+ t(u^-_2-u^-_1))(u^-_2-u^-_1),
u^-_2-u^-_1\right)_H\nonumber\\
&&\qquad -a_0\|u^-_2-u^-_1\|^2
\end{eqnarray}
because of Lemma~\ref{lem:S.2.2}(iii).   Recall that we have assumed $\delta<\frac{\min\{a_0,a_1\}}{8nM}$
in Step~3 of the proof of Theorem~\ref{th:S.4.3}. From this and
 (\ref{e:S.4.30.1}) it follows that
\begin{eqnarray*}
&&\sum^n_{j=1}\Big|\lambda_j\left(\mathcal{G}''_j(z+ \psi_{\vec{\lambda}}(z)+ u^++ u^-_1+ t(u^-_2-u^-_1))(u^-_2-u^-_1),
u^-_2-u^-_1\right)_H\Big|\\
&&\le n\delta M\|u^-_2-u^-_1\|^2\le\frac{a_0}{8}\|u^-_2-u^-_1\|^2.
\end{eqnarray*}
This, (\ref{e:S.5.16}) and (\ref{e:S.5.17}) lead to
\begin{eqnarray*}
 (D_2F_{\vec{\lambda}}(z, u^+ + u^-_2)-D_2F_{\vec{\lambda}}(z, u^++ u^-_1))[u^-_2-u^-_1]\le
-\frac{a_0}{2}\|u^-_2-u^-_1\|^2.
\end{eqnarray*}
 This implies the condition (ii) of \cite[theorem~A.1]{Lu2}.\\

\noindent{\bf Step 2}. For $\vec{\lambda}\in [-\delta,\delta]^n$, $z\in\bar B_{N}(\theta,\epsilon)$, $u^+\in\bar B_{H^+}(\theta, r)$ and $u^-\in\bar
B_{H^-}(\theta, s)$, by (\ref{e:S.5.14}) and the mean value
theorem, for some $t\in (0, 1)$ we have
\begin{eqnarray}\label{e:S.5.18}
&&D_2F_{\vec{\lambda}}(z, u^++u^-)[u^+-u^-]\nonumber\\
&=&D_2F_{\vec{\lambda}}(z, u^++u^-)[u^+-u^-]- D_2F_{\vec{\lambda}}(z, \theta)[u^+-u^-]\nonumber\\
&=&(\nabla\mathcal{L}_{\vec{\lambda}}(z+ \psi_{\vec{\lambda}}(z)+ u^++u^-), u^+-u^-)_H-(\nabla\mathcal{L}_{\vec{\lambda}}(z+ \psi_{\vec{\lambda}}(z)), u^+-u^-)_H\nonumber\\
&=&\left(B(z+ \psi_{\vec{\lambda}}(z)+ t(u^++u^-))(u^++u^-), u^+-u^-\right)_H\nonumber\\
&+&\sum^n_{j=1}\lambda_j\left(\mathcal{G}''_j(z+ \psi_{\vec{\lambda}}(z)+ t(u^++u^-))(u^++u^-), u^+-u^-\right)_H\nonumber\\
&=&\left(B(z+ \psi_{\vec{\lambda}}(z)+ t(u^++u^-))u^+, u^+\right)_H
-\left(B(z+ \psi_{\vec{\lambda}}(z)+
t(u^++u^-))u^-, u^-\right)_H\nonumber\\
&+&\sum^n_{j=1}\lambda_j\left(\mathcal{G}''_j(z+ \psi_{\vec{\lambda}}(z)+ t(u^++u^-))(u^++u^-), u^+-u^-\right)_H\nonumber\\
&\ge & a_1\|u^+\|^2+ a_0\|u^-\|^2
+\sum^n_{j=1}\lambda_j\left(\mathcal{G}''_j(z+ \psi_{\vec{\lambda}}(z)+ t(u^++u^-))(u^++u^-), u^+-u^-\right)_H\nonumber\\
\end{eqnarray}
because of Lemma~\ref{lem:S.2.2}(i) and (iii).  As above we have
\begin{eqnarray*}
&&\sum^n_{j=1}\Big|\lambda_j\left(\mathcal{G}''_j(z+ \psi_{\vec{\lambda}}(z)+ t(u^++u^-))(u^++u^-), u^+-u^-\right)_H\Big|\\
&&\le n\delta M\|u^++u^-\|\cdot\|u^+-u^-\|\\
&&\le\frac{\min\{a_0,a_1\}}{4}(\|u^+\|^2+\|u^-\|^2)\\
&&\le\frac{a_1}{4}\|u^+\|^2+ \frac{a_0}{4}\|u^-\|^2.
\end{eqnarray*}
This and (\ref{e:S.5.18}) give
\begin{equation}\label{e:S.5.19}
D_2F_{\vec{\lambda}}(z, u^++u^-)[u^+-u^-]\ge  \frac{a_1}{2}\|u^+\|^2+ \frac{a_0}{2}\|u^-\|^2.
\end{equation}
Thus the condition (iii) of \cite[Theorem~A.1]{Lu2} is satisfied.
In particular, (\ref{e:S.5.19}) also implies
 $$
D_2F_{\vec{\lambda}}(z, u^+)[u^+] \ge \frac{a_1}{2}\|u^+\|^2> p(\|u^+\|),\quad\forall u^+\in \bar
B_{H^+}(\theta,s)\setminus\{\theta\},
$$
where $p:(0, \varepsilon]\to (0, \infty)$ is a non-decreasing
function given by $p(t)=\frac{a_1}{4}t^2$. Namely, $F_{\vec{\lambda}}$ satisfies the condition
(iv) of \cite[Theorem~A.1]{Lu2} (the parameterized version of \cite[Theoren~1.1]{DHK}).

The other arguments are as before.\\

\noindent{\bf Step 3}. The claim (i) in the part of ``Moreover" follows from
(\ref{e:S.4.22.1}) directly. For the second one,
since $\psi(\lambda, \cdot)$ is $G$-equivariant,
and $\mathcal{L}_\lambda$ is $G$-invariant, we derive from (\ref{e:S.5.13}) that
$F_{\vec{\lambda}}$ is $G$-invariant. By the construction of
$\Phi_{\vec{\lambda}}(\cdot,\cdot)$ (cf. \cite{DHK} and \cite[Theorem~A.1]{Lu1}),
it is expressed by $F_{\vec{\lambda}}(z, \cdot)$, one easily sees that
$\Phi_{\vec{\lambda}}(\cdot,\cdot)$ is $G$-equivariant.
\end{proof}

\begin{theorem}[Parameterized Shifting Theorem]\label{th:S.5.4}
Suppose for some $\vec{\lambda}\in [-\delta, \delta]^n$ that $\theta\in H$ is an isolated critical point of $\mathcal{L}_{\vec{\lambda}}$ (thus $\theta\in H^0$ is that of $\mathcal{L}_{\vec{\lambda}}^\circ$). Then
\begin{equation}\label{e:S.5.20}
C_q(\mathcal{L}_{\vec{\lambda}},\theta;{\bf K})=C_{q-\mu}(\mathcal{L}^\circ_{\vec{\lambda}},\theta;{\bf K})\quad\forall
q\in\mathbb{N}\cup\{0\},
\end{equation}
where $\mathcal{L}^\circ_{\vec{\lambda}}(z)=\mathcal{
L}_{\vec{\lambda}}(z+ \psi(\vec{\lambda}, z))= \mathcal{L}(z+\psi(\vec{\lambda}, z))+
\sum^n_{j=1}\lambda_j\mathcal{G}_j(z+\psi(\vec{\lambda}, z))$ is
 as in (\ref{e:S.5.3}).
\end{theorem}

\begin{proof}[\bf Proof]
Though $\mathcal{L}_{\vec{\lambda}}$ and $\mathcal{L}_{\vec{\lambda}}^\circ$
are only of class $C^1$, the construction of the Gromoll-Meyer pair
on the pages 49-51 of \cite{Ch1} is also effective for them
(see \cite{ChGh}). Hence the result can be obtained by repeating
the proof of \cite[Theorem~I.5.4]{Ch1}.
Of course,  with a stability theorem of critical groups the present case can also be reduced to
that of \cite[Theorem~I.5.4]{Ch1}. See \cite{Lu7} for a detailed proof.
\end{proof}

\subsection{Splitting and shifting theorems around critical orbits}\label{sec:S.6}

 We shall list main results and related corollaries for convenience of later applications as
 in Section~\ref{sec:Funct} and \cite{Lu8}. Outlines for their proofs are also given because
 our methods are completely different from those in the literature.
 Let $H$ be a Hilbert space with inner product $(\cdot,\cdot)_H$ and
let $({\cal H}, (\!(\cdot, \cdot)\!))$ be a $C^3$ Hilbert-Riemannian  manifold modeled on $H$.
Let $\mathcal{ O}\subset{\cal H}$
be a compact $C^3$ submanifold without boundary, and let $\pi:
N\mathcal{ O}\to \mathcal{ O}$ denote the normal bundle of it. The
bundle is a $C^2$-Hilbert vector bundle over $\mathcal{ O}$, and can
be considered as a subbundle of $T_\mathcal{ O}{\cal H}$ via the
Riemannian metric $(\!(\cdot, \cdot)\!)$. The metric $(\!(\cdot, \cdot)\!)$
induces a natural $C^2$ orthogonal bundle
projection ${\bf \Pi}:T_{\mathcal{O}}\mathcal{H}\to N\mathcal{O}$. For $\varepsilon>0$,
the so-called normal disk bundle of radius $\varepsilon$ is denoted by
$N\mathcal{ O}(\varepsilon):=\{(x,v)\in N\mathcal{O}\,|\,\|v\|_{x}<\varepsilon\}$.
 If $\varepsilon>0$ is small enough  the exponential map $\exp$ gives a $C^2$-diffeomorphism
 $\digamma$ from  $N\mathcal{ O}(\varepsilon)$ onto an open
neighborhood of $\mathcal{ O}$ in ${\cal H}$, $\mathcal{
N}(\mathcal{ O},\varepsilon)$.
For $x\in\mathcal{ O}$, let  $\mathscr{L}_s(N\mathcal{O}_x)$ denote the space
of those operators $S\in \mathscr{L}(N\mathcal{ O}_x)$ which are self-adjoint
with respect to the inner product $(\!(\cdot, \cdot)\!)_x$, i.e.
$(\!(S_xu, v)\!)_x=(\!(u, S_xv)\!)_x$ for all $u, v\in N\mathcal{
O}_x$. Then we have a $C^1$ vector bundle $\mathscr{L}_s(N\mathcal{ O})\to
\mathcal{O}$ whose fiber at $x\in\mathcal{ O}$ is given by
$\mathscr{L}_s(N\mathcal{ O}_x)$.

 Let  $\mathcal{L}:{\cal H}\to\mathbb{R}$ be a $C^1$  functional.
  A connected $C^3$ submanifold $\mathcal{O}\subset
{\cal H}$ is called a {\it critical manifold} of $\mathcal{L}$ if
$\mathcal{L}|_\mathcal{O}={\rm const}$ and $D\mathcal{L}(x)[v]=0$
for any $x\in\mathcal{ O}$ and $v\in T_x{\cal H}$. If there exists a
neighborhood ${\cal V}$ of $\mathcal{O}$ such that ${\cal V}\setminus \mathcal{O}$
contains no critical points of $\mathcal{L}$ we say  $\mathcal{O}$  to be {\it isolated}.
We make:

\begin{hypothesis}\label{hyp:S.6.1}
{\rm The gradient field
 $\nabla\mathcal{L}:\mathcal{H}\to T\mathcal{H}$ is G\^ateaux differentiable
 and thus  we there exists an operator $d^2\mathcal{L}(x)\in \mathscr{L}_s(T_x\mathcal{H})$
 for each $x\in\mathcal{O}$; moreover,
  $\mathcal{O}\ni x\mapsto d^2\mathcal{L}(x)$ is a continuous section of
$\mathscr{L}_s(T\mathcal{H})\to\mathcal{O}$,
$\dim{\rm Ker}(d^2\mathcal{L}(x))={\rm const}\;\forall
x\in\mathcal{O}$, and there exists $a_0>0$ such that
$\sigma(d^2\mathcal{L}(x))\cap([-2a_0,
2a_0]\setminus\{0\})=\emptyset,\;\forall x\in\mathcal{ O}$.}
\end{hypothesis}

This implies that
$\mathcal{O}\ni x\mapsto\mathcal{B}_x(\theta_x):=
{\bf \Pi}_x\circ d^2\mathcal{L}(x)|_{N{\cal O}_x}=d^2(\mathcal{L}\circ\exp_x|_{N{\cal O}_x})(\theta_x)$
        is a continuous section of  $\mathscr{L}_s(N\mathcal{O})\to\mathcal{O}$,
$\dim{\rm Ker}({\cal B}_x(\theta_x))={\rm const}\;\forall
x\in\mathcal{O}$, and
$\sigma({\cal B}_x(\theta_x))\cap([-2a_0, 2a_0]\setminus\{0\})=\emptyset$ for all $x\in\mathcal{O}$.
Let $\chi_\ast$ ($\ast=+, -, 0$) be the characteristic function of
the intervals $[2a_0, +\infty)$, $(-2a_0, 2a_0)$ and $(-\infty,
-2a_0]$, respectively. Then we have the orthogonal bundle
projections on the normal bundle $N\mathcal{ O}$, $P^\ast$ (defined
by $P^\ast_x(v)=\chi_\ast({\cal B}_x(\theta_x))v$), $\ast=+, -, 0$.
Denote by $N^\ast\mathcal{ O}=P^\ast N\mathcal{O}$, $\ast=+, -, 0$.
(Clearly,  ${\cal B}_x(\theta_x)(N^\ast\mathcal{ O}_x)\subset
N^\ast\mathcal{O}_x$ for any $x\in\mathcal{ O}$ and $\ast=+, -, 0$).
By \cite[Lem.7.4]{Ch}, we have
$N\mathcal{ O}=N^+\mathcal{ O}\oplus N^-\mathcal{ O}\oplus
N^0\mathcal{O}$. If ${\rm rank}N^0\mathcal{O}=0$, the critical orbit $\mathcal{O}$ is called {\it nondegenerate}.

In the following we only consider the case $\mathcal{O}$ is a critical orbit of a compact Lie group.
The general case can be treated  as in \cite{Lu4}.
The following assumption implies naturally Hypothesis~\ref{hyp:S.6.1} in this case.

\begin{hypothesis}\label{hyp:S.6.2}
{\rm {\bf (i)} Let $G$ be a compact
Lie group, and let  ${\cal H}$ be a $C^3$ Hilbert-Riemannian $G$-space
(that is, ${\cal H}$ is a $C^3$ $G$-Hilbert manifold with a Riemannian
metric $(\!(\cdot,\cdot)\!)$ such that $T{\cal H}$ is a $C^2$ Riemannian $G$-vector bundle, see \cite{Was}).\\
 {\bf (ii)} The $C^1$ functional $\mathcal{ L}:\mathcal{H}\to\mathbb{R}$ is $G$-invariant,
 $\nabla\mathcal{L}:\mathcal{H}\to T\mathcal{H}$ is G\^ateaux differentiable
 (i.e., under any $C^3$ local chart the functional $\mathcal{L}$
 has a G\^ateaux differentiable gradient map), and $\mathcal{ O}$ is an isolated
 critical orbit which is a $C^3$ critical submanifold with  Morse index $\mu_\mathcal{O}$.}
\end{hypothesis}

Since $\exp_{g\cdot x}(g\cdot v)=g\cdot\exp_x(v)$ for any $g\in G$
and $(x,v)\in T{\cal H}$, we have
$\mathcal{L}\circ\exp(g\cdot x,g\cdot v)=\mathcal{ L}(\exp(g\cdot x,
g\cdot v))=\mathcal{ L}(g\cdot\exp(x,v))=\mathcal{ L}(\exp(x,v))$.
 It follows  that $g^{-1}\cdot\nabla\mathcal{L}(g\cdot x)=  \nabla\mathcal{L}(x)$ and
\begin{equation}\label{e:S.6.2}
\nabla\left(\mathcal{L}\circ\exp|_{N\mathcal{O}(\varepsilon)_{gx}}\right)(g\cdot v)= g\cdot \nabla\left(\mathcal{L}\circ\exp|_{N\mathcal{O}(\varepsilon)_x}\right)(v)
\end{equation}
for any $g\in G$ and $(x,v)\in N\mathcal{O}(\varepsilon)_x$, which leads to
\begin{equation}\label{e:S.6.3}
d^2\left(\mathcal{L}\circ\exp|_{N\mathcal{O}(\varepsilon)_{gx}}\right)(g\cdot v)\cdot g= g\cdot d^2\left(\mathcal{L}\circ\exp|_{N\mathcal{O}(\varepsilon)_x}\right)(v)
\end{equation}
as bounded linear operators from $N\mathcal{O}_x$ onto $N\mathcal{O}_{gx}$.

Corresponding to Theorems~\ref{th:S.3.1},\ref{th:S.5.3} we have the following two theorems.

\begin{theorem}[Parameterized Morse-Palais lemma around critical orbits]\label{th:S.6.0}
Under Hypothesis~\ref{hyp:S.6.2},
let for some $x_0\in\mathcal{ O}$ the pair $\bigl(\mathcal{L}\circ\exp_{x_0},
B_{T_{x_0}\mathcal{H}}(\theta,\varepsilon)\bigr)$
(and so the pair $(\mathcal{L}\circ\exp|_{N\mathcal{O}(\varepsilon)_{x_0}},  N\mathcal{O}(\varepsilon)_{x_0})$
by Lemma~\ref{lem:S.2.4}) satisfies the corresponding conditions in Hypothesis~\ref{hyp:1.1} with $X=H$.
Let $G$-invariant functionals $\mathcal{G}_j\in C^1(\mathcal{H},\mathbb{R})$,
$j=1,\cdots,n$, have value zero and vanishing derivative at each point of $\mathcal{O}$, and also
fulfill:
\begin{description}
\item[(i)] gradients $\nabla\mathcal{G}_j$ have  G\^ateaux derivatives
$\mathcal{G}''_j(u)$ at each point $u$ near $\mathcal{O}$,
\item[(ii)] $\mathcal{G}''_j(u)$ are continuous at each point $u\in\mathcal{O}$ (and hence
each $\mathcal{G}_j$ is of class $C^{2-0}$ near $\mathcal{O}$).
\end{description}
Suppose that the critical orbit $\mathcal{O}$ is nondegenerate. Then there exist $\delta>0$, $\epsilon>0$
and  a continuous map
$\Phi: [-\delta,\delta]^n\times N^0\mathcal{ O}(\epsilon)\oplus N^+\mathcal{
O}(\epsilon)\oplus N^-\mathcal{ O}(\epsilon)\to N\mathcal{O}$
such that each $\Phi(\vec{\lambda},\cdot): N^+\mathcal{O}(\epsilon)\oplus N^-\mathcal{ O}(\epsilon)\to N\mathcal{O}$ is
 a  $G$-equivariant homeomorphism onto an open neighborhood of
the zero section preserving fibers, and that  $\mathcal{L}_{\vec{\lambda}}:=\mathcal{L}+ \sum^n_{j=1}\mathcal{G}_j$ satisfies
\begin{eqnarray}\label{e:S.6.4}
\mathcal{L}_{\vec{\lambda}}\circ\exp\circ\Phi(\vec{\lambda}, x,  u^++
u^-)=\|u^+\|^2_x-\|u^-\|^2_x+ \mathcal{L}_{\vec{\lambda}}|_{\mathcal{O}}
\end{eqnarray}
for any $\vec{\lambda}\in [-\delta,\delta]^n$, $x\in\mathcal{O}$ and  $(u^+, u^-)\in  N^+\mathcal{ O}(\epsilon)_x\times N^-\mathcal{O}(\epsilon)_x$.
\end{theorem}

This theorem will be proved after the proof of the following theorem.

\begin{theorem}[Parameterized Splitting Theorem around critical orbits]\label{th:S.6.1}
Suppose that the critical orbit $\mathcal{O}$ in Theorem~\ref{th:S.6.0} is degenerate,
i.e., ${\rm rank}N^0\mathcal{O}>0$. Then  for sufficiently small $\epsilon>0$, $\delta>0$,
the following hold:\\
\noindent{\bf (I)}\;  There exists a unique continuous map
 $$
\mathfrak{h}:[-\delta,\delta]^n\times N^0\mathcal{O}(3\epsilon)\to N^+\mathcal{O}\oplus N^-\mathcal{O},\;
(\vec{\lambda}, x,v)\mapsto \mathfrak{h}_x(\vec{\lambda}, v),
$$
such that for each $\vec{\lambda}\in[-\delta,\delta]^n$, $\mathfrak{h}(\vec{\lambda},\cdot):N^0\mathcal{O}(3\epsilon)\to N^+\mathcal{O}\oplus N^-\mathcal{O}$ is a $G$-equivariant topological  bundle
morphism that preserves the zero section and satisfies
\begin{eqnarray*}
(P^+_x+P^-_x)\circ{\bf \Pi}_x\nabla(\mathcal{L}_{\vec{\lambda}}\circ\exp_x)(v+ \mathfrak{h}_x(\vec{\lambda}, v))=0\quad\forall
(x,v^0)\in N^0\mathcal{O}(\epsilon).
\end{eqnarray*}
\noindent{\bf (II)}\;  There exists a continuous map
$\Phi: [-\delta,\delta]^n\times N^0\mathcal{ O}(\epsilon)\oplus N^+\mathcal{
O}(\epsilon)\oplus N^-\mathcal{ O}(\epsilon)\to N\mathcal{ O}$
such that for each $\vec{\lambda}\in[-\delta,\delta]^n$, $\Phi(\vec{\lambda},\cdot):N^0\mathcal{ O}(\epsilon)\oplus N^+\mathcal{
O}(\epsilon)\oplus N^-\mathcal{ O}(\epsilon)\to N\mathcal{ O}$ is
 a  $G$-equivariant homeomorphism onto an open neighborhood of
the zero section preserving fibers, and such that
\begin{eqnarray*}
\mathcal{L}_{\vec{\lambda}}\circ\exp\circ\Phi(\vec{\lambda}, x, v, u^++
u^-)=\|u^+\|^2_x-\|u^-\|^2_x+ \mathcal{ L}_{\vec{\lambda}}\circ\exp_x(v+
\mathfrak{h}_x(\vec{\lambda}, v))
\end{eqnarray*}
for any $\vec{\lambda}\in [-\delta,\delta]^n$, $x\in\mathcal{O}$ and  $(v, u^+, u^-)\in N^0\mathcal{
O}(\epsilon)_x\times N^+\mathcal{ O}(\epsilon)_x\times N^-\mathcal{
O}(\epsilon)_x$. \\
\noindent{\bf (III)}\; For each $(\vec{\lambda}, x)\in [-\delta,\delta]^n\times\mathcal{O}$ the functional
\begin{eqnarray*}
N^0\mathcal{O}(\epsilon)_x\to\R,\;v\mapsto \mathcal{L}_{\vec{\lambda},x}^\circ(v):=
\mathcal{L}_{\vec{\lambda}}\circ\exp_x(v+ \mathfrak{h}_x(\vec{\lambda},v))
\end{eqnarray*}
is $G_x$-invariant, of class $C^{1}$,  and has differential given by
$$
D\mathcal{L}_{\vec{\lambda}, x}^\circ(v)[v']=
D(\mathcal{L}_{\vec{\lambda}}\circ\exp_x)(v+ \mathfrak{h}_x(\vec{\lambda},v))[v'],\quad\;\forall v'\in N^0\mathcal{O}_x.
$$
Moreover, each $\mathfrak{h}_x(\vec{\lambda},\cdot)$ is of class $C^{1-0}$, and  if $\mathcal{L}$ is of class $C^{2-0}$
so is $\mathcal{L}_{\vec{\lambda}, x}^\circ$.
\end{theorem}

\begin{proof}[\bf Proof]\quad
We only outline main procedures in case $\vec{\lambda}={\bf 0}$, i.e., $\mathcal{L}_{\vec{\lambda}}=\mathcal{L}$. By the assumption and  (\ref{e:S.6.3})
we deduce that each pair
$(\mathcal{L}\circ\exp|_{N\mathcal{O}(\varepsilon)_{x}},  N\mathcal{O}(\varepsilon)_{x})$
satisfies the corresponding conditions with Hypothesis~\ref{hyp:1.1} with $X=H$ too,
and that there exists $a_0>0$ such that
\begin{eqnarray}\label{e:S.6.10}
\sigma\left(d^2\left(\mathcal{L}\circ\exp|_{N\mathcal{O}(\varepsilon)_x}\right)(\theta_x)\right)\cap([-2a_0,
2a_0]\setminus\{0\})=\emptyset,\quad\forall x\in\mathcal{O}.
\end{eqnarray}
By Theorem~\ref{th:S.1.2} we have  $\epsilon\in (0, \varepsilon/3)$ and a continuous map
$\mathfrak{h}_{x_0}: N^0\mathcal{ O}(3\epsilon)_{x_0}\to
 N^\pm\mathcal{ O}(\varepsilon/2)_{x_0}$,
 such that $\mathfrak{h}_{x_0}(g\cdot v)=g\cdot \mathfrak{h}_{x_0}(v)$,
 $\mathfrak{h}_{x_0}(\theta_{x_0})=\theta_{x_0}$ and
$$
(P^+_{x_0}+ P^-_{x_0})\nabla\left(\mathcal{L}\circ\exp|_{N\mathcal{O}(\varepsilon)_{x_0}}\right)(v+
\mathfrak{h}_{x_0}(v))=0,\quad\forall v\in N^0\mathcal{
O}(3\epsilon)_{x_0}.
$$
Furthermore, the function
$\mathcal{ L}^\circ_{x_0}: N^0\mathcal{O}(\epsilon)_{x_0}\to
\R,\;v\mapsto \mathcal{ L}\circ\exp_{x_0}(v+ \mathfrak{h}_{x_0}(v))$
is of class $C^{1}$, and
$D\mathcal{L}^\circ_{x_0}(v)[u]=D(\mathcal{L}\circ\exp|_{N\mathcal{O}(\varepsilon)_{x_0}})(v+
\mathfrak{h}_{x_0}(v))[u]$. Define
\begin{eqnarray*}
\mathfrak{h}: N^0\mathcal{ O}(3\epsilon)\to T{\cal H},\quad
(x,v)\mapsto g\cdot\mathfrak{h}_{x_0}(g^{-1}\cdot v),
\end{eqnarray*}
where $g\cdot x_0=x$. We claim: {\it $\mathfrak{h}$ is continuous}. Otherwise, there exists a sequence
$(x_j, v_j)\subset  N^0\mathcal{ O}(3\epsilon)$ converging  to
a point $(\bar{x},\bar{v})\in  N^0\mathcal{O}(3\epsilon)$, such that
 $(\mathfrak{h}(x_j, v_j))$ has no intersection with an open neighborhood ${\bf U}$ of $\mathfrak{h}(\bar{x},\bar{v})$
 in $T{\cal H}$. Let $\bar{g}, g_j\in G$ be such that
 $\bar{g}\cdot x_0=\bar{x}$ and $g_j\cdot x_0=x_j$, $j=1,2,\cdots$.
 Then $\mathfrak{h}(\bar{x},\bar{v})=\bar{g}\cdot\mathfrak{h}_{x_0}(\bar{g}^{-1}\cdot\bar{v})$
 and $\mathfrak{h}(x_j, v_j)=g_j\cdot\mathfrak{h}_{x_0}(g_j^{-1}\cdot{v}_j)$
 for each $j\in\mathbb{N}$. Note that $\bar{g}^{-1}\cdot{\bf U}$ is an open neighborhood of
 $\mathfrak{h}_{x_0}(\bar{g}^{-1}\cdot\bar{v})=\bar{g}^{-1}\cdot\mathfrak{h}(\bar{x},\bar{v})$
 and that the sequences $\bar{g}^{-1}\cdot\mathfrak{h}(x_j, v_j)=\bar{g}^{-1}\cdot g_j\cdot\mathfrak{h}_{x_0}(g_j^{-1}\cdot{v}_j)$
 have no intersection with $\bar{g}^{-1}\cdot{\bf U}$.
  Since $G$ is compact,  we may assume $\bar{g}^{-1}\cdot g_j\to \hat{g}\in G$
  and so $g_j^{-1}\to (\bar{g}\hat{g})^{-1}\in G$  after passing to a
  subsequence (if necessary). Then
 $$
 \bar{g}^{-1}\cdot\mathfrak{h}(x_j, v_j)=\bar{g}^{-1}\cdot g_j\cdot\mathfrak{h}_{x_0}(g_j^{-1}\cdot{v}_j)
 \to \hat{g}\cdot\mathfrak{h}_{x_0}((\bar{g}\hat{g})^{-1}\cdot\bar{v})=
 \mathfrak{h}_{x_0}(\bar{g}^{-1}\cdot\bar{v}).
 $$
 It follows that $\mathfrak{h}_{x_0}(\bar{g}^{-1}\cdot\bar{v})$ does not belong to
$\bar{g}^{-1}\cdot{\bf U}$. This contradicts the fact that $\bar{g}^{-1}\cdot{\bf U}$ is an open neighborhood of
 $\mathfrak{h}_{x_0}(\bar{g}^{-1}\cdot\bar{v})$.

By the definition of $\mathfrak{h}$, it is clearly $G$-equivariant and
satisfies
\begin{equation}\label{e:S.6.11}
(P^+_{x}+ P^-_{x})\nabla\left(\mathcal{L}\circ\exp|_{N\mathcal{O}_{x}(\varepsilon)}\right)(v+
\mathfrak{h}_{x}(v))=0,\quad\forall (x,v)\in N^0\mathcal{
O}(3\epsilon).
\end{equation}
 Moreover, the map $\mathcal{F}: N^0\mathcal{O}(\epsilon)\oplus N^+\mathcal{O}(\epsilon)\oplus N^-\mathcal{O}(\epsilon)\to\R$ defined by
\begin{eqnarray}\label{e:S.6.12}
\mathcal{ F}(x, v, u^++u^-)&=&\mathcal{F}_x(v, u^++u^-)\nonumber\\
&=&\mathcal{ L}\circ\exp_x(v+\mathfrak{h}_x(v)+ u^++u^-)-\mathcal{ L}\circ\exp_x(v+
\mathfrak{h}_x(v)),\quad
\end{eqnarray}
is $G$-invariant, and satisfies for any $(x,v)\in N^0\mathcal{O}(\epsilon)$ and $u\in N^+\mathcal{ O}_x\oplus
N^-\mathcal{O}_x$,
\begin{eqnarray}\label{e:S.6.13}
\mathcal{F}_x(v, \theta_x)=0\quad\hbox{and}\quad D_2\mathcal{F}_x(v, \theta_x)[u]=0.
\end{eqnarray}
By (\ref{e:S.6.2}), (\ref{e:S.6.3}) and Lemmas~\ref{lem:S.2.1},~\ref{lem:S.2.2}  we can
immediately obtain:

\begin{lemma}\label{lem:S.6.2}
There exist positive numbers $\varepsilon_1\in (0, \varepsilon)$ and
$a_1\in (0, 2a_0)$, and a function $\Omega:N\mathcal{
O}(\varepsilon_1)\to [0, \infty)$ with the property that
$\Omega(x,v)\to 0$ as $\|v\|_x\to 0$,
 such that for any $(x,v)\in N\mathcal{ O}(\varepsilon_1)$ the
 following conclusions hold with ${\cal B}_x=d^2\left(\mathcal{L}\circ\exp|_{N\mathcal{O}_x(\varepsilon)}\right)$:
\begin{description}
\item[(i)]  $|(\!({\cal B}_x(v)u, w)\!)_x- (\!({\cal B}_x(\theta_x)u, w)\!)_x |\le \Omega(x,v)
\|u\|_x\cdot\|w\|_x$ for any $u\in N^0\mathcal{ O}_x\oplus
N^-\mathcal{ O}_x$ and $w\in N\mathcal{ O}_x$;
\item[(ii)] $(\!({\cal B}_x(v)u, u)\!)_x\ge a_1\|u\|^2_x$ for all $u\in N^+\mathcal{ O}_x$;
\item[(iii)] $|(\!({\cal B}_x(v)u,w)_x\!)|\le\Omega(x,v)\|u\|_x\cdot\|w\|_x$
for all $u^+\in N^+\mathcal{ O}_x,  w\in N^-\mathcal{
O}_x\oplus N^0\mathcal{ O}_x$;
\item[(iv)] $(\!({\cal B}_x(v)u,u)_x\le-a_0\|u\|^2$ for all $u\in N^-\mathcal{ O}_x$.
\end{description}
\end{lemma}

Let us choose $\varepsilon_2\in (0,\epsilon/2)$ so small that
$(x, v^0+ \mathfrak{h}_x(v^0)+ u^++ u^-)\in  N\mathcal{O}(\varepsilon_1)$
for $(x, v^0)\in N^0\mathcal{ O}(2\varepsilon_2)$ and $(x, u^\ast)\in
N^\ast\mathcal{ O}(2\varepsilon_2)$, $\ast=+, -$.
As in the proof of \cite[Lemma~3.5]{Lu2}, we may use
 \cite[Lemma~2.4]{Lu2}
to derive

\begin{lemma}\label{lem:S.6.3}
Let the constants $a_1$ and $a_0$  be given by
Lemma~\ref{lem:S.6.2}(ii),(iv).
For the above $\varepsilon_2>0$ and each $x\in\mathcal{O}$ the
restriction of the functional $\mathcal{ F}_x$  to
$\overline{N^0\mathcal{
O}(2\varepsilon_2)_x}\oplus[\overline{N^+\mathcal{
O}(2\varepsilon_2)_x}\oplus\overline{N^-\mathcal{
O}(2\varepsilon_2)_x}]$
 satisfies:
\begin{description}
\item[(i)] $D_2\mathcal{ F}_x(v^0, u^++ u^-_2)[u^-_2-u^-_1]- D_2\mathcal{ F}_x(v^0, u^++ u^-_1)[u^-_2-u^-_1] \le -a_1\|u^-_2-u^-_1\|^2_x$ for any $(x, v^0)\in\overline{N^0\mathcal{ O}(2\varepsilon_2)}$,
$(x, u^+)\in \overline{N^+\mathcal{ O}(2\varepsilon_2)}$
and $(x, u^-_j)\in \overline{N^-\mathcal{ O}(2\varepsilon_2)}$, $j=1,2$;

\item[(ii)] $D_2\mathcal{ F}_x(v^0, u^++ u^-)[u^+-u^-]\ge a_1\|u^+\|^2_x+
a_0\|u^-\|^2_x$ for any $(x, v^0)\in\overline{N^0\mathcal{ O}(2\varepsilon_2)}$ and
$(x, u^\ast)\in \overline{N^\ast\mathcal{ O}(2\varepsilon_2)}$, $\ast=+,-$;

\item[(iii)] $D_2\mathcal{ F}_x(v^0, u^+)[u^+] \ge a_1\|u^+\|^2_x$
for any $(x, v^0)\in\overline{N^0\mathcal{ O}(2\varepsilon_2)}$ and
$(x, u^+)\in \overline{N^+\mathcal{ O}(2\varepsilon_2)}$.
\end{description}
\end{lemma}

Denote by bundle projections $\Pi_0:\overline{N^0\mathcal{ O}(\varepsilon_2)}\to\mathcal{ O}$ and
$\Pi_{\pm}:N^+\mathcal{O}\oplus N^-\mathcal{O}\to\mathcal{O}$,
$\Pi_{\ast}:N^\ast\mathcal{O}\to\mathcal{O}$, $\ast=+,-$.
Let $\Lambda=\overline{N^0\mathcal{ O}(2\varepsilon_2)}$, $p:\mathcal{E}\to\Lambda$
and $p_\ast:\mathcal{E}^\ast\to\Lambda$ be the pullbacks of
$N^+\mathcal{O}\oplus N^-\mathcal{O}$ and
$N^\ast\mathcal{O}$ via $\Pi_0$, $\ast=+,-$. Then $\mathcal{E}=\mathcal{E}^+\oplus\mathcal{E}^-$,
and for $\lambda=(x,v^0)\in\Lambda$ we have $\mathcal{E}_\lambda=
N^+\mathcal{O}_x\oplus N^-\mathcal{O}_x$ and
$\mathcal{E}^\ast_\lambda=N^\ast\mathcal{O}_x$, $\ast=+,-$. Moreover, for each $\eta>0$ we write
\begin{eqnarray*}
&&B_\eta(\mathcal{E})=\left\{(\lambda, w)\,|\, \lambda=(x,v^0)\in\Lambda\;\&\;w\in
(N^+\mathcal{O}\oplus N^-\mathcal{O})_x(\eta)\right\},\\
&&\bar{B}_\eta(\mathcal{E})=\left\{(\lambda, w)\,|\, \lambda=(x,v^0)\in\Lambda\;\&\;w\in
\overline{(N^+\mathcal{O}\oplus N^-\mathcal{O})_x(\eta)}\right\}.
\end{eqnarray*}
Similarly,  $B_\eta(\mathcal{E}^\ast)$ and $\bar{B}_\eta(\mathcal{E}^\ast)$ ($\ast=+,-$)
are defined. Let  $\mathcal{J}: B_{2\varepsilon_2}(\mathcal{E})\to\R$ be given by
\begin{eqnarray}\label{e:S.6.14}
\mathcal{J}(\lambda,  v^\pm)=\mathcal{J}_\lambda(v^\pm) =\mathcal{F}(x, v^0, v^\pm),
\quad\forall\lambda=(x, v^0)\in\Lambda\;\&\;\forall v^\pm\in B_{2\varepsilon_2}(\mathcal{E})_\lambda.
\end{eqnarray}
It is continuous, and of class $C^1$ with respect to $v^\pm$. From (\ref{e:S.6.13}) and Lemma~\ref{lem:S.6.3}
we directly obtain:

\begin{lemma}\label{lem:S.6.4}
The functional $\mathcal{J}_\lambda$ satisfies the conditions in
Theorem~A.2 of \cite{Lu2} (the bundle parameterized version of \cite[Theoren~1.1]{DHK}), that is,
\begin{description}
\item[(i)] $\mathcal{J}_\lambda(\theta_\lambda)=0$ and $D\mathcal{J}_\lambda(\theta_\lambda)=0$;
\item[(ii)] $D\mathcal{J}_\lambda(u^++ u^-_2)[u^-_2-u^-_1]- D\mathcal{J}_\lambda(u^++ u^-_1)[u^-_2-u^-_1] \le -a_1\|u^-_2-u^-_1\|^2_x$ for any
$\lambda=(x,v^0)\in\Lambda$, $u^+\in \bar{B}_{\varepsilon_2}(\mathcal{E}^+)_\lambda$ and
$u^-_j\in \bar{B}_{\varepsilon_2}(\mathcal{E}^-)_\lambda$, $j=1,2$;

\item[(iii)] $D\mathcal{J}_\lambda(\lambda, u^++ u^-)[u^+-u^-] \ge a_1\|u^+\|^2_x+
a_0\|u^-\|^2_x$ for any $\lambda=(x,v^0)\in\Lambda$ and
$u^\ast\in \bar{B}_{\varepsilon_2}(\mathcal{E}^\ast)_\lambda$, $\ast=+,-$;

\item[(iv)] $D\mathcal{J}_\lambda(u^+)[u^+] \ge a_1\|u^+\|^2_x$
for any $\lambda=(x, v^0)\in\Lambda$ and
$u^+\in \bar{B}_{\varepsilon_2}(\mathcal{E}^+)_\lambda$.
\end{description}
\end{lemma}

By this  we can use Theorem~A.2 of \cite{Lu2}
to get $\epsilon\in (0, \varepsilon_2)$, an open neighborhood
$U$ of the zero section $0_\mathcal{ E}$ of $\mathcal{ E}$ in $B_{2\varepsilon_2}(\mathcal{E})$ and a
homeomorphism
\begin{eqnarray}\label{e:S.6.15}
\phi: B_{\epsilon}(\mathcal{ E}^+) \oplus
B_{\epsilon}(\mathcal{ E}^-)\to U,\;(\lambda, u^++ u^-)\mapsto (\lambda, \phi_\lambda(u^++u^-))
\end{eqnarray}
such that for all $(\lambda, u^++ u^-)\in  B_{\epsilon}(\mathcal{
E}^+) \oplus B_{\epsilon}(\mathcal{ E}^-)$ with $\lambda=(x,v^0)\in\Lambda$,
\begin{eqnarray}\label{e:S.6.16}
J(\phi(\lambda, u^++
u^-))=\|u^+\|^2_{x}-\|u^-\|^2_{x}.
\end{eqnarray}
 Moreover,
for each $\lambda\in\Lambda$,
$\phi_\lambda(\theta_\lambda)=\theta_\lambda$, $\phi_\lambda(x+y)\in
\mathcal{E}^-_\lambda$ if and only if $x=\theta_\lambda$, and
$\phi$ is a homoeomorphism from $B_{\epsilon}(\mathcal{
E}^-)$ onto $U\cap \mathcal{ E}^-$.

Note that $B_{\epsilon}(\mathcal{ E}^+) \oplus
B_{\epsilon}(\mathcal{ E}^-)= \overline{N^0\mathcal{
O}(2\varepsilon_2)}\oplus N^+\mathcal{O}(\epsilon)\oplus N^-\mathcal{
O}(\epsilon)$ and
$U=\overline{N^0\mathcal{O}(2\varepsilon_2)}\oplus\widehat{U}$,
where $\widehat{U}$ is an open neighborhood
 of the zero section of $N^+\mathcal{O}\oplus N^-\mathcal{
O}$ in $N^+\mathcal{O}(2\varepsilon_2)\oplus N^-\mathcal{
O}(\varepsilon_2)$. Let $\mathcal{W}=N^0\mathcal{O}(\epsilon)\oplus\widehat{U}$,
which is an open neighborhood
 of the zero section of $N\mathcal{O}$ in
$N^0\mathcal{O}(2\varepsilon_2)\oplus N^+\mathcal{O}(2\varepsilon_2)\oplus N^-\mathcal{
O}(\varepsilon_2)$. By (\ref{e:S.6.15}) we get a
 homeomorphism
\begin{eqnarray*}
\phi: N^0\mathcal{O}(\epsilon)\oplus N^+\mathcal{O}(\epsilon)\oplus N^-\mathcal{
O}(\epsilon)\to \mathcal{W},\;
(x,v, u^++ u^-)\mapsto (x,v, \phi_{(x,v)}(u^++u^-)),
\end{eqnarray*}
and therefore a topological embedding bundle morphism that preserves the
zero section,
\begin{align*}
\Phi: &N^0\mathcal{ O}(\epsilon)\oplus N^+\mathcal{
O}(\epsilon)\oplus N^-\mathcal{ O}(\epsilon)\to N\mathcal{ O},\\
&(x,v, u^++ u^-)\mapsto (x,v+ \mathfrak{h}_{x}(v), \phi_{(x,v)}(u^++u^-)).
\end{align*}
 From (\ref{e:S.6.12}), (\ref{e:S.6.14}) and (\ref{e:S.6.16}) it follows that $\Phi$ and $\phi$ satisfy
 \begin{eqnarray*}
 \mathcal{ L}\circ\exp\circ\Phi(x,v+u^++u^-)&=&
\mathcal{ L}\circ\exp_x(v+ \mathfrak{h}_{x}(v)+ \phi_{(x,v)}(u^++u^-))\\
&=&\|u^+\|^2_x-\|u^-\|_x^2+\mathcal{ L}\circ\exp_x(v+
\mathfrak{h}_{x}(v))
\end{eqnarray*}
for all $(x,v+u^+,u^-)\in N^0\mathcal{O}(\epsilon)\oplus N^+\mathcal{O}(\epsilon)\oplus N^-\mathcal{
O}(\epsilon)$. The other conclusions  easily follow from the above arguments.
Theorem~\ref{th:S.6.1} is proved.
\end{proof}

\begin{proof}[\bf Proof of Theorem~\ref{th:S.6.0}]
We also consider the case $\vec{\lambda}={\bf 0}$ merely.
In the present case Lemma~\ref{lem:S.6.2} also holds with $N^0\mathcal{O}_x=\{\theta_x\}\;\forall x\in\mathcal{O}$.
 But we need to replace the map $\mathcal{F}$ in (\ref{e:S.6.12}) by
$$
\mathcal{F}(x,  u^++u^-): N^+\mathcal{O}(\epsilon)\oplus N^-\mathcal{
O}(\epsilon)\to\R,\; (x, u^++u^-)\mapsto\mathcal{ L}\circ\exp_x( u^++u^-).
$$
For any $x\in\mathcal{O}_x$, let $\mathcal{F}_x$ be the restriction of $\mathcal{F}$ to
$N^+\mathcal{O}(\epsilon)_x\oplus N^-\mathcal{O}(\epsilon)_x$. As in the
proof of Theorem~\ref{th:S.1.1}, Lemma~\ref{lem:S.6.3}
is still true with $\overline{N^0\mathcal{
O}(2\varepsilon_2)_x}=\{\theta_x\}$.
Then the desired conclusions can be obtained
by  applying \cite[Theorem~A.2]{Lu2} to $\Lambda=\mathcal{O}$
and $J_\lambda=\mathcal{F}_x$ with $\lambda=x\in\mathcal{O}$.
\end{proof}

 As in \cite{Bot,Ch, MaWi, Wa, Was}, from Theorems~\ref{th:S.6.0}, \ref{th:S.6.1}
 we may, respectively, deduce

\begin{corollary}\label{cor:S.6.5}
 Under the assumptions of Theorem~\ref{th:S.6.0}, let
$\theta^-$ be the orientation bundle (or sheaf) of $N^-\mathcal{O}$ and
 ${\bf K}$ any commutative ring. Then
it holds that
\begin{eqnarray}\label{e:S.6.17}
C_\ast(\mathcal{L}_{\vec{\lambda}}, \mathcal{O};{\bf K})\cong
H_{\ast-\mu_\mathcal{O}}(\mathcal{O};\theta^-\otimes{\bf K})\quad\hbox{and}\quad
C_G^\ast(\mathcal{L}_{\vec{\lambda}}, \mathcal{O};{\bf K})\cong
H_G^{\ast-\mu_\mathcal{O}}(\mathcal{O};\theta^-\otimes{\bf K}),
\end{eqnarray}
where for $q\in\mathbb{N}_0$, $C^q_G(\mathcal{L}_{\vec{\lambda}},\mathcal{O};{\bf K})=H^q(E\times_G((\mathcal{L}_{\vec{\lambda}})_c\cap U),
E\times_G(((\mathcal{L}_{\vec{\lambda}})_c\setminus\mathcal{O})\cap U);{\bf K})$
is the so-called the $q^{\rm th}$ $G$ critical group of $\mathcal{O}$ defined
with a universal smooth principal $G$-bundle $E\to B_G$, a $G$-invariant neighborhood $U$ of $\mathcal{O}$
and $c=\mathcal{L}_{\vec{\lambda}}(\mathcal{O})$. In particular, for ${\bf K}=\mathbb{Z}_2$ there hold
\begin{eqnarray}\label{e:S.6.18}
C_\ast(\mathcal{L}_{\vec{\lambda}}, \mathcal{O};\mathbb{Z}_2)\cong
H_{\ast-\mu_\mathcal{O}}(\mathcal{O};\mathbb{Z}_2)\quad\hbox{and}\quad
C_{G}^\ast(\mathcal{L}_{\vec{\lambda}}, \mathcal{O};\mathbb{Z}_2)\cong
H_G^{\ast-\mu_\mathcal{O}}(\mathcal{O};\mathbb{Z}_2).
\end{eqnarray}
\end{corollary}

\begin{corollary}[Shifting Theorem]\label{cor:S.6.6}
 Under the assumptions of Theorem~\ref{th:S.6.1}, if $\mathcal{ O}$ has
trivial normal bundle then
$C_q(\mathcal{L}_{\vec{\lambda}}, \mathcal{ O};{\bf K})\cong
\oplus^q_{j=0}C_{q-j-\mu_\mathcal{O}}((\mathcal{L}_{\vec{\lambda}})^{\circ}_x,
\theta_x; {\bf K})\otimes H_j(\mathcal{ O};{\bf K})\;\forall q\in\mathbb{N}_0$
for any commutative group ${\bf K}$ and $x\in\mathcal{O}$.
  \end{corollary}

\section{A generalization of Marino--Prodi's perturbation theorem}\label{sec:MP}
\setcounter{equation}{0}

 Marino and Prodi  \cite{MP} studied local Morse function approximations for $C^2$
  functionals on Hilbert spaces.  We shall generalize their result to a class of functionals
    satisfying the following stronger assumption than Hypothesis~\ref{hyp:1.1}.

\begin{hypothesis}\label{hyp:MP.1}
{\rm Let  $V$ be an open set of a Hilbert space $H$ with inner product $(\cdot,\cdot)_H$,
and $\mathcal{L}\in C^1(V,\mathbb{R})$.
Assume that the gradient $\nabla\mathcal{L}$ has a G\^ateaux derivative $B(u)\in \mathscr{L}_s(H)$
 at every point $u\in V$, and that the map $B:V\to \mathscr{L}_s(H)$  has a decomposition
$B=P+Q$, where for each $u\in V$,  $P(u)\in\mathscr{L}_s(H)$ is positive definitive,
$Q(u)\in\mathscr{L}_s(H)$ is compact, and they also satisfy the following
properties:\\
{\bf (i)} For any $u\in H$, the map $V\ni x\mapsto P(x)u\in H$
is continuous;\\
{\bf (ii)} The  map $Q:V\to\mathscr{L}(H)$ is continuous;\\
{\bf (iii)} $P$ is local positive definite uniformly, i.e., each $u_0\in V$
 has a neighborhood $\mathscr{U}(u_0)$ such that for some
 constants $C_0>0$,
$(P(u)v, v)_H\ge C_0\|v\|^2$, $\forall v\in H$, $\forall u\in \mathscr{U}(u_0)$.
}
\end{hypothesis}

 As in the proofs of Theorems~\ref{th:4.1},~\ref{th:4.2} under Hypothesis~$\mathfrak{F}_{2,N,m,n}$,
we can check that the functional $\mathfrak{F}$ in (\ref{e:1.3}) satisfies this hypothesis.
By improving methods in \cite{MP, Ch, CiVa} we may prove

\begin{theorem}\label{th:MP.2}
Under Hypothesis~\ref{hyp:MP.1}, suppose: {\bf (a)} $u_0\in V$ is
a unique critical point of $\mathcal{L}$, {\bf (b)} the corresponding maps $\varphi$
and $\mathcal{L}^\circ$ as in Theorem~\ref{th:S.1.2} near $u_0$ are of classes $C^1$ and $C^2$,
respectively, {\bf (c)} $\mathcal{L}$
satisfies the (PS) condition. Then
for any $\epsilon>0$ and $r>0$ such that
$\bar{B}_H(u_0, r)\subset V$ and $\sup\{|\mathcal{L}(u)|\,|\,u\in \bar{B}_H(u_0, r)\}<\infty$,
there exists a functional $\tilde{\mathcal{L}}\in C^1(V,\mathbb{R})$ with the following properties:\\
{\bf (i)} $\tilde{\mathcal{L}}$ satisfies Hypothesis~\ref{hyp:MP.1} and the (PS) condition;\\
{\bf (ii)} $\sup_{u\in V}\|\mathcal{L}(u)-\tilde{\mathcal{L}}(u)\|<\epsilon$,
$\sup_{u\in V}\|\mathcal{L}'(u)-\tilde{\mathcal{L}}'(u)\|<\epsilon$ and
$\sup_{u\in V}\|\mathcal{L}''(u)-\tilde{\mathcal{L}}''(u)\|<\epsilon$,
where $\mathcal{L}''(u)$ and $\tilde{\mathcal{L}}''(u)$ are  G\^ateaux derivatives of
$\mathcal{L}'(u)$ and $\tilde{\mathcal{L}}'(u)$, respectively;\\
{\bf (iii)} $\mathcal{L}(x)=\tilde{\mathcal{L}}(x)$ if $x\in V$ and $\|u-u_0\|\ge r$;\\
{\bf (iv)} the critical points of $\tilde{\mathcal{L}}$, if any, are in ${B}_H(u_0, r)$ and nondegenerate
(so finitely many by the arguments below \ref{e:S.1.1}); moreover the Morse indexes of these critical
points sit in $[m^-, m^-+n^0]$, where $m^-$ and $n^0$ are the Morse index and nullity of $u_0$,
respectively.
\end{theorem}

As showed, the functionals in \cite{Lu1, Lu10}  satisfy the conditions of this theorem.
If $N=1$, $\dim\Omega=2$ and $F$ is smooth enough, we may also prove
under \textsf{Hypothesis} $\mathfrak{F}_{2,1,m,2}$ that
Theorem~\ref{th:MP.2} is applicable for the functional $\mathfrak{F}$  on $W^{m,2}_0(\Omega)$.
In general, under \textsf{Hypothesis} $\mathfrak{F}_{2,N,m,n}$,
for a critical point $\vec{u}$ of the functional $\mathfrak{F}_H$ on $H:=W^{m,2}_0(\Omega, \mathbb{R}^N)$
defined by the right side of (\ref{e:1.3}),  if there exist a real $p\ge 2$ and
  an integer $k>m+ \frac{n}{p}$  such that $\vec{u}\in C^k(\overline{\Omega}, \mathbb{R}^N)$,  and
  $F$ and $\partial\Omega$ are of classes $C^{k-m+2}$ and $C^{k-1,1}$, respectively,
   then Theorem~\ref{th:compare.1} (or Theorem~\ref{th:compare.1.4}) shows that
{\bf (b)} of Theorem~\ref{th:MP.2} can be satisfied for $\mathfrak{F}_H$ near $\vec{u}$.

Marino--Prodi's result has many important applications in the critical point theory,
see \cite{Ch, CiVa, Gh, LazS} and literature therein. With Theorem~\ref{th:MP.2} they may be given in our
framework. Moreover, it is very possible to give a corresponding result with Theorem~\ref{th:MP.2}
in the setting of \cite{Lu1,Lu2}.

   Marino--Prodi's perturbation theorem in \cite{MP}
was also generalized to the equivariant case under the finite (resp. compact Lie) group
action by Wasserman \cite{Was} (resp. Viterbo \cite{Vit1}), see the proof of Theorem~7.8 in
\cite[Chapter~I]{Ch} for full details.
Similarly, we can present an equivariant version of Theorem~\ref{th:MP.2} for compact
Lie group action, but it is omitted here.\\

\begin{proof}[\bf Proof of Theorem~\ref{th:MP.2}]
 Without loss of generality we may assume $\theta\in V$ and $u_0=\theta$.
By the assumption (b) we have a $C^2$ reduction functional $\mathcal{L}^\circ: B_H(\theta,\delta)\cap H^0\to
\mathbb{R}$ such that $\theta$ is the unique critical point of it.
In this case, from (\ref{e:S.1.2}) and (\ref{e:S.5.12.2}) with $\bar{\lambda}=0$ and $\psi(0,\cdot)=\varphi$
it follows that $d^2\mathcal{L}^\circ(\theta)=0$ and hence $\mathcal{L}^\circ(z)=o(\|z\|^2)$.
Clearly, we can shrink $\delta>0$ so that $\delta<\min\{r,1\}$ (hence $\bar{B}_H(\theta,\delta)\subset V$)
and $\omega$ in Lemma~\ref{lem:S.2.2} satisfies
\begin{eqnarray}\label{e:MP.0}
\omega(z+\varphi(z))<\frac{1}{2}\min\{a_0,a_1\},\quad\forall z\in B_H(\theta,\delta)\cap H^0.
\end{eqnarray}
By the uniqueness of solutions we can also require that if
$v\in {B}_{H}(\theta,\delta)$ satisfies $(I-P^0)\nabla{\mathcal{L}}(v)=0$ then
$v=z+\varphi(z)$ for some $z\in B_H(\theta,\delta)\cap H^0$.

Take a smooth function $\rho:[0,\infty)\to\mathbb{R}$ satisfying: $0\le\rho\le 1$,
$\rho(t)=1$ for $t\le\delta/2$, $\rho(t)=0$ for $t\ge\delta$, and $|\rho'(t)|<4/\delta$.
For $b\in H^0$ we set $\mathcal{L}_b^\circ(z)=\mathcal{L}^\circ(z)+ \rho(\|z\|)(b,z)_H$. Then
\begin{eqnarray}\label{e:MP.1}
D\mathcal{L}_b^\circ(z)[\xi]&=&D\mathcal{L}(z+\varphi(z))[\xi+\varphi'(z)\xi]+ \rho(\|z\|)(b,\xi)_H\nonumber\\
&&+\rho'(\|z\|)(b, z)_H(z/\|z\|,\xi)_H,\quad\forall\xi\in H^0.
\end{eqnarray}
Note that $\nu:=\inf\{\|D\mathcal{L}^\circ(z)\|\,|\,z\in
\bar{B}_{H^0}(\theta,\delta)\setminus {B}_{H^0}(\theta, {\delta/2})\}>0$.
Suppose $\|b\|<\nu/5$. Then
\begin{eqnarray}\label{e:MP.2}
\|D\mathcal{L}_b^\circ(z)\|&=&\big\|D\mathcal{L}^\circ(z)+ \rho(\|z\|)b+ (b,z)_H\rho'(\|z\|)z/\|z\|\big\|
\ge  \nu- 5\|b\|>0,
\end{eqnarray}
and therefore $\mathcal{L}_b^\circ$ has no critical point in
$\bar{B}_{H^0}(\theta,\delta)\setminus {B}_{H^0}(\theta,{\delta/2})$.
By Sard's theorem we may take arbitrary small $b\ne 0$ such that the critical points of
$\mathcal{L}_b^\circ$, if any, are nondegenerate.
Choose a $C^2$ function $\beta:H\to\mathbb{R}$ such that $\beta(u)=0$ for $u\in H\setminus B_H(\theta,r)$,
and $\beta(u)=1$ for $u\in B_H(\theta,\delta)$. Clearly,
we can require $\sup\{\|\beta(u)\|, \|\beta'(u)\|, \|\beta''(u)\|\,|\,u\in H\}\le M$ for some $M>0$. Define
\begin{equation}\label{e:MP.3}
\tilde{\mathcal{L}}_b(u)={\mathcal{L}}(u)+ \beta(u)\rho(\|P^0u\|)(b,P^0u)_H
\end{equation}
We shall prove that $\tilde{\mathcal{L}}_b$ satisfies the expected requirements  for sufficiently small $b\ne 0$
produced by Sard's theorem above.\\

\noindent{\bf Step 1}.\quad {\it Prove that $\tilde{\mathcal{L}}_b$ satisfies (iv) if $b$ is small enough}.
Since $\mathcal{L}$ satisfies the (PS) condition, $c:=\inf\{
\|D\mathcal{L}(u)\|\,|\,u\in {B}_{H}(\theta, r)\setminus {B}_{H}(\theta,\delta)\}>0$.
Hence all critical points of $\tilde{\mathcal{L}}_b$ belong to ${B}_{H}(\theta,\delta)$
as long as $b$ is small enough.

Let us prove that each critical point $v$ of $\tilde{\mathcal{L}}_b$ in
${B}_{H}(\theta,\delta)$ is  nondegenerate. Obverse that
\begin{eqnarray}\label{e:MP.4}
0=\tilde{\mathcal{L}}'_b(v)[\xi]&=&(\nabla{\mathcal{L}}(v),\xi)_H+ \rho(\|P^0v\|)(b, P^0\xi)_H\nonumber\\
&&+\rho'(\|P^0v\|)(b, P^0v)_H(P^0v, P^0\xi)_H/\|P^0v\|,\quad\forall\xi\in H.
\end{eqnarray}
Since $\rho(\|P^0v\|)=1$ for $\|P^0v\|\le\|v\|<\delta$,
 this implies $(\nabla{\mathcal{L}}(v),\xi)_H=0$ for any $\xi\in H^+\oplus H^-$,
i.e., $(I-P^0)\nabla{\mathcal{L}}(v)=0$. It follows that
$v=z+\varphi(z)$ for some $z\in B_H(\theta,\delta)\cap H^0$. [This $z$ is nonzero. Otherwise,
$v=\theta$. But $\theta$ is not a critical point of $\tilde{\mathcal{L}}_b$ if $b\ne\theta$].
Note that  $(\nabla{\mathcal{L}}(z+\varphi(z)),\varphi'(z)\xi)_H=0\;\forall\xi\in H^0$
because $\varphi'(z)\xi\in H^+\oplus H^-$. (\ref{e:MP.4}) leads to
\begin{eqnarray*}
0&=&(\nabla{\mathcal{L}}(z+\varphi(z)),\xi)_H+ \rho(\|z\|)(b,\xi)_H
+\rho'(\|z\|)(b, z)_H(z, \xi)_H/\|z\|\\
&=&(\nabla{\mathcal{L}}(z+\varphi(z)),\xi)_H+ (\nabla{\mathcal{L}}(z+\varphi(z)),\varphi'(z)\xi)_H\\
&&+\rho(\|z\|)(b,\xi)_H +\rho'(\|z\|)(b, z)_H(z, \xi)_H/\|z\|\quad\forall \xi\in H^0,
\end{eqnarray*}
and therefore $D\mathcal{L}_b^\circ(z)=0$ by (\ref{e:MP.1}). That is, $z$ is a critical point of $\mathcal{L}_b^\circ$,
and so $z\in B_{H^0}(\theta,{\delta/2})$ by  (\ref{e:MP.2}).
 It follows from (\ref{e:MP.3}) that
\begin{eqnarray}\label{e:MP.5}
\tilde{\mathcal{L}}''_b(v)[\xi,\eta]=({\mathcal{L}}''(v)\xi,\eta)_H,\quad\forall\xi,\eta\in H.
\end{eqnarray}
Let $\xi\in{\rm Ker}(\tilde{\mathcal{L}}''_b(v))$.
By (\ref{e:MP.5}), we have
\begin{eqnarray}\label{e:MP.5.1}
{\mathcal{L}}''(v)[\xi,\eta]=({\mathcal{L}}''(z+\varphi(z))[\xi],\eta)_H=0,\quad\forall\eta\in H.
\end{eqnarray}
Decompose $\xi$ into $\xi^0+\xi^\bot$, where $\xi^0\in H^0$ and $\xi^\bot\in H^+\oplus H^-$. A direct computation yields
\begin{eqnarray}\label{e:MP.6}
({\mathcal{L}}''(z+\varphi(z))[\xi^0], \eta+\varphi'(z)[\eta])_H+
({\mathcal{L}}''(z+\varphi(z))[\xi^\bot], \eta+\varphi'(z)[\eta])_H
=0,\;\forall\eta\in H^0.
\end{eqnarray}
Note that $(I-P^0)\nabla{\mathcal{L}}(w+\varphi(w))=0\;\forall w\in B_{H^0}(\theta,\delta)$
by (\ref{e:S.1.2}).
Hence 
$$
(\nabla{\mathcal{L}}(w+\varphi(w)), \zeta)_H=0\;\forall \zeta\in H^+\oplus H^-.
$$
Differentiating this equality with respect to $w$ yields
$$
(\mathcal{L}''(w+\varphi(w))[\tau+\varphi'(w)[\tau]], \zeta)_H=0,\quad \forall\tau\in H^0,\;
\forall w\in B_{H^0}(\theta,\delta),\;\forall\zeta\in H^+\oplus H^-.
$$
In particular, we have $({\mathcal{L}}''(z+\varphi(z))[\xi^\bot], \eta+\varphi'(z)[\eta])_H
=0\;\forall\eta\in H^0$. This and (\ref{e:MP.6}) yield
\begin{eqnarray}\label{e:MP.7}
d^2\mathcal{L}^\circ(z)[\xi^0,\eta]=({\mathcal{L}}''(z+\varphi(z))[\xi^0], \eta+\varphi'(z)[\eta])_H=0,\quad\forall\eta\in H^0.
\end{eqnarray}
Moreover, $d^2\mathcal{L}_b^\circ(z')=d^2\mathcal{L}^\circ(z')\;\forall z'\in B_{H^0}(\theta,\delta/2)$
by the construction of $\mathcal{L}_b^\circ$. We obtain that
$d^2\mathcal{L}_b^\circ(z)[\xi^0,\eta]=0\;\forall\eta\in H^0$.  Since $z$ is a nondegenerate critical point of $\mathcal{L}_b^\circ$ by the choice of $b$,  $\xi^0=\theta$ and thus $\xi=\xi^\bot$. By (\ref{e:MP.5}), (\ref{e:MP.5.1}) and $(\tilde{\mathcal{L}}''_b(v)[\xi],\eta)_H=0\;\forall\eta\in H$, we get
\begin{eqnarray}\label{e:MP.8}
({\mathcal{L}}''(z+\varphi(z))[\xi^\bot],\eta)_H= ({\mathcal{L}}''(z+\varphi(z))[\xi],\eta)_H=0,\quad\forall\eta\in H.
\end{eqnarray}
Hence ${\mathcal{L}}''(z+\varphi(z))[\xi^\bot]=0$.
Decompose $\xi^\bot$ into $\xi^++\xi^-$, where $\xi^+\in H^+$ and $\xi^-\in H^-$.
Then ${\mathcal{L}}''(z+\varphi(z))[\xi^+]=-{\mathcal{L}}''(z+\varphi(z))[\xi^-]$.
By Lemma~\ref{lem:S.2.2} and (\ref{e:MP.0}) we derive
\begin{eqnarray*}
&&a_1\|\xi^+\|^2\le ({\mathcal{L}}''(z+\varphi(z))[\xi^+], \xi^+)_H= -({\mathcal{L}}''(z+\varphi(z))[\xi^-],\xi^+)_H\le
\frac{a_1}{2}\|\xi^+\|\cdot\|\xi^-\|,\\
&&-a_0\|\xi^-\|^2\ge ({\mathcal{L}}''(z+\varphi(z))[\xi^-],\xi^-)_H=-({\mathcal{L}}''(z+\varphi(z))[\xi^+],\xi^-)_H\ge-
\frac{a_0}{2}\|\xi^-\|\cdot\|\xi^+\|.
\end{eqnarray*}
These imply that $\xi^+=\xi^-=\theta$ and so $\xi=\theta$.
Hence $v$ is a nondegenerate critical point of $\tilde{\mathcal{L}}_b$.

Note that Lemma~\ref{lem:S.2.2} and (\ref{e:MP.5}) give rise to
\begin{eqnarray*}
&&\tilde{\mathcal{L}}''_b(v)(\xi,\xi)=({\mathcal{L}}''(v)[\xi],\xi)_H\ge a_1\|\xi\|^2,\quad\forall\xi\in H^+,\\
&&\tilde{\mathcal{L}}''_b(v)(\xi,\xi)=({\mathcal{L}}''(v)[\xi],\xi)_H\le -a_0\|\xi\|^2,\quad\forall\xi\in H^-.
\end{eqnarray*}
But $H=H^+\oplus H^0\oplus H^-$, $\dim H^-=m^-$ and $\dim H^0=n^0$. These show that
the Morse index of $\tilde{\mathcal{L}}''_b(v)$ must sit in $[m^-, m^-+n^0]$. (iv) is proved.\\

\noindent{\bf Step 2}.\quad
{\it Prove that $\tilde{\mathcal{L}}_b$ satisfies Hypothesis~\ref{hyp:MP.1}
on $V$ if $b\ne 0$ is small enough}. By (\ref{e:MP.3}) we have for all $\xi,\eta\in H$,
\begin{eqnarray}\label{e:MP.9}
\tilde{\mathcal{L}}'_b(u)[\xi]&=&{\mathcal{L}}'(u)[\xi]+ (\beta'(u)[\xi])\rho(\|P^0u\|)(b,P^0u)_H+
\beta(u)\rho(\|P^0u\|)(b, P^0\xi)_H\nonumber\\
&&+\beta(u)\rho'(\|P^0u\|)(b, P^0u)_H(P^0u, P^0\xi)_H\cdot\frac{1}{\|P^0u\|}
\end{eqnarray}
and
\begin{eqnarray*}
&&(\tilde{\mathcal{L}}''_b(u)[\eta],\xi)_H=({\mathcal{L}}''(u)[\eta],\xi)_H+ (\beta''(u)[\eta],\xi)_H\rho(\|P^0u\|)(b,P^0u)_H\\
&&+(\beta'(u)[\xi])\rho(\|P^0u\|)(b,P^0\eta)_H+ (\beta'(u)[\xi])(b,P^0u)_H\rho'(\|P^0u\|)(P^0u, P^0\eta)_H\cdot\frac{1}{\|P^0u\|}\\
&&+(\beta'(u)[\eta])\rho(\|P^0u\|)(b, P^0\xi)_H+ \beta(u)(b, P^0\xi)_H\rho'(\|P^0u\|)(P^0u, P^0\eta)_H\cdot\frac{1}{\|P^0u\|}\\
&&+(\beta'(u)[\eta])\rho'(\|P^0u\|)(b, P^0u)_H(P^0u, P^0\xi)_H\cdot\frac{1}{\|P^0u\|}\\
&&+\beta(u)\bigg(\rho''(\|P^0u\|)(P^0u, P^0\eta)_H\cdot\frac{1}{\|P^0u\|}\bigg)(b, P^0u)_H(P^0u, P^0\xi)_H\cdot\frac{1}{\|P^0u\|}\\
&&+\beta(u)\rho'(\|P^0u\|)(b, P^0\eta)_H(P^0u, P^0\xi)_H\cdot\frac{1}{\|P^0u\|}\\
&&+\beta(u)\rho'(\|P^0u\|)(b, P^0u)_H(P^0\eta, P^0\xi)_H\cdot\frac{1}{\|P^0u\|} \\
&&-\beta(u)\rho'(\|P^0u\|)(b, P^0u)_H (P^0u, P^0\xi)_H(P^0u, P^0\eta)_H\cdot\frac{1}{\|P^0u\|^3}\\
&&=({\mathcal{L}}''(u)[\eta],\xi)_H+ \Upsilon(u,b,\xi,\eta).
\end{eqnarray*}
By the constructions of $\beta$ and $\rho$, after the tedious estimate we get a constant $M_2>0$ such that
\begin{eqnarray}\label{e:MP.10}
|\Upsilon(u,b,\xi,\eta)|\le M_2\|b\|\cdot\|\xi\|\cdot\|\eta\|,\quad
\forall u\in V,\;\forall\xi,\eta\in H.
\end{eqnarray}
Since we may require that  the support of $\beta$ can be contained
a neighborhood of $\theta$ on which (iii) of Hypothesis~\ref{hyp:MP.1} holds, for sufficiently small $b\ne 0$
the positive definite part $\tilde P$ of $\tilde{\mathcal{L}}''_b$ given by
 $(\tilde P(u)\xi,\eta)_H=(P(u)\xi,\eta)_H+ \Upsilon(u,b,\xi,\eta)$,
   is also uniformly positive definite on this neighborhood.
Hence  $\tilde{\mathcal{L}}_b$ satisfies Hypothesis~\ref{hyp:MP.1}.\\

\noindent{\bf Step 3}.\quad {\it Prove that (ii) and (iii) can be satisfied if $b\ne 0$ is small}.
Indeed, (\ref{e:MP.10}) implies that $\|\tilde{\mathcal{L}}''_b(u)-{\mathcal{L}}''(u)\|\le M_2\|b\|$
for all $u\in V$. Moreover,  by (\ref{e:MP.3}) and (\ref{e:MP.9}) we have positive numbers $M_i$, $i=0,1$, such that $|\tilde{\mathcal{L}}_b(u)-{\mathcal{L}}(u)|\le M_0\|b\|$ and
$\|\tilde{\mathcal{L}}'_b(u)-{\mathcal{L}}'(u)\|\le M_1\|b\|$ for
all $u\in V$. Hence it suffices to require that $\|b\|<\epsilon/M_i$ for $i=0,1,2$.\\

\noindent{\bf Step 4}.\quad {\it Prove that $\tilde{\mathcal{L}}_b$ satisfies  the (PS) condition for small $b$}.
By (ii) and (iii) in Hypothesis~\ref{hyp:MP.1}, there exists  $\varepsilon\in (0,\delta/2)$ such that
for all $u\in B_H(\theta,\varepsilon)$ and $\xi\in H$,
\begin{eqnarray}\label{e:MP.11}
(P(u)\xi,\xi)_H\ge C_0\|\xi\|^2\quad\hbox{and}\quad \|Q(u)-Q(\theta)\|<C_0/2.
\end{eqnarray}
Recall that $\mathcal{L}$ is bounded in $\bar{B}_H(u_0, r)$ and that $\theta$ is a unique critical point
 of $\mathcal{L}$ in $V$. Since  $\mathcal{L}$ satisfies  the (PS) condition,  we have $\nu_0>0$ such that $\|\mathcal{L}'(u)\|\ge\nu_0$ for
 all $u\in \bar{B}_H(u_0, r)\setminus B_H(\theta,\varepsilon)$.
Choose $b$ so small that
\begin{eqnarray}\label{e:MP.12}
\|\tilde{\mathcal{L}}'_b(u)\|\ge \nu_0/2,\quad\forall u\in \bar{B}_H(u_0, r)\setminus B_H(\theta,\varepsilon).
\end{eqnarray}
Let  $(u_n)\subset V$ satisfy $\tilde{\mathcal{L}}'_b(u_n)\to 0$ and $\sup_n|\tilde{\mathcal{L}}_b(u_n)|<\infty$.
Assume that $(u_n)$ has a subsequence $(u_{n_k})$ sitting in $V\setminus \bar{B}_H(u_0, r)$.
Since $\mathcal{L}$ satisfies  the (PS) condition, by (iii) we deduce that
$(u_{n_k})$ has a converging subsequence. Thus after removing finitely many terms  we may
assume that $(u_n)\subset  \bar{B}_H(u_0, r)$, and by (\ref{e:MP.12})
we may further assume that $(u_n)\subset B_H(\theta,\varepsilon)$. It follows from
(\ref{e:MP.9}) that $\nabla\tilde{\mathcal{L}}_b(u_n)=\nabla{\mathcal{L}}(u_n)+ P^0b$ for all $n$.
 For any two natural numbers $n$ and $m$,
 using the mean value theorem we have $\tau\in (0,1)$ such that
 \begin{eqnarray*}
 &&(\nabla\mathcal{L}(u_n)-\nabla\mathcal{L}(u_m), u_n-u_m)_H=(B(\tau u_n+ (1-\tau)u_m)(u_n-u_m), u_n-u_m)_H\\
&&=(P(\tau u_n+ (1-\tau)u_m)(u_n-u_m), u_n-u_m)_H+ (Q(\theta)(u_n-u_m), u_n-u_m)_H\\
&&+([Q(\tau u_n+ (1-\tau)u_m)-Q(\theta)](u_n-u_m), u_n-u_m)_H\\
&&\ge C_0\|u_n-u_m\|^2-\frac{C_0}{2}\|u_n-u_m\|^2+(Q(\theta)(u_n-u_m), u_n-u_m)_H,
\end{eqnarray*}
where the last inequality comes from (\ref{e:MP.11}).
Passing to a subsequence we may assume $u_n\rightharpoonup u_0$. Since $Q(\theta)$ is compact,
$Q(\theta)u_n\to Q(\theta)u_0$ and so $(Q(\theta)(u_n-u_m), u_n-u_m)_H\to 0$ as $n,m\to\infty$.
Note that $\nabla\mathcal{L}(u_n)-\nabla\mathcal{L}(u_m)=
(\nabla\mathcal{L}(u_n)+ P^0b)-(\nabla\mathcal{L}(u_m)+ P^0b)\to 0$ as $n,m\to\infty$.  From
the above inequality we conclude that $\|u_n-u_m\|\to 0$ as $n, m\to\infty$. This implies
$u_n\to u_0$.  Theorem~\ref{th:MP.2} is proved.
\end{proof}

\section{Applications to quasi-linear elliptic systems of higher order}\label{sec:Funct}
\setcounter{equation}{0}

\subsection{Fundamental analytic properties for functionals  $\mathfrak{F}$}\label{sec:Funct.1}

For $p\in [2,\infty)$ and integers $m\ge 1$, $n\ge 2$,
 a bounded domain $\Omega$ in $\R^n$ is said to be a {\it Sobolev domain} for $(p,m,n)$
 if  the Sobolev embeddings theorems for the spaces $W^{m, p}(\Omega)$ hold.
  The following two theorems summarize fundamental analytic properties of the functional  $\mathfrak{F}$.

\begin{theorem}\label{th:4.1}
 Given $p\in [2,\infty)$ and integers $m, N\ge 1$, $n\ge 2$, let $\Omega\subset\R^n$ be a Sobolev domain
 for $(p,m,n)$, and let $V_0$ be a closed subspace of $W^{m,p}(\Omega, \mathbb{R}^N)$ and $V=\vec{w}+V_0$
 for some $\vec{w}\in W^{m,p}(\Omega, \mathbb{R}^N)$.
Suppose that (i)-(ii) in \textsf{Hypothesis} $\mathfrak{F}_{p,N,m,n}$ hold.
 Then we have\\
     {\bf A)}. The restriction of  the functional $\mathfrak{F}$
in (\ref{e:1.3}) to $V$, $\mathfrak{F}_V$, is bounded on any bounded subset, of class $C^1$, and the derivative
$\mathfrak{F}'_V(\vec{u})$ of it at $\vec{u}$ is given by
\begin{equation}\label{e:4.1}
\langle \mathfrak{F}'_V(\vec{u}), \vec{v}\rangle=\sum^N_{i=1}\sum_{|\alpha|\le m}\int_\Omega F^i_\alpha(x,
\vec{u}(x),\cdots, D^m \vec{u}(x))D^\alpha v^i dx,\quad\forall \vec{v}\in V_0.
\end{equation}
Moreover, the map $V\ni\vec{u}\to \mathfrak{F}'_V(\vec{u})\in V_0^\ast$ also maps bounded subset into bounded ones.\\
{\bf B)}. The map $\mathfrak{F}'_V$  is of class $C^1$ on $V$ if $p>2$,  G\^ateaux differentiable on $V$ if $p=2$,
 and for each $\vec{u}\in V$ the derivative
$D\mathfrak{F}'_V(\vec{u})\in \mathscr{L}(V_0, V^\ast_0)$ is given by
  \begin{equation}\label{e:4.2}
   \langle D\mathfrak{F}'_V(\vec{u})[\vec{v}],\vec{\varphi}\rangle=\sum^N_{i,j=1}
   \sum_{\scriptsize{\begin{array}{ll}
   &|\alpha|\le m,\\
   &|\beta|\le m
   \end{array}}}\int_\Omega
  F^{ij}_{\alpha\beta}(x, \vec{u}(x),\cdots, D^m \vec{u}(x))D^\beta v^j\cdot D^\alpha\varphi^i dx.
    \end{equation}
(In the case $p=2$,  equivalently,   the gradient map of $\mathfrak{F}_V$,
$V\ni \vec{u}\mapsto\nabla \mathfrak{F}_V(\vec{u})\in V_0$,  given by
\begin{equation}\label{e:4.3}
(\nabla\mathfrak{F}_V(\vec{u}), \vec{v})_{m,2}=\langle \mathfrak{F}'_V(\vec{u}), \vec{v}\rangle\quad
 \forall \vec{v}\in V_0,
\end{equation}
has a G\^ateaux derivative $D(\nabla \mathfrak{F}_V)(\vec{u})\in\mathscr{L}_s(V_0)$ at every $\vec{u}\in V$.)
Moreover,   $D\mathfrak{F}'_V$ also satisfies the following properties:
\begin{description}
\item[(i)] For every given $R>0$, $\{D\mathfrak{F}'_V(\vec{u})\,|\, \|\vec{u}\|_{m,p}\le R\}$
is bounded in $\mathscr{L}_s(V_0)$.
Consequently, when $p=2$, $\mathfrak{F}_V$ is  of class $C^{2-0}$.
\item[(ii)] For any $\vec{v}\in V_0$, $\vec{u}_k\to
\vec{u}_0$ implies
$D\mathfrak{F}'_V(\vec{u}_k)[\vec{v}]\to D\mathfrak{F}'_V(\vec{u}_0)[\vec{v}]$ in $V^\ast_0$.
\item[(iii)] If $p=2$ and $F(x,\xi)$ is independent of all variables $\xi^k_\alpha$, $|\alpha|=m$,
$k=1,\cdots,N$, then $V\ni \vec{u}\mapsto D\mathfrak{F}'_V(\vec{u})\in\mathscr{L}(V_0, V^\ast_0)$ is  continuous,
(i.e., $\mathfrak{F}_V$ is of class $C^2$),  and  $D(\nabla\mathfrak{F}_V)(\vec{u}): V_0\to V_0$ is completely continuous  linear operator for each $\vec{u}\in V$.
\end{description}
\end{theorem}

\begin{theorem}\label{th:4.2}
Under assumptions of Theorem~\ref{th:4.1},  suppose that (iii) in
\textsf{Hypothesis} $\mathfrak{F}_{p,N,m,n}$ is also satisfied. Then\\
{\bf C)}.  $\mathfrak{F}': W^{m,p}(\Omega, \mathbb{R}^N)\to (W^{m,p}(\Omega, \mathbb{R}^N))^\ast$ is of class  $(S)_+$.\\
{\bf D)}. Suppose  $p=2$.   For $u\in V$, let $D(\nabla\mathfrak{F}_V)(\vec{u})$,
$P(\vec{u})$ and $Q(\vec{u})$ be
  operators in $\mathscr{L}(V_0)$ defined by
  \begin{eqnarray*}
  (D(\nabla\mathfrak{F}_V)(\vec{u})[\vec{v}],\vec{\varphi})_{m,2}&=&\sum^N_{i,j=1}
  \sum_{\scriptsize\begin{array}{ll}
  &|\alpha|\le m,\\
  &|\beta|\le m
  \end{array}}\int_\Omega
  F^{ij}_{\alpha\beta}(x, \vec{u}(x),\cdots, D^m \vec{u}(x))D^\beta v^j\cdot D^\alpha\varphi^i dx,\\
   (P(\vec{u})\vec{v}, \vec{\varphi})_{m,2}&=&\sum^N_{i,j=1}\sum_{|\alpha|=|\beta|=m}\int_\Omega
  F^{ij}_{\alpha\beta}(x, \vec{u}(x),\cdots, D^m \vec{u}(x))D^\beta v^j\cdot D^\alpha\varphi^i dx\\
  &&+ \sum^N_{i=1}\sum_{|\alpha|\le m-1}\int_\Omega  D^\alpha v^i\cdot D^\alpha\varphi^i dx,\\
   (Q(\vec{u})\vec{v},\vec{\varphi})_{m,2}&=&\sum^N_{i,j=1}
   \sum_{\scriptsize\begin{array}{ll}
   &|\alpha|\le m, |\beta|\le m,\\
   &|\alpha|+|\beta|<2m
   \end{array}}\int_\Omega
  F^{ij}_{\alpha\beta}(x, \vec{u}(x),\cdots, D^m \vec{u}(x))D^\beta v^j\cdot D^\alpha\varphi^i dx\\
  &&-\sum^N_{i=1}\sum_{|\alpha|\le m-1}\int_\Omega  D^\alpha v^i\cdot D^\alpha\varphi^i dx,
    \end{eqnarray*}
  respectively. (If $V\subset W^{m,p}_0(\Omega, \mathbb{R}^N)$,
  the final terms in the definitions of $P$ and $Q$ may be deleted.)
    Then $D(\nabla\mathfrak{F}_V)=P+ Q$,  and
 \begin{description}
\item[(i)]  for any $\vec{v}\in V_0$, the map $V\ni \vec{u}\mapsto P(\vec{u})\vec{v}\in V_0$ is continuous;
\item[(ii)] for every given $R>0$ there exist positive constants $C(R, n, m, \Omega)$ such that
$$
(P(\vec{u})\vec{v},\vec{v})_{m,2}\ge C\|\vec{v}\|^2_{m,2},\quad\forall \vec{v}\in V_0,\;
\forall\vec{u}\in V\;\hbox{with}\;\|\vec{u}\|_{m,2}\le R;
$$
\item[(iii)] $V\ni \vec{u}\mapsto Q(\vec{u})\in\mathscr{L}(V_0)$ is continuous,
and  $Q(\vec{u})$ is completely continuous
for each $\vec{u}$;
\item[(iv)] for every given $R>0$ there exist positive constants $C_j(R, n, m, \Omega), j=1,2$ such that
$$
(D(\nabla\mathfrak{F}_V)(\vec{u})[\vec{v}],\vec{v})_{m,2}\ge C_1\|\vec{v}\|^2_{m,2}-C_2\|\vec{v}\|^2_{m-1,2},\quad\forall \vec{v}\in V_0,\;\forall\vec{u}\in V\;\hbox{with}\;\|\vec{u}\|_{m,2}\le R.
$$
\end{description}
\end{theorem}

The proofs of these two theorems are not difficult, but cumbersome.
The main tools are the Sobolev embedding theorems and Krasnoselski theorem
concerning the continuity of the Nemytski operator (cf. \cite[Theorem I.2.1]{Kra} and
\cite[Proposition 1.1, page 3]{Skr3}).
When $N=1$ and $V=V_0=W^{m,p}_0(\Omega)$, most of them were stated (or roughly proved), see
 the auxiliary theorem 16  in \cite[Chap.3, Sec.3.4]{Skr1} or Lemma~3.2 on the page 112 of \cite{Skr3}.
 When $N=1$ and $V=V_0=W^{m,p}(\Omega)$, a full proof was given in  \cite{Lu7};
 it is obvious that this implies general case.
For the case $N>1$,  proofs of Theorems~\ref{th:4.1},~\ref{th:4.2} can be
completed by non-essentially changing that of \cite[Theorem~3.1]{Lu7}, i.e.,
 only using the following proposition (easily verified as in the proof of \cite[Prop.4.3]{Lu7})
  adding or estimating more terms in each step. We omit them.

\begin{proposition}\label{prop:4.3}
For the function $\mathfrak{g}_1$ in \textsf{Hypothesis}
$\mathfrak{F}_{p,N,m,n}$,
let continuous positive nondecreasing functions $\mathfrak{g}_k:[0,\infty)\to\mathbb{R}$, $k=3,4,5$, be given by
\begin{eqnarray*}
&&\mathfrak{g}_3(t):=1+ \mathfrak{g}_1(t)[t^2 M(m)N+ t(M(m)N+1)^2]+\mathfrak{g}_1(t)t(M(m)N+1)\\
&&\hspace{40mm}+\mathfrak{g}_1(t)(M(m)N+1)^2,\\
&&\mathfrak{g}_4(t):=\mathfrak{g}_1(t)t+\mathfrak{g}_1(t)\qquad\hbox{and}\qquad
\mathfrak{g}_5(t):=(M(m)N+1)\mathfrak{g}_1(t)(t+1).
\end{eqnarray*}
Then (ii) in \textsf{Hypothesis} $\mathfrak{F}_{p,N,m,n}$ implies that for all $(x,\xi)$,
\begin{eqnarray}\label{e:4.4}
|F(x,\xi)|&\le& |F(x,0)|+
(\sum^N_{k=1}|\xi^k_\circ|)\sum^N_{i=1}\sum_{|\alpha|<m-n/p}|F^i_{\alpha}(x, 0)|+
\sum^N_{i=1}\sum_{m-n/p\le|\alpha|\le m}|F^i_{\alpha}(x, 0)|^{q_\alpha}
\nonumber\\
&&+ \mathfrak{g}_3(\sum^N_{k=1}|\xi^k_\circ|)\Bigg(1+\sum^N_{k=1}\sum_{m-n/p\le|\alpha|\le m}|\xi^k_\alpha|^{p_\alpha}\Bigg),
\end{eqnarray}
\begin{eqnarray}\label{e:4.5}
|F^k_\alpha(x,\xi)|&\le&|F^k_\alpha(x,0)|
+ \mathfrak{g}_4(\sum^N_{i=1}|\xi^i_\circ|)\sum_{|\beta|<m-n/p}\Bigg(1+
\sum^N_{i=1}\sum_{m-n/p\le |\gamma|\le
m}|\xi^i_\gamma|^{p_\gamma }\Bigg)^{p_{\alpha\beta}}\nonumber\\
&+&\mathfrak{g}_4(\sum^N_{i=1}|\xi^i_\circ|)\sum_{m-n/p\le |\beta|\le m} \Bigg(1+
\sum^N_{i=1}\sum_{m-n/p\le |\gamma|\le
m}|\xi^i_\gamma|^{p_\gamma }\Bigg)^{p_{\alpha\beta}}\sum^N_{j=1}|\xi^j_\beta|;
\end{eqnarray}
for the latter we further have
\begin{eqnarray}\label{e:4.6}
|F^k_\alpha(x,\xi)|&\le&|F^k_\alpha(x,0)|
+ \mathfrak{g}_5(\sum^N_{i=1}|\xi^i_\circ|)\Bigg(1+
\sum^N_{i=1}\sum_{m-n/p\le |\gamma|\le m}|\xi^i_\gamma|^{p_\gamma }\Bigg),
\end{eqnarray}
if $|\alpha|<m-n/p$,  and
\begin{eqnarray}\label{e:4.7}
|F^k_\alpha(x,\xi)|&\le&|F^k_\alpha(x,0)|+\mathfrak{g}_5(\sum^N_{i=1}|\xi^i_\circ|)
+ \mathfrak{g}_5(\sum^N_{i=1}|\xi^i_\circ|)\Bigg(\sum^N_{i=1}\sum_{m-n/p\le |\gamma|\le
m}|\xi^i_\gamma|^{p_\gamma }\Bigg)^{1/q_{\alpha}}
\end{eqnarray}
if $m-n/p\le|\alpha|\le m$.
\end{proposition}

As a direct consequence of Theorems~\ref{th:4.1},~\ref{th:4.2} we have:

\begin{corollary}\label{cor:4.4}
Let  $N, m\ge 1$, $n\ge 2$ be integers and let $\Omega\subset\mathbb{R}^n$  a bounded Sobolev domain.
Under Hypothesis~$\mathfrak{F}_{2,N,m,n}$, the restriction of
 the functional  $\mathfrak{F}$ in (\ref{e:1.3}) with $V=W^{m,2}(\Omega, \mathbb{R}^N)$
 to any  closed subspace $H$ of $W^{m,2}(\Omega, \mathbb{R}^N)$
 satisfies  Hypothesis~\ref{hyp:1.1}  with $X=H$.
\end{corollary}

\begin{remark}\label{rem:4.4}
{\rm Theorems~\ref{th:4.1},\ref{th:4.2} have also more general versions in the setting of
\cite{PaSm, Sma, Pa2}.  Let $M$ be a $n$-dimensional compact $C^\infty$ manifold with a strictly positive smooth
measure $\mu$, and possibly with boundary, and $\pi:E\to M$ a real finite dimensional $C^\infty$ vector space
bundle over $M$ of rank $N$. A $m^{\rm th}$ order Lagrangian $L$ on $E$ is said to satisfy \textsf{Hypothesis} $\mathfrak{F}_{p,N,m,n}$
if it has a representation satisfying \textsf{Hypothesis} $\mathfrak{F}_{p,N,m,n}$ under any local
trivialization of $E$. (As usual, $W^{m,p}(M, E)$ is identified with  $W^{m,p}(M, \mathbb{R}^N)$
if $E$ is a trivial bundle $M\times\mathbb{R}^N\to M$.)
The integral functional of such a  Lagrangian  on $W^{m,p}(M, E)$ possess corresponding
conclusions as in Theorems~\ref{th:4.1},\ref{th:4.2}; the full detail will be given at other places. In particular,
if $M$ is the torus $\mathbb{T}^n=\mathbb{R}^n/\mathbb{Z}^n$,
 a $m^{\rm th}$ order Lagrangian  on $M\times\mathbb{R}^N$ fulfilling  \textsf{Hypothesis} $\mathfrak{F}_{p,N,m,n}$
is understood as a function $F:\mathbb{R}^n\times\prod^m_{k=0}\mathbb{R}^{N\times M_0(k)}\to\mathbb{R}$,
which is not only $1$-periodic in each variable $x_i$, $i=1,\cdots,n$, but also
satisfies \textsf{Hypothesis} $\mathfrak{F}_{p,N,m,n}$  with $\overline\Omega=[0,1]^n$.
Then Theorems~\ref{th:4.1},\ref{th:4.2} also hold if $W^{m,p}(\Omega, \mathbb{R}^N)$ is replaced
by $W^{m,p}(\mathbb{T}^n, \mathbb{R}^N)$.}
\end{remark}

\subsection{(PS)- and (C)-conditions}\label{sec:PS}

A $C^1$ functional $\varphi$ on a Banach $X$  is said to satisfy
{$(PS)_c$-condition} (resp. {\it $(C)_c$-condition}) at the level $c\in\mathbb{R}$
if every sequence $(x_j)\subset X$ such that
$\varphi(x_j)\to c\in\mathbb{R}$ and $\varphi'(x_j)\to 0$ (resp. $(1+\|x_j\|)\varphi'(x_j)\to 0$) in $X^\ast$
has a convergent subsequence in $X$. When $\varphi$ satisfies the $(PS)_c$-condition
(resp. $(C)_c$-condition) at every level $c\in\mathbb{R}$ we say that it satisfies the
$(PS)$-{\it condition} (resp. $(C)$-{\it condition}).
For a $C^1$ functional $\varphi$ on a Banach space $X$, which is bounded below,
it was further proved in \cite[Proposition~5.23]{MoMoPa}
 that $\varphi$ satisfies the $(PS)$-condition if and only if it does the $(C)$-condition.
If $\varphi\in C^1(X,\mathbb{R})$ is bounded below and satisfies the $(PS)$-condition, then it
 is coercive \cite{CaLiWi}. Conversely, Proposition~3 in \cite[Chap.4, \S5]{AuEk} claimed that
 any G\^ateaux differentiable, convex, lower semicontinuous coercive functional $\varphi$ on a
 reflexive Banach space X satisfies condition (weak C), that is,
for any sequence $(x_n)\subset X$ such that $\sup|\varphi(x_n)|<\infty$ and $(\varphi'(x_n))\subset X^\ast\setminus\{0\}$
and $\varphi'(x_n)\to 0$ in $X^\ast$, where $X^\ast$ is the dual space of $E$, there is some point $\bar{x}\in X$
such that $\varphi'(\bar{x})=0$ and $\lim\inf\varphi(x_n)\le\varphi(\bar{x})\le\lim\sup\varphi(x_n)$.
 For $\mathfrak{F}_V$ we have a similar result.

\begin{theorem}\label{th:4.5}
Let $\Omega\subset\R^n$, $N\in\mathbb{N}$, $p\in [2,\infty)$ and $V\subset W^{m,p}(\Omega,\mathbb{R}^N)$ be as in Theorem~\ref{th:4.1}.
  Suppose that  \textsf{Hypothesis} $\mathfrak{F}_{p,N,m,n}$ hold and that
   $\mathfrak{F}_V$ is coercive.  Then $\mathfrak{F}_V$ satisfies the {\rm (PS)}- and {\rm (C)}-conditions on $V$.
In particular, the {\rm controllable growth conditions} (see Appendix A)  imply that
$\mathfrak{F}$ is coercive on any closed affine subspace of  $W^{1,2}(\Omega,\mathbb{R}^N)$.
\end{theorem}

\begin{proof}[\bf Proof]
 Since the coercivity of $\mathfrak{F}$ implies that it is bounded below,
 by \cite[Proposition~5.23]{MoMoPa} it suffices to prove that
$\mathfrak{F}$ satisfies the (PS)-condition.
Let $(\vec{u}_j)\subset V$ satisfy
$\mathfrak{F}_V(\vec{u}_j)\to c\in\mathbb{R}$ and $\mathfrak{F}'_V(\vec{u}_j)\to 0$.
Since $\mathfrak{F}$ is coercive,  $(\vec{u}_j)$ is bounded.
Note that $\vec{u}_j=\vec{w}+\vec{v}_j$, $\vec{v}_j\in V_0$ and that $V_0$ is a Hilbert subspace.
 After passing to a subsequence we may assume $\vec{v}_j\rightharpoonup \vec{v}$ in $V_0$. Moreover, $\mathfrak{F}'_V(\vec{u}_j)\to 0$ implies
 $$
 \overline{\lim}_{j\to\infty}\langle\mathfrak{F}'(\vec{u}_j), \vec{u}_j-\vec{u}\rangle=
 \overline{\lim}_{j\to\infty}\langle\mathfrak{F}'(\vec{u}_j), \vec{v}_j-\vec{v}\rangle=
 \overline{\lim}_{j\to\infty}\langle\mathfrak{F}'_V(\vec{u}_j), \vec{v}_j-\vec{v}\rangle=0.
 $$
By {\bf C)} of Theorem~\ref{th:4.2},  $\mathfrak{F}'$ is of class $(S)_+$, and hence $\vec{u}_j\to \vec{u}$ in $V$.
The final claim is obvious.
\end{proof}

There exist some explicit conditions on $F$ under which  $\mathfrak{F}$ is coercive on $W^{m,p}_0(\Omega,\mathbb{R}^N)$,
for example,  there exist  some two positive constants $c_0, c_1$ such that
$F(x,\xi)\ge c_0\sum^N_{i=1}\sum_{|\alpha|=m}|\xi^i_\alpha|^p-c_1$ for all $(x,\xi)$.
  The coercivity requirement is too strong. In fact, the proof of Theorem~\ref{th:4.5}
shows that under  \textsf{Hypothesis} $\mathfrak{F}_{p,N,m,n}$ we only need to add some conditions
so that
$$
\sup_j|\mathfrak{F}(\vec{u}_j)|<\infty\quad\hbox{and}\quad \mathfrak{F}'(\vec{u}_j)\to 0\;\Longrightarrow\;
\sup_j\|\vec{u}_j\|_{m,p}<\infty.
$$
For example, the following two results are easily verified, see \cite{Lu7} for full proofs.

\begin{theorem}\label{th:4.6}
 Let $N\in\mathbb{N}$, $p\in [2,\infty)$ and $\Omega\subset\R^n$ be a Sobolev domain
 for $(p,m,n)$.  Then $\mathfrak{F}$ satisfies the {\rm (PS)}- and {\rm (C)}-conditions
 on $W^{m,p}_0(\Omega,\mathbb{R}^N)$  provided that \textsf{Hypothesis} $\mathfrak{F}_{p,N,m,n}$
 is satisfied and that there exist $\kappa\in\mathbb{R}$  and $\Upsilon\in L^1(\Omega)$ such that
  $$
  F(x,\xi)-\kappa\sum^N_{i=1}\sum_{|\alpha|\le m}F^i_\alpha(x,\xi)\xi^i_\alpha\ge c_0\sum^N_{i=1}\sum_{|\alpha|=m}|\xi^i_\alpha|^p-c_1\sum^N_{i=1}|\xi^i_{\bf 0}|^p
  -\Upsilon(x)  \quad\forall (x,\xi),
  $$
 where  $c_0>0$ and $c_0-c_1S_{m,p}>0$ for the best constant $S_{m,p}>0$ with
 $$
 \int_\Omega |{u}|^p dx\le S_{m,p}\int_\Omega|D^m{u}|^p dx=S_{m,p}\sum_{|\alpha|=m}
 \int_\Omega|D^\alpha{u}|^p\quad\forall {u}\in W^{m,p}_0(\Omega,\mathbb{R}^N).
 $$
 \end{theorem}

 \begin{theorem}\label{th:4.7}
 Let $\Omega\subset\R^n$ be  a  Sobolev domain for $(2,m,n)$.
 Suppose  that  \textsf{Hypothesis} $\mathfrak{F}_{2,N,m,n}$ is satisfied with the constant function $\mathfrak{g}_2$, and that
   $$
  F(x,\hat\xi, {\bf 0})\le \varphi(x)+ C\sum^N_{i=1}\sum_{|\alpha|\le m-1}|\xi^i_\alpha|^{r},
  \quad\forall (x, \hat\xi)\in\overline\Omega\times\prod^{m-1}_{k=0}\mathbb{R}^{N\times M_0(k)},
  $$
 where  $\varphi\in L^1(\Omega)$ and $1\le r<2$. Then $\mathfrak{F}$ satisfies the {\rm (PS)}- and {\rm (C)}-conditions  on $W^{m,2}_0(\Omega,\mathbb{R}^N)$.
\end{theorem}

When $m=1$ and $F$ does not depend on $x, \hat\xi$,
more characterizations of coercivity for $\mathfrak{F}$ can be found in
\cite{ChKr} and references therein.

\subsection{Morse inequalities and corollaries}\label{sec:Morse}

Firstly, we show that Corollary~\ref{cor:4.4} and Theorems~\ref{th:S.6.0},~\ref{th:S.6.1}
(taking $\vec{\lambda}={\bf 0}$) imply:

\begin{theorem}\label{th:Morse.1}
Let $\Omega\subset\R^n$ be  a  Sobolev domain for $(2,m,n)$, $N\in\mathbb{N}$, and $H$  a closed subspace of $W^{m,2}(\Omega,\mathbb{R}^N)$, $\mathcal{H}=\vec{\omega}+ H$ for some $\vec{\omega}\in W^{m,2}(\Omega,\mathbb{R}^N)$. Let $G$ be a compact Lie group which acts on $\mathcal{H}$ in a $C^3$-smooth
isometric way.  Suppose that \textsf{Hypothesis} $\mathfrak{F}_{2,N,m,n}$ is satisfied and that
the restriction functional $\mathfrak{F}_{\mathcal{H}}:=\mathfrak{F}|_{\mathcal{H}}$  is $G$-invariant, where $\mathfrak{F}$ is given by (\ref{e:1.3}).
Let $\mathcal{O}$ be an isolated critical orbit of $\mathfrak{F}_{\mathcal{H}}$ and also a compact $C^3$ submanifold. Its normal bundle $N\mathcal{O}$ has fiber at $\vec{u}\in\mathcal{O}$,
$N\mathcal{O}_{\vec{u}}=\{\vec{v}\in H\,|\, (\vec{v}, \vec{w})_{m,2}=0\;\forall\vec{w}\in T_{\vec{u}}\mathcal{O}\subset H\}$.
Let $N^+\mathcal{O}_{\vec{u}},  N^0\mathcal{O}_{\vec{u}}$ and $N^-\mathcal{O}_{\vec{u}}$ be the positive definite, null and negative definite spaces of the bounded linear self-adjoint operator associated with the bilinear form
   $$
  N\mathcal{O}_{\vec{u}}\times N\mathcal{O}_{\vec{u}}\ni (\vec{v}, \vec{w})\mapsto \sum^N_{i=1}
  \sum_{\scriptsize\begin{array}{ll}
  &|\alpha|\le m,\\
  &|\beta|\le m
  \end{array}}\int_\Omega
  F^{ij}_{\alpha\beta}(x, \vec{u}(x),\cdots, D^m \vec{u}(x))D^\beta v^j\cdot D^\alpha w^i dx.
  $$
Then $\dim N^0\mathcal{O}_{\vec{u}}$ and $\dim N^-\mathcal{O}_{\vec{u}}$ are finite and independent of choice
of $\vec{u}\in\mathcal{O}$. They are called {\rm nullity} and {\rm Morse index} of $\mathcal{O}$, denoted by $\nu_{\mathcal{O}}$
and $\mu_{\mathcal{O}}$, respectively. Moreover, the following holds.
\begin{description}
\item[(i)] If $\nu_{\mathcal{O}}=0$ (i.e., the critical orbit $\mathcal{O}$ is nondegenerate),
there exist  $\epsilon>0$
and a  $G$-equivariant homeomorphism  onto an open neighborhood of
the zero section preserving fibers
$\Phi:  N^+\mathcal{O}(\epsilon)\oplus N^-\mathcal{ O}(\epsilon)\to N\mathcal{ O}$
such that for any $\vec{u}\in\mathcal{O}$ and  $(\vec{v}_+, \vec{v}_-)\in
N^+\mathcal{ O}(\epsilon)_{\vec{u}}\times N^-\mathcal{O}(\epsilon)_{\vec{u}}$,
\begin{eqnarray}\label{e:4.9}
\mathfrak{F}_{\mathcal{H}}\circ E\circ\Phi(\vec{u},  \vec{v}_++
\vec{v}_-)=\|\vec{v}_+\|^2_{m,2}-\|\vec{v}_-\|^2_{m,2}+ \mathfrak{F}|_{\mathcal{O}},
\end{eqnarray}
where $E:N\mathcal{ O}\to \mathcal{H}$ is given by $E(\vec{u},\vec{v})=\vec{u}+\vec{v}$.
\item[(ii)] If $\nu_{\mathcal{O}}\ne 0$ there exist $\epsilon>0$,
a  $G$-equivariant topological  bundle
morphism that preserves the zero section,
 $\mathfrak{h}:N^0\mathcal{O}(3\epsilon)\to N^+\mathcal{O}\oplus N^-\mathcal{O}\subset \mathcal{H}\times H,\;(\vec{u},\vec{v})\mapsto \mathfrak{h}_{\vec{u}}(\vec{v})$,
and a  $G$-equivariant homeomorphism onto an open neighborhood of
the zero section preserving fibers,
$\Phi: N^0\mathcal{ O}(\epsilon)\oplus N^+\mathcal{
O}(\epsilon)\oplus N^-\mathcal{ O}(\epsilon)\to N\mathcal{O}$,
such that  the following properties hold:\\
\noindent{\bf (ii.1)} for any $\vec{u}\in\mathcal{O}$ and  $(\vec{v}_0, \vec{v}_+, \vec{v}_-)\in N^0\mathcal{O}(\epsilon)_{\vec{u}}\times N^+\mathcal{ O}(\epsilon)_{\vec{u}}\times N^-\mathcal{
O}(\epsilon)_{\vec{u}}$,
\begin{eqnarray*}
\mathfrak{F}_{\mathcal{H}}\circ E\circ\Phi(\vec{u}, \vec{v}_0, \vec{v}_++
\vec{v}_-)=\|\vec{v}_+\|^2_{m,2}-\|\vec{v}_-\|^2_{m,2}+ \mathfrak{F}(\vec{u}+\vec{v}_0+
\mathfrak{h}_{\vec{u}}(\vec{v}_0));
\end{eqnarray*}
\noindent{\bf (ii.2)} for each $\vec{u}\in\mathcal{O}$ the function
$N^0\mathcal{O}(\epsilon)_{\vec{u}}\to\R,\;\vec{v}\mapsto \mathfrak{F}_{\vec{u}}^\circ(\vec{v}):=
\mathfrak{F}(\vec{u}+\vec{v}+ \mathfrak{h}_{\vec{u}}(\vec{v}))$
is $G_{\vec{u}}$-invariant, of class $C^{1}$,  and satisfies:
$D\mathfrak{F}_{\vec{u}}^\circ(\vec{v})[\vec{w}]=
D\mathfrak{F}(\vec{u}+\vec{v}+ \mathfrak{h}_{\vec{u}}(\vec{v}))[\vec{w}],\;\forall \vec{w}\in N^0\mathcal{O}_{\vec{u}}$.
\end{description}
\end{theorem}

\begin{proof}[\bf Proof]
Since $T\mathcal{H}=\mathcal{H}\times H$, the exponential map
$\exp:T\mathcal{H}\to \mathcal{H}$ (with respect to the Riemannian-Hilbert
structure on $\mathcal{H}$ induced by the inner product $(\cdot,\cdot)_{m,2}$) is
given by $\exp(\vec{u},\vec{v})=\vec{u}+\vec{v}$ for $(\vec{u},\vec{v})\in \mathcal{H}\times H$.
Let $\mathfrak{F}_{N\mathcal{O}_{\vec{u}}}$ be the restriction of
$(\mathfrak{F}_{\mathcal{H}})\circ\exp$ to the fiber of $N\mathcal{O}$ at $\vec{u}\in\mathcal{O}$.
Then $\mathfrak{F}_{N\mathcal{O}_{\vec{u}}}(\vec{v})=\mathfrak{F}(\vec{u}+\vec{v})$ for
$\vec{v}\in N\mathcal{O}_{\vec{u}}$. It follows from Corollary~\ref{cor:4.4}
that $\mathfrak{F}_{N\mathcal{O}_{\vec{u}}}$ satisfies Hypothesis~\ref{hyp:1.1} with $X=N\mathcal{O}_{\vec{u}}$
around the origin of $N\mathcal{O}_{\vec{u}}$.
 Theorems~\ref{th:S.6.0},~\ref{th:S.6.1} lead to the desired conclusions immediately.
\end{proof}

If $n=2$, $\partial\Omega$ is smooth, and either $F$ is analytic, or $m=1$ and $F$ is
suitable smooth, then the Morse indexes of critical points of $\mathfrak{F}$ on  $W^{m,2}_0(\Omega,\mathbb{R}^N)$ can be computed by  Uhlenbeck's generalizations \cite[Theorem~3.5]{Uhl73}
for Smale's Morse index theorem \cite{Sma1}.
The case of Neumann type boundary conditions may still be considered by  Dalbono and Portaluri \cite{DaPo}.

Write $\mathfrak{F}_{\mathcal{H},d}=\{x\in\mathcal{H}\,|\,\mathfrak{F}(x)\le d\}$ for $d\in\mathbb{R}$.
Using Corollary~\ref{cor:S.6.5}, the standard arguments (see \cite[Chapter I, Theorem~7.6]{Ch},
 \cite[Chapter 10]{MaWi} and \cite[Corollary~6.5.10]{Ber}) yield

\begin{theorem}\label{th:Morse.2}
Under the assumptions of Theorem~\ref{th:Morse.1},
Let  $a<b$ be two regular values of $\mathfrak{F}_{\mathcal{H}}$ and
  $\mathfrak{F}_{\mathcal{H}}^{-1}([a,b])$ contains only nondegenerate  critical orbits $\mathcal{O}_j$
  with Morse indexes $\mu_j$, $j=1,\cdots,k$.
 Suppose that  $\mathfrak{F}_{\mathcal{H}}$ satisfies the $(PS)_c$ condition
for each $c\in [a,b)$. (This is true if  either  $\mathfrak{F}_{\mathcal{H}}$ is coercive or  one of Theorems~\ref{th:4.5},~\ref{th:4.6}  holds in case ${\mathcal{H}}=W^{m,2}_0(\Omega,\mathbb{R}^N)$.)
Then
\begin{eqnarray}\label{e:Morse.7}
\sum^k_{j=1}\dim H_{q-\mu_{{\cal O}_j}}(\mathcal{O}_j;\mathbb{Z}_2)
=\dim H_q(\mathfrak{F}_{\mathcal{H},b}, \mathfrak{F}_{\mathcal{H},a};\mathbb{Z}_2),\quad \forall q\in\mathbb{N}_0,
\end{eqnarray}
and there exists a polynomial with nonnegative
integral coefficients $Q(t)$ such that
\begin{eqnarray}\label{e:Morse.8}
\sum^\infty_{i=0}\sum^k_{j=1}{\rm rank}H^i_G(\mathcal{O}_j, \theta^-_j\otimes{\bf K})t^{\mu_j+i}
=\sum^\infty_{i=0}{\rm rank}H^i_G(\mathfrak{F}_{\mathcal{H},b}, \mathfrak{F}_{\mathcal{H},a};{\bf K})t^i+ (1+t)Q(t),
\end{eqnarray}
where $\theta^-_j$ is the orientation bundle of $N^-\mathcal{O}_j$, $j=1,\cdots,k$.
In particular, if  $G$ is a trivial group and each $\mathcal{O}_j$ becomes a nondegenerate
 critical point $\vec{u}_j$, we have the Morse inequalities:
\begin{eqnarray}\label{e:Morse.9}
\sum^l_{j=0}(-1)^{l-j}N_j(a,b)\ge\sum^l_{j=0}(-1)^{l-j}\beta_j(a,b),\quad\forall l\in\mathbb{N}_0,
\end{eqnarray}
where for each $q\in\mathbb{N}_0$, $N_q(a,b)=\sharp\{1\le i\le k\,|\,\mu_i=q\}$ (the number of points in $\{\vec{u}_j\}^k_{j=1}$ with Morse index $q$)  and
$\beta_q(a,b)=\sum^k_{i=1}{\rm rank} H_q(\mathfrak{F}_{\mathcal{H},b}, \mathfrak{F}_{\mathcal{H},a}; {\bf K})$.
Furthermore, if  $\mathfrak{F}_{\mathcal{H}}$ is coercive,  has only nondegenerate critical points,  and
  for each $q\in\mathbb{N}_0$ there exist only  finitely many critical points with Morse index $q$,
 then the following relations hold:
\begin{eqnarray}\label{e:Morse.10}
 \sum^q_{i=0}(-1)^{q-i}N_i\ge (-1)^q,\;\forall q\in\mathbb{N}_0,\quad\hbox{and}\quad
  \sum^\infty_{i=0}(-1)^iN_i=1,
 \end{eqnarray}
 where $N_i$ is the number of critical points of $\mathfrak{F}_{\mathcal{H}}$ with Morse index $i$.
\end{theorem}

\begin{remark}\label{rem:Morse.10}
{\rm {\bf (i)} When $N=1$,  ${\mathcal{H}}=W^{m,2}_0(\Omega)$ and $\mathfrak{F}_{\mathcal{H}}$ is
coercive, (\ref{e:Morse.9}) was first obtained by Skrypnik in \cite[\S5.2]{Skr1} and \cite[Theorem~4.7, Chap.1]{Skr2}.
Instead of using Morse-Palais lemma, his ideas are similar to Smale's \cite{Sma}, but some new techniques
are employed, which motivated our current work. \\
{\bf (ii)} If $m=N=1$, $\partial\Omega$ is of class $C^{2+\alpha}$
 for some $\alpha\in (0,1)$, and $p>n=\dim\Omega$ such that $W^{p,2}\subset C^1$, under some conditions
 on $F$ Strohmer \cite{Str} proved a handle body theorem for $\mathfrak{F}$ on $Z_\varphi=\{u\in W^{2,p}(\Omega)\,|\, u|_{\partial\Omega}=\varphi|_{\partial\Omega}\}$  for $\varphi\in C^{2+\alpha}$.
His conditions and those of  Theorem~\ref{th:Morse.2}  cannot be contained each other.\\
{\bf (iii)} As in Remark~\ref{rem:4.4}, we may give the corresponding versions of Theorems~\ref{th:Morse.1},\ref{th:Morse.2}
 in the setting of \cite{PaSm, Sma, Pa2}.  In particular, replace $\Omega\subset\mathbb{R}^n$ by $\mathbb{T}^n=\mathbb{R}^n/\mathbb{Z}^n$ and assume that a $C^2$ function
 $F:\mathbb{T}^n\times\prod^m_{k=0}\mathbb{R}^{N\times M_0(k)}\to\mathbb{R}$
 satisfies \textsf{Hypothesis} $\mathfrak{F}_{2,N,m,n}$ when restricted to
$[0,1]^n\times\prod^m_{k=0}\mathbb{R}^{N\times M_0(k)}$, then for the action of
$G=\mathbb{T}^n$ or $\mathbb{T}^1$ on $W^{m,2}(\mathbb{T}^n,\mathbb{R}^N)$ given by the
isometric linear representation
\begin{eqnarray*}
&&([t_1,\cdots,t_n]\cdot\vec{u})(x_1,\cdots,x_n)=\vec{u}(x_1+t_1,\cdots, x_n+t_n),\quad
 [t_1,\cdots,t_n]\in \mathbb{T}^n,\\
 \hbox{or}&& ([t]\cdot\vec{u})(x_1,\cdots,x_n)=\vec{u}(x_1+t,\cdots, x_n+t),\quad
 [t]\in \mathbb{T}^1,
\end{eqnarray*}
Theorems~\ref{th:Morse.1},\ref{th:Morse.2}  hold true.
These provide necessary tools for generalizing works in \cite{Lu1, Van}.
If the function $F$ is defined on $\mathbb{T}^n\times\mathbb{T}^N\times\prod^m_{k=1}\mathbb{R}^{N\times M_0(k)}$
the corresponding variational problem on $W^{m,2}(\mathbb{T}^n,\mathbb{T}^N)$ is related to
\cite{Mos} and may be also considered with our theory.}
\end{remark}

\begin{corollary}\label{cor:Morse.11}
 Given  integers $m, N\ge 1$, $n\ge 2$, let $\Omega\subset\R^n$ be a Sobolev domain
 for $(2,m,n)$, and let $V_0$ be a closed subspace of $W^{m,2}(\Omega, \mathbb{R}^N)$ and $V=\vec{w}+V_0$
 for some $\vec{w}\in W^{m,2}(\Omega, \mathbb{R}^N)$.
Suppose that \textsf{Hypothesis} $\mathfrak{F}_{2,N,m,n}$ holds. Then
each critical point of $\mathfrak{F}_V$  has finite Morse index $\mu$ and nullity $\nu$;
moreover, if  $\vec{u}\in V$ is an isolated critical point of $\mathfrak{F}_V$,
 for any Able group ${\bf K}$, ${\rm rank}C_j(\mathfrak{F}_V, \vec{u};{\bf K})<\infty\;\forall j\in\mathbb{N}_0$,  and $C_j(\mathfrak{F}_V, \vec{u};{\bf K})=0$ for $j<\mu$ or $j>\mu+\nu$.
\end{corollary}

This is a direct consequence of Theorem~\ref{th:S.1.3} and Corollary~\ref{cor:4.4}.
From Corollary~\ref{cor:4.4}, Theorems~\ref{th:S.1.1},~\ref{th:4.5} and \cite[Proposition~6.93]{MoMoPa}
we immediately deduce

\begin{corollary}\label{cor:Morse.12}
Let $V\subset W^{m,2}(\Omega, \mathbb{R}^N)$ be as in Corollary~\ref{cor:Morse.11}. If
$\mathfrak{F}_V$ is bounded below, satisfies the (PS)-condition, and  has  a nondegenerate
critical point which is not a global minimizer, then it has at least three critical points.
\end{corollary}

The last two corollaries also hold if $\Omega$ is replaced by $\mathbb{T}^n$.
For a $C^2$ Lagrangian satisfying  the controllable growth conditions the
corresponding integral functional is bounded below and coercive.
From the above results we immediately get

\begin{theorem}\label{th:Morse.14}
Let $\Omega\subset\R^n$ be  a  Sobolev domain for $(2,1,n)$, $N\in\mathbb{N}$, and $H$  a closed subspace of $W^{1,2}(\Omega,\mathbb{R}^N)$, $\mathcal{H}=\vec{\omega}+ H$ for some $\vec{\omega}\in W^{1,2}(\Omega,\mathbb{R}^N)$.
 Assume that
$\overline\Omega\times\mathbb{R}^N\times\mathbb{R}^{N\times n}\ni (x,
z,p)\mapsto F(x, z,p)\in\R$ is a $C^2$ function fulfilling   the controllable growth conditions
 (see Appendix A). Let $G$ be a compact Lie group which acts on $\mathcal{H}$ in a $C^3$-smooth
isometric way. Suppose that the restriction functional $\mathfrak{F}_{\mathcal{H}}:=\mathfrak{F}|_{\mathcal{H}}$  is $G$-invariant.
Then\\
 {\bf (i)}  If  $a<b$ are two regular values of $\mathfrak{F}_{\mathcal{H}}$ and
  $\mathfrak{F}_{\mathcal{H}}^{-1}([a,b])$  contains only nondegenerate  critical orbits $\mathcal{O}_j$
  with Morse indexes $\mu_j$, $j=1,\cdots,k$, then (\ref{e:Morse.7}) and (\ref{e:Morse.8}) hold;
  in particular, if $G$ is trivial and  $\mathfrak{F}_{\mathcal{H}}^{-1}([a,b])$
    contains only nondegenerate  critical points,   then (\ref{e:Morse.9}) holds.\\
{\bf (ii)}  If  $\mathfrak{F}_{\mathcal{H}}$   has only nondegenerate critical points,  and
  for each $q\in\mathbb{N}_0$ there exist only  finitely many critical points with Morse index $q$,
  then (\ref{e:Morse.10}) holds true.\\
{\bf (iii)} If  $\vec{u}$ is a critical point of $\mathfrak{F}$ on $V:=\vec{w}+ W^{1,2}_0(\Omega,\mathbb{R}^N)\subset W^{1,2}(\Omega,\mathbb{R}^N)$ which is not a global minimizer,
then it has at least three critical points on $V$
provided that the bilinear form
   $$
  W^{1,2}_0(\Omega,\mathbb{R}^N)\times W^{1,2}_0(\Omega,\mathbb{R}^N)\ni (\vec{v}, \vec{w})\mapsto \sum^N_{i=1}\sum_{|\alpha|,|\beta|\le 1}\int_\Omega
  F^{ij}_{\alpha\beta}(x, \vec{u}(x), D\vec{u}(x))D^\beta v^j\cdot D^\alpha w^i dx
  $$
is nondegenerate.\\
{\bf (iv)} If $\Omega$ is replaced by $\mathbb{T}^n$ in (i) and (ii) the corresponding conclusions
also holds.
\end{theorem}

\subsection{Applicability of related previous work}\label{sec:compare}

In this section we study under what conditions on $F$ splitting theorems in \cite{Lu1, Lu2} and
\cite{BoBu, JM} are applicable. As consequences,  under \textsf{Hypothesis} $\mathfrak{F}_{2,N,m,n}$,
if an isolated critical point $\vec{u}$ of the functional $\mathfrak{F}_H$ on $H:=W^{m,2}_0(\Omega, \mathbb{R}^N)$
defined by the right side of (\ref{e:1.3}) is smooth enough then the critical groups of
$\mathfrak{F}_H$ at $\vec{u}$ are equal to those of the restriction of $\mathfrak{F}_H$ to a smaller appropriate space containing it.
For these, the following special case of \cite[Theorem 6.4.8]{Mo} is very key.
%%%%%%%%%%%%%%%%%%%%%%%%%%%%%%%%%%%%%%%%%%%%%%%%%%%%%%%%%%%%%%%
%%see \cite[Theorem~4.14]{GiMa} for the case  $m=1$.
%%%%%%%%%%%%%%%%%%%%%%%%%%%%%%%%%%%%%%%%%%%%%%%%

\begin{proposition}\label{prop:BifE.9.1}
  For a real $p\ge 2$ and an integer $k\ge m+\frac{n}{p}$,  let $\Omega\subset\R^n$ be a bounded
   domain with boundary of class $C^{k-1,1}$, $N\in\mathbb{N}$, and let bounded and measurable functions on $\overline{\Omega}$,
    $A^{ij}_{\alpha\beta}$, $i,j=1,\cdots,N$, $|\alpha|,|\beta|\le m$,  fulfill the following conditions:\\
{\bf (i)}  $A^{ij}_{\alpha\beta}\in C^{k+|\alpha|-2m-1,1}(\overline{\Omega})$ if $2m-k<|\alpha|\le m$,\\
{\bf (ii)} there exists $c_0>0$ such that
$$
\sum^N_{i,j=1}\sum_{|\alpha|=|\beta|=m}\int_\Omega
A^{ij}_{\alpha\beta}\eta^i_\alpha\eta^j_\beta\ge c_0\sum^N_{i=1}\sum_{|\alpha|=m}|\eta^i_\alpha|^2,\quad\forall \eta\in\mathbb{R}^{N\times M_0(m)}.
$$
Suppose that $\vec{u}=(u^1,\cdots,u^N)\in W^{m,2}_0(\Omega,\mathbb{R}^N)$ and $\lambda\in (-\infty, 0]$ satisfy
$$
\sum^N_{i,j=1}\sum_{|\alpha|,|\beta|\le m}\int_\Omega
(A^{ij}_{\alpha\beta}-\lambda\delta_{ij}\delta_{\alpha\beta})D^\beta u^i\cdot D^\alpha v^j dx=0\quad\forall v\in W^{m,2}_0(\Omega,\mathbb{R}^N).
$$
Then $\vec{u}\in W^{k,p}(\Omega,\mathbb{R}^N)$. Moreover, for $f_j=\sum_{|\alpha|\le m}(-1)^{|\alpha|}D^\alpha f^j_\alpha$,
where $f^j_\alpha\in W^{k-2m+|\alpha|,p}(\Omega)$ if $|\alpha|>2m-k$, and
$f^j_\alpha\in L^p(\Omega)$ if $|\alpha|\le 2m-k$, suppose that
$\vec{u}\in W^{m,2}_0(\Omega,\mathbb{R}^N)$ satisfy
$$
\int_\Omega\sum^N_{j=1}\sum_{|\alpha|\le m}\left[\sum^N_{i=1}\sum_{|\beta|\le m}A^{ij}_{\alpha\beta}D^\beta u^i-f^j_\alpha      \right]D^\alpha v^jdx=0,\quad\forall v\in W^{m,2}_0(\Omega,\mathbb{R}^N),
$$
we have also $\vec{u}\in W^{k,p}(\Omega,\mathbb{R}^N)$.
\end{proposition}

Without special statements,  the Hilbert space $H=W^{m,2}_0(\Omega,\mathbb{R}^N)$
with the usual inner product
\begin{equation}\label{e:compare.0}
(\vec{u},\vec{v})_H=\sum^N_{i=1}\sum_{|\alpha|=m}\int_\Omega D^\alpha u^i D^\alpha v^i dx.
\end{equation}

\begin{theorem}\label{th:compare.1}
 Under \textsf{Hypothesis} $\mathfrak{F}_{2,N,m,n}$,
  let $\mathfrak{F}_H$ denote the functional on $H$ defined by the right side of
 (\ref{e:1.3}), and  let $\vec{u}^\ast\in H$  be a
 critical point of $\mathfrak{F}_H$. Suppose that there exist a real $p>1$ and
  an integer $k>m+ \frac{n}{p}$  such that
 $\vec{u}^\ast$ is contained in the Banach subspace
 $X_{k,p}:=W^{k,p}(\Omega,\mathbb{R}^N)\cap W^{m,2}_0(\Omega,\mathbb{R}^N)$ of
 $W^{k,p}(\Omega,\mathbb{R}^N)$. Suppose also that $F$ is of class $C^{k-m+2}$.
 Then near $\vec{u}^\ast\in H$  the triple $(\mathfrak{F}, X, H)$ satisfies the conditions of
 \cite[Theorem~2.1]{Lu2} except  (C2) in  \cite[page 2944]{Lu2}.
 Moreover,  if $\vec{u}^\ast\in C^k(\overline{\Omega}, \mathbb{R}^N)$, $p\ge 2$ and $\partial\Omega$ is of class $C^{k-1,1}$, then (C2) in  \cite[page 2944]{Lu2} is also fulfilled, and the negative definite space of
 $D(\nabla\mathfrak{F}_H)(\vec{u}^\ast)\in\mathscr{L}_s(H)$ is contained in $X_{k,p}$.

 %%%%%%%%%%%%%%%%%%%%%%%%%%%%%%%%%%%%%%%%%%%%%%%%%%%%%%%%%%%%%%%%%%%%%%%%%%%%%%%%%%%%%%%%%
 %%Then near $\vec{u}^\ast\in H$  the triple $(\mathfrak{F}, X, H)$ satisfies the conditions of
 %%\cite[Theorem~2.1]{Lu2}  if $\partial\Omega$ and $F$ are of classes $C^{k-1,1}$ and $C^{k-m+2}$
 %%respectively.
 %%%%%%%%%%%%%%%%%%%%%%%%%%%%%%%%%%%%%%%%%%%%%%%%%%%%%%%%%%%%%%%%%%%%%%
  \end{theorem}

\begin{proof}[\bf Proof]
Denote by $\nabla\mathfrak{F}_H$  the gradient of $\mathfrak{F}_H$.
For $s<0$ let $W^{s,p}(\Omega)=[W_0^{-s,p'}(\Omega)]^\ast$ as usual, where $p'=p/(p-1)$.
Note that the $m$th power of the Laplace operator, $\triangle^m$, is an isomorphism from a Banach subspace
$W^{k,p}(\Omega)\cap W^{m,2}_0(\Omega)$ of $W^{k,p}(\Omega)$
to $W^{k-2m,p}(\Omega)$, and thus that its inverse, denoted by $\triangle^{-m}$, is
from $W^{k-2m,p}(\Omega)$ to $W^{k,p}(\Omega)\cap W^{m,2}_0(\Omega)$.
By (\ref{e:4.1}) and (\ref{e:compare.0}), it is easily computed that for $i=1,\cdots,N$,
\begin{equation}\label{e:compare.1}
(\nabla\mathfrak{F}_H(\vec{u}))^i=\triangle^{-m}\sum_{|\alpha|\le m}(-1)^{m+|\alpha|}D^\alpha(F^i_\alpha(\cdot,
\vec{u}(\cdot),\cdots, D^m \vec{u}(\cdot))),\quad\forall \vec{u}\in H.
\end{equation}
 As in Theorem~\ref{th:4.1},
$\nabla\mathfrak{F}_H$  has the G\^ateaux derivative
$D(\nabla\mathfrak{F}_H)(\vec{u})\in\mathscr{L}_s(H)$ at $\vec{u}\in H$
such that for any $\vec{v}, \vec{\varphi}\in H$,
$(D(\nabla\mathfrak{F}_H)(\vec{u})[\vec{v}],\vec{\varphi})_H$ is given by the right side
of (\ref{e:4.2}). Denote by  $\mathbb{B}$  the restriction of
 $D(\nabla\mathfrak{F}_H)$ to $X_{k,p}$.
For $\vec{u}\in X_{k,p}$ and $\vec{v}\in H$ let $\mathbb{B}(\vec{u})\vec{v}=((\mathbb{B}(\vec{u})\vec{v})^1,\cdots,
(\mathbb{B}(\vec{u})\vec{v})^N)$. It is easily verified that
\begin{equation}\label{e:compare.2}
(\mathbb{B}(\vec{u})\vec{v})^i= \triangle^{-m}\sum^N_{j=1}\sum_{\scriptsize\begin{array}{ll}
   &|\alpha|\le m,\\
   &|\beta|\le m
   \end{array}}(-1)^{m+|\alpha|}
D^\alpha(F^{ij}_{\alpha\beta}(\cdot, \vec{u}(\cdot),\cdots, D^m \vec{u}(\cdot))D^\beta v^j).
\end{equation}
Let $\mathbb{A}$ denote the restriction of $\nabla\mathfrak{F}_H$ to $X_{k,p}$. We have

\begin{claim}\label{cl:compare.1.2}
 $\mathbb{A}$ is a $C^1$ map from $X_{k,p}$ to itself,  and satisfies
 $d\mathbb{A}(\vec{u})[\vec{v}]=\mathbb{B}(\vec{u})\vec{v}$ for all $\vec{u}, \vec{v}\in X_{k,p}$.
   \end{claim}

We first admit this claim and postpone its proof.

For each $\vec{u}\in X_{k,p}$, we may write $\mathbb{B}(\vec{u})=\mathbb{P}(\vec{u})+ \mathbb{Q}(\vec{u})$, where
for $i=1,\cdots,N$,
\begin{eqnarray}\label{e:compare.3}
&&(\mathbb{P}(\vec{u})\vec{v})^i= \triangle^{-m}\sum^N_{j=1}\sum_{|\alpha|=|\beta|=m}(-1)^{m+|\alpha|}
D^\alpha(F^{ij}_{\alpha\beta}(\cdot, \vec{u}(\cdot),\cdots, D^m \vec{u}(\cdot))D^\beta v^j),\\
&&(\mathbb{Q}(\vec{u})\vec{v})^i=\triangle^{-m}\sum^N_{j=1}\sum_{\scriptsize\begin{array}{ll}
   &|\alpha|\le m, |\beta|\le m,\\
   &|\alpha|+|\beta|<2m
   \end{array}}D^\alpha(F^{ij}_{\alpha\beta}(\cdot, \vec{u}(\cdot),\cdots, D^m \vec{u}(\cdot))D^\beta v^j).\label{e:compare.4}
\end{eqnarray}
As in the proofs of Theorem~\ref{th:4.1},~\ref{th:4.2} (cf. \cite{Lu7}) we can derive that
$\mathbb{P}(\vec{u})$ and $\mathbb{Q}(\vec{u})$ are positive definite and
 completely continuous, respectively, and they also satisfy the condition (D) in
 \cite[page 2944]{Lu2}. By \cite[Proposition~B.2]{Lu2} we also see that
 (C1) in  \cite[page 2944]{Lu2} holds for the operator $\mathbb{B}(\vec{u}^\ast)$.

 Next, we prove the second claim.  Let $A^{ij}_{\alpha\beta}:=F^{ij}_{\alpha\beta}(\cdot, \vec{u}^\ast(\cdot),\cdots, D^m \vec{u}^\ast(\cdot))$. They sit in $C^{k-m}(\overline{\Omega},\mathbb{R}^N)$ because
  $\vec{u}^\ast\in C^k(\overline{\Omega}, \mathbb{R}^N)$.
 Let $\vec{u}\in H$ be such that $\vec{w}:=\mathbb{B}(\vec{u}^\ast)\vec{u}$ sits in $X_{k,p}$.
Then (\ref{e:compare.2}) implies that
$$
\int_\Omega\sum^N_{j=1}\sum_{|\alpha|\le m}\left[\sum^N_{i=1}\sum_{|\beta|\le m}A^{ij}_{\alpha\beta}D^\beta u^i-f^j_\alpha      \right]D^\alpha v^jdx=0,\quad\forall v\in W^{m,2}_0(\Omega,\mathbb{R}^N),
$$
where $f_\alpha^j=0$ for $|\alpha|<m$, and $f_\alpha^j=(-1)^{|\alpha|}D^\alpha w^j$ for $|\alpha|=m$.
 Note that $w^j\in W^{k,p}(\Omega)$ and $k>m$ lead to $f_\alpha^j\in W^{k-m,p}(\Omega)=W^{k-2m+|\alpha|,p}(\Omega)$
for $|\alpha|=m>2m-k$.
Since $p\ge 2$,  $\partial\Omega$ is of classes $C^{k-1,1}$ and
 (\ref{e:1.2}) implies that Proposition~\ref{prop:BifE.9.1}(ii) is satisfied,
we may use the second claim of Proposition~\ref{prop:BifE.9.1} to deduce that
$\vec{u}\in W^{k,p}(\Omega,\mathbb{R}^N)$. That is, (C2) in  \cite[page 2944]{Lu2} is fulfilled.

The final conclusion may be derived from the first claim of Proposition~\ref{prop:BifE.9.1} as above.
\end{proof}

When $n=2$,  we see from  \cite[Chapter 7, Th.4.4]{Skr3} that
 the conditions of this theorem can be satisfied if $F$ is smooth enough.

\begin{proof}[\bf Proof of Claim~\ref{cl:compare.1.2}]\quad
 Let $r=\dim(\Omega\times\prod^m_{k=0}\mathbb{R}^{N\times M_0(k)})=n+ (m+1)N+
  \sum^m_{k=0}M_0(k)$. For $\vec{u}\in X_{k,p}$ and $x\in\overline{\Omega}$ put ${\bf u}(x)=(x,
  \vec{u}(x),\cdots, D^m \vec{u}(x))$. Then $\Upsilon(\vec{u})={\bf u}$ defines
   an affine (and thus smooth) map $\Upsilon$ from $X_{k,p}$ to $W^{k-m, p}(\Omega,\mathbb{R}^r)$.
   Since $(k-m)p>n$ and $F^i_\alpha$ is of class $C^{k-m+1}$,  by \cite[Lemma~2.96]{Wend} (with
  ${\bf X}=W^{k-m,p}$) the map
  $$
  \Phi_{F^i_\alpha}: W^{k-m, p}(\Omega,\mathbb{R}^r)\to W^{k-m,p}(\Omega),\;{\bf u}\mapsto F^i_\alpha\circ{\bf u}
  $$
  is of class $C^1$, and $d\Phi_{F^i_\alpha}({\bf u}){\bf v}=(dF^i_\alpha\circ{\bf u}){\bf v}$
  for any ${\bf u},{\bf v}\in W^{k-m, p}(\Omega,\mathbb{R}^r)$.
  Hence
   $$
   \mathbb{A}_{i,\alpha}=\Phi_{F^i_\alpha}\circ\Upsilon:X_{k,p}\to W^{k-m,p}(\Omega),\;\vec{u}\mapsto
  F^i_\alpha(\cdot, \vec{u}(\cdot),\cdots, D^m \vec{u}(\cdot))
  $$
   is of class $C^1$ and
  $$
  d\mathbb{A}_{i,\alpha}(\vec{u})[\vec{v}]=d(\Phi_{F^i_\alpha}\circ\Upsilon)(\vec{u})[\vec{v}]=
  d(\Phi_{F^i_\alpha})(\Upsilon(\vec{u}))[d\Upsilon(\vec{u})[\vec{v}]]=(dF^i_\alpha\circ{\bf u})({\bf v}-\psi),
  $$
 where $\psi:\Omega\to \mathbb{R}^r$ is given by $\psi(x)=(x,0,\cdots,0)$ for $x\in\Omega$.
Clearly, (\ref{e:compare.1}) implies
  $$
(\mathbb{A}(\vec{u}))^i=\triangle^{-m}\sum_{|\alpha|\le m}(-1)^{m+|\alpha|}D^\alpha(\mathbb{A}_{i,\alpha}(\vec{u})).
$$
 Since $\triangle^{-m}:W^{k-2m,p}(\Omega)\to W^{k,p}(\Omega)\cap W^{m,2}_0(\Omega)$
is a Banach space isomorphism and
$D^\alpha: W^{k-m,p}(\Omega)\to W^{k-m-|\alpha|,p}(\Omega)$ is a continuous linear operator,
we deduce that $\mathbb{A}$ is a $C^1$ map from $X_{k,p}$ to itself.
The second conclusion easily follows from the above arguments.
We may also obtain it as follows.
 Denote by the inclusion $\imath:X_{k,p}\hookrightarrow H$.
 Since  $\imath\circ\mathbb{A}=\nabla\mathfrak{F}_H|_{X_{k,p}}$
  and  $(D(\nabla\mathfrak{F}_H)(\vec{u})[\vec{v}],\vec{\varphi})_H$ is equal to the right side
of (\ref{e:4.2}), it follows that $d\mathbb{A}(\vec{u})[\vec{v}]=\mathbb{B}(\vec{u})\vec{v}$.
  \end{proof}

  Let $Y$ be the Banach subspace $C^m(\overline{\Omega}, \mathbb{R}^N)\cap W^{m,2}_0({\Omega}, \mathbb{R}^N)$ of
  $C^m(\overline{\Omega}, \mathbb{R}^N)$. (Since $k\ge m+1$ and $\partial\Omega$ is of class $C^{k-1,1}$,
   we have 
   $$
   Y=\{\vec{u}\in C^m(\overline{\Omega}, \mathbb{R}^N)\,|\, D^s\vec{u}|_{\partial\Omega}=0,\;
   s=0,\cdots,m-1,\;D^{m-1}\vec{u}\in W^{1,2}_0\}
   $$ 
   by \cite[Theorem~9.17]{Bre}.)
   Let $\mathfrak{F}_Y$ denote the restriction of $\mathfrak{F}_H$ to $Y$.
  Since $F$ is of class $C^{k-m+2}$, it follows from $\omega$-lemma (cf. \cite[Lemma~2.96]{Wend}) that $\mathfrak{F}_Y$ is of class  $C^{k-m+2}$.  Define ${\bf B}:Y\to \mathscr{L}_s(H)$ by ${\bf B}(\vec{u})=D(\nabla\mathfrak{F}_H)(\vec{u})$.
  Then $d^2(\mathfrak{F}_Y)(\vec{u})[\vec{v},\vec{w}]=({\bf B}(\vec{u})\vec{v},\vec{w})_H$
  for any $\vec{u},\vec{v},\vec{w}\in Y$.

   For $\vec{u}, \vec{v}\in X_{k,p}$ (resp. $\vec{u}, \vec{v}\in Y$) and $\vec{\varphi}, \vec{\psi}\in H$, (\ref{e:4.2}) yields
  \begin{eqnarray*}
   &&((\mathbb{B}(\vec{u})-\mathbb{B}(\vec{v}))\vec{\psi},\vec{\varphi})_H\quad\hbox{(resp.
   $(({\bf B}(\vec{u})-{\bf B}(\vec{v}))\vec{\psi},\vec{\varphi})_H$)}\\
   &&=\sum^N_{i,j=1}\!\!\!\!\!\!
   \sum_{\scriptsize{\begin{array}{ll}
   &|\alpha|\le m,\\
   &|\beta|\le m
   \end{array}}}\!\!\!\int_\Omega\bigl[
  F^{ij}_{\alpha\beta}(x, \vec{u}(x),\cdots, D^m \vec{u}(x))-
  F^{ij}_{\alpha\beta}(x, \vec{v}(x),\cdots, D^m \vec{v}(x))
  \bigr]D^\beta {\psi}^j\cdot D^\alpha\varphi^i dx.
    \end{eqnarray*}
 Since $X_{k,p}$ may be continuously embedded into the space $Y$,
 the above equality implies

  \begin{claim}\label{cl:compare.1.3}
The map $\mathbb{B}:X_{k,p}\to\mathscr{L}_s(H)$ (resp. ${\bf B}:Y\to\mathscr{L}_s(H)$)
is uniformly continuous on any bounded subset
of $X_{k,p}$ (resp. $Y$).
   \end{claim}

  \begin{theorem}\label{th:compare.1.4}
  For a real $p\ge 2$ and an integer $k\ge m+\frac{n}{p}$,  let $\Omega\subset\R^n$ be a bounded
   domain with boundary of class $C^{k-1,1}$, $N\in\mathbb{N}$, and let
$\overline\Omega\times\prod^m_{k=0}\mathbb{R}^{N\times M_0(k)}\ni (x,
\xi)\mapsto F(x,\xi)\in\R$ be of class $C^{k-m+2}$.
 Suppose that a critical point $\vec{u}^\ast$ of $\mathfrak{F}_{X_{k,p}}$ on $X_{k,p}=W^{k,p}(\Omega,\mathbb{R}^N)\cap W^{m,2}_0(\Omega,\mathbb{R}^N)$ belongs to $C^k(\overline{\Omega}, \mathbb{R}^N)$ and that
 there exists $c>0$ such that
for any $\eta=(\eta^{i}_{\alpha})\in\R^{N\times M_0(m)}$ and for any $x\in\overline{\Omega}$,
\begin{eqnarray}\label{e:compare.6}
\sum^N_{i,j=1}\sum_{|\alpha|=|\beta|=m}F^{ij}_{\alpha\beta}
(x,\vec{u}^\ast(x),D\vec{u}^\ast(x),\cdots,D^m\vec{u}^\ast(x))\eta^i_\alpha\eta^j_\beta\ge
c\sum^N_{i=1}\sum_{|\alpha|= m}(\eta^i_\alpha)^2.
\end{eqnarray}
For $\vec{u}\in X_{k,p}$ (resp. $\vec{u}\in Y$) and $\vec{v}\in H$, let $(\mathbb{A}(\vec{u}))^i$ and
$({\bf B}(\vec{u})\vec{v})^i$ be still defined by the right side of (\ref{e:compare.1}) and
(\ref{e:compare.2}), respectively. Then maps $\mathbb{A}:X_{k,p}\to X_{k,p}$ and ${\bf B}:Y\to\mathscr{L}(H)$
 satisfy Claims~\ref{cl:compare.1.2},~\ref{cl:compare.1.3}, and so
$(\mathfrak{F}_{X_{k,p}}, \mathbb{A}, {\bf B},  X_{k,p}, Y)$ fulfills the conditions of  \cite[Theorem~2.5]{JM}  near $\vec{u}^\ast\in X_{k,p}$. Moreover the negative definite space of
 ${\bf B}(\vec{u}^\ast)\in\mathscr{L}_s(H)$ is contained in $X_{k,p}$.
    \end{theorem}

 \begin{corollary}\label{cor:compare.1.4}
  For a real $p\ge 2$ and an integer $k\ge m+\frac{n}{p}$,  let $\Omega\subset\R^n$ be a bounded
   domain with boundary of class $C^{k-1,1}$, $N\in\mathbb{N}$, and let
$\overline\Omega\times\prod^m_{k=0}\mathbb{R}^{N\times M_0(k)}\ni (x,
\xi)\mapsto F(x,\xi)\in\R$ be of class $C^{k-m+2}$. Then\\
{\bf (i)} if a critical point $\vec{u}^\ast$ of $\mathfrak{F}_{X_{k,p}}$ on $X_{k,p}=W^{k,p}(\Omega,\mathbb{R}^N)\cap W^{m,2}_0(\Omega,\mathbb{R}^N)$ belongs to $C^k(\overline{\Omega}, \mathbb{R}^N)$ and
(\ref{e:compare.6}) also holds for some $c>0$ and for all $\eta=(\eta^{i}_{\alpha})\in\R^{N\times M_0(m)}$ and $x\in\overline{\Omega}$, we have
   $C_\ast(\mathfrak{F}_{X_{k,p}},\vec{u}^\ast;{\bf K})=C_\ast(\mathfrak{F}_{Y},\vec{u}^\ast;{\bf K})$
provided that $\vec{u}^\ast$ is  an isolated critical point for $\mathfrak{F}_Y$ (and so for $\mathfrak{F}_{X_{k,p}}$);\\
   {\bf (ii)} if \textsf{Hypothesis} $\mathfrak{F}_{2,N,m,n}$ is also satisfied
     and a critical point $\vec{u}^\ast$ of $\mathfrak{F}_H$ (as in Theorem~\ref{th:compare.1})
     belongs to $C^k(\overline{\Omega}, \mathbb{R}^N)$,
 we have $C_\ast(\mathfrak{F}_H,\vec{u}^\ast;{\bf K})=
C_\ast(\mathfrak{F}_{X_{k,p}},\vec{u}^\ast;{\bf K})$
provided that $\vec{u}^\ast$ is  an isolated critical
 point for $\mathfrak{F}_H$ (and so for  $\mathfrak{F}_{X_{k,p}}$).
   \end{corollary}

By Theorem~\ref{th:compare.1.4}, (i) follows from \cite[Corollary~2.8]{JM}.
Using Theorem~\ref{th:compare.1} we may obtain (ii) from Remark~2.2(i) and Corollary~2.6 in
\cite{Lu2}.

Finally, we state the following more general version of \cite[Theorems~2.1,2.2]{BoBu}.

\begin{theorem}\label{th:compare.1.6}
  For a real $p\ge 2$ and an integer $k> m+\frac{n}{p}$,  let $\Omega\subset\R^n$ be a bounded
   domain with boundary of class $C^{k-1,1}$, $N\in\mathbb{N}$, and let
   $\overline\Omega\times\prod^m_{k=0}\mathbb{R}^{N\times M_0(k)}\ni (x,
\xi)\mapsto F(x,\xi)\in\R$ be of class $C^{k-m+3}$.
 As above, we have the functional $\mathfrak{F}_{X_{k,p}}$ on
 $X_{k,p}=W^{k,p}(\Omega,\mathbb{R}^N)\cap W^{m,2}_0(\Omega,\mathbb{R}^N)$,
 which is of class $C^{k-m+3}$  by \cite[Lemma~2.96]{Wend} (with
  ${\bf X}=C^0$) as in the proof of Claim~\ref{cl:compare.1.2}).
Suppose that a critical point of $\mathfrak{F}_{X_{k,p}}$, $\vec{u}^\ast$, belongs to
$C^k(\overline{\Omega}, \mathbb{R}^N)$ and that there exists $c>0$ such that
(\ref{e:compare.6}) holds for any $\eta=(\eta^{i}_{\alpha})\in\R^{N\times M_0(m)}$ and for any $x\in\overline{\Omega}$.
Let $\mathbb{B}(\vec{u}^\ast)$ be given by (\ref{e:compare.2}). Then
it has finite dimensional kernel and negative definite space, $H^0$ and $H^-$, which are contained
in $X_{k,p}$. Denote by $X^+_{k,p}$ the intersection of $X_{k,p}$ with the positive definite space
of $\mathbb{B}(\vec{u}^\ast)$, and by $P^0, P^-, P^+$ the projections onto $H^0, H^-, X^+_{k,p}$
yielded by the Banach space direct sum decomposition $X_{k,p}=H^0\oplus H^-\oplus X^+_{k,p}$.
Then we have:
\begin{description}
\item[(i)] if $H^0=\{\theta\}$,  there exists a $C^1$ diffeomorphism $\varphi:U\to X_{k,p}$
in some neighborhood $U\subset X_{k,p}$ of zero  such that $\varphi(\theta)=\theta$ and
$$
\mathfrak{F}_{X_{k,p}}(\varphi(\vec{u})+ \vec{u}^\ast)=\|P^+\vec{u}\|_H^2-\|P^-\vec{u}\|_H^2+
\mathfrak{F}_{X_{k,p}}(\vec{u}^\ast)\quad\forall\vec{u}\in U;
$$
\item[(ii)] if $H^0\ne\{\theta\}$, there exist $\epsilon>0$, a (unique) $C^1$ map
$\mathfrak{h}: B_{H^0}(\theta,\epsilon)\to X_{k,p}^+\oplus H^-$
satisfying $\mathfrak{h}(\theta)=\theta$ and
$(P^++P^-)\nabla\mathfrak{F}_{X_{k,p}}(\vec{u}^\ast+ z+ \mathfrak{h}(z))=0\;\forall z\in B_{H^0}(\theta,\epsilon)$,
and a $C^1$ diffeomorphism $\varphi:U\to X_{k,p}$
in some neighborhood $U\subset X_{k,p}$ of zero  such that $\varphi(\theta)=\theta$ and
\begin{eqnarray*}
\mathfrak{F}_{X_{k,p}}(\varphi(\vec{u})+ \vec{u}^\ast)=\|P^+\vec{u}\|_H^2-\|P^-\vec{u}\|_H^2+
\mathfrak{F}_{X_{k,p}}^\circ(P^0\vec{u})\quad\forall\vec{u}\in U,
\end{eqnarray*}
where $\nabla\mathfrak{F}_{X_{k,p}}$ is the gradient of $\mathfrak{F}_{X_{k,p}}$ with respect to
the inner product in (\ref{e:compare.0}), and
$\mathfrak{F}_{X_{k,p}}^\circ$ is a $C^2$ map on $B_{H^0}(\theta,\epsilon)$ defined by
$\mathfrak{F}_{X_{k,p}}^\circ(z)=\mathfrak{F}_{X_{k,p}}(z+ \mathfrak{h}(z)+\vec{u}^\ast)$,
which has zero as a critical point.
\end{description}
 \end{theorem}

In the present case, for $\vec{u}\in X_{k,p}$  and $\vec{v}\in H$, we still assume that $(\mathbb{A}(\vec{u}))^i$ and $(\mathbb{B}(\vec{u})\vec{v})^i$ are defined by the right side of (\ref{e:compare.1}) and (\ref{e:compare.2}), respectively.
Then $d\mathfrak{F}_{X_{k,p}}(\vec{u})[\vec{w}]=(\mathbb{A}(\vec{u}),\vec{w})_H\;\forall\vec{w}\in X_{k,p}$,
 and Claims~\ref{cl:compare.1.2},~\ref{cl:compare.1.3} also hold
 for maps $\mathbb{A}:X_{k,p}\to X_{k,p}$ and $\mathbb{B}:X_{k,p}\to\mathscr{L}(H)$, respectively.
  Claim~\ref{cl:compare.1.2} implies that $\mathbb{B}(\vec{u})$ restricts to an element in $\mathscr{L}(X_{k,p})$,
  still denoted by $\mathbb{B}(\vec{u})$, and that $\mathbb{B}:X_{k,p}\to\mathscr{L}(X_{k,p})$ is  $C^0$.
 Moreover, if $F$ is of class $C^{k-m+3}$, the map
  $\Phi_{F^i_\alpha}: W^{k-m, p}(\Omega,\mathbb{R}^r)\to W^{k-m,p}(\Omega)$
  in the proof of Claim~\ref{cl:compare.1.2} will be of class $C^2$, and so is
 $\mathbb{A}$. This implies that $\mathbb{B}:X_{k,p}\to\mathscr{L}_s(X_{k,p})$ is $C^1$.
For a $C^2$ map $A$ from Banach spaces $X$ to $Y$ and any fixed $x_0\in X$
it easily follows from the Hahn-Banach theorem and the mean value theorem that
there exists a ball $B(x_0, r)\subset X$ centred at $x_0$ such that
$A$ is  uniformly continuously differentiable on $B(x_0, r)$.
These  and Claim~\ref{cl:compare.1.3} show that
$\mathfrak{F}_{X_{k,p}}$ is $(B(\vec{u}^\ast, r), H)$-regular for some ball
$B(\vec{u}^\ast, r)\subset X_{k,p}$. Using Proposition~\ref{prop:BifE.9.1}
we can also prove that for the  spectrum $\sigma(\mathbb{B}(\vec{u})^C)$ of
the complexification of $\mathbb{B}(\vec{u})\in\mathscr{L}_s(X_{k,p})$
either $\sigma(\mathbb{B}(\vec{u})^C)$ or $\sigma(\mathbb{B}(\vec{u})^C)\setminus\{0\}$
is bounded away from the imaginary axis, see \cite[Theorem~7.17]{Lu8} for some related proof details.
Hence Theorems~1.1,1.2 in \cite{BoBu} lead to Theorem~\ref{th:compare.1.6}.

In applications, we may use the regularity results for solutions of the Euler-Lagrangian
equations or systems to modify $F$ suitably so that useful information can be obtained
by combing the theories developed in this paper with results in this subsection.
We expect that they can be used in studies of geometric variational problems such as
minimal surfaces and harmonic maps.

\vspace{3mm}

\noindent{\bf Acknowledgement}\; %The author is grateful to the anonymous
% referees for useful remarks. I am grateful to Dr. Xiaohang Jin for identifying errors 
% in inequalities (2.7)-(2.8) of the original published version and for his valuable suggestions, 
% which led to this correction.
The author is grateful to the anonymous referees for useful remarks and to Dr. Xiaohang Jin 
for pointing out the errors in inequalities (\ref{e:S.4.1})--(\ref{e:S.4.2}) of the published version and for helpful suggestions that led to this correction.

\appendix
\section{Appendix:\quad
Comparing  {\textsf{Hypothesis} $\mathfrak{F}_{2,N,1,n}$} with controllable growth conditions}\label{app:A}\setcounter{equation}{0}

It is easily checked that {\textsf{Hypothesis} $\mathfrak{F}_{2,N,1,n}$} for $n\ge 2$ may be equivalently formulated  as

\noindent{\textsf{Hypothesis} $\mathfrak{F}_{2,N,1,n}$}.\quad
 Let  $z=(z_1,\cdots,z_N)\in \mathbb{R}^{N}$,  $p=\left(p^i_\alpha\right)\in \mathbb{R}^{N\times n}$, where $1\le i\le N$ and
 $\alpha\in \mathbb{N}_0^n$ with $|\alpha|=1$. Let
$\overline\Omega\times\mathbb{R}^N\times\mathbb{R}^{N\times n}\ni (x, z,p)\mapsto F(x, z,p)\in\R$
be twice continuously differentiable in $(z,p)$ for almost all $x$,
measurable in $x$ for all values of $(z,p)$, and $F(\cdot,z,p)\in L^1(\Omega)$ for $(z,p)=0$.
Let $\kappa_n=2n/(n-2)$ for $n>2$, and $\kappa_n\in (2,\infty)$ for $n=2$.
The derivatives of $F$ fulfill  the following properties:\\
{\bf (i)} $F_{z_i}(\cdot, 0)\in L^{\kappa_n/(\kappa_n-1)}$ and $F_{p^i_\alpha}(\cdot, 0)\in L^{2}$ for $i=1,\cdots,N$ and $|\alpha|=1$.\\
{\bf (ii)} There exist positive constants  $\mathfrak{g}_1$,  $\mathfrak{g}_2$ and
$s\in (0, \frac{\kappa_n-2}{\kappa_n})$, $r_\alpha\in (0,\frac{\kappa_n-2}{2\kappa_n})$ for each $\alpha\in \mathbb{N}_0^n$ with $|\alpha|=1$,  such that for $i,j=1,\cdots,N$, $|\alpha|=|\beta|=1$,
\begin{eqnarray*}
&|F_{p^i_\alpha p^j_\beta}(x,z,p)|\le \mathfrak{g}_1\quad\hbox{and}\\
&|F_{p^i_\alpha z_j}(x,z,p)|\le \mathfrak{g}_1\left(1+\sum^N_{l=1}|z_l|^{\kappa_n}+
\sum^N_{k=1}|p^k_\alpha|^2\right)^{r_\alpha},\\
&|F_{z_iz_j}(x,z,p)|\le \mathfrak{g}_1\left(1+\sum^N_{l=1}|z_l|^{\kappa_n}+
\sum^N_{k=1}|p^k_\alpha|^2\right)^{s},\\
&\sum^N_{i,j=1}\sum_{|\alpha|=|\beta|=1}F_{p^i_\alpha p^j_\beta}(x,z,p)\eta^i_\alpha\eta^j_\beta\ge
\mathfrak{g}_2
\sum^N_{i=1}\sum_{|\alpha|= 1}(\eta^i_\alpha)^2,\quad\forall\eta=(\eta^{i}_{\alpha})\in\R^{N\times n}.
\end{eqnarray*}

The {\it controllable growth conditions} (abbreviated to CGC below) \cite[page 40]{Gi}
(that is, the so-called `common condition of Morrey' or `the natural assumptions of Ladyzhenskaya and Ural'tseva'
\cite[page 38,(I)]{Gi}) may be, in our notation, expressed as:\\
{\bf CGC}: $\overline\Omega\times\mathbb{R}^N\times\mathbb{R}^{N\times n}\ni (x,
z,p)\mapsto F(x, z,p)\in\R$ is of class $C^2$, and
there exist positive constants $\nu, \mu, \lambda, M_1, M_2$,  such that with
$|z|^2:=\sum^N_{l=1}|z_l|^2$ and $|p|^2:=\sum_{|\alpha|=1}\sum^N_{k=1}|p^k_\alpha|^2$,
\begin{eqnarray*}
&\nu\left(1+|z|^2+|p|^2\right)-\lambda\le F(x, z,p)
\le\mu\left(1+|z|^2+|p|^2\right),\\
&|F_{p^i_\alpha}(x,z,p)|, |F_{p^i_\alpha x_l}(x,z,p)|, |F_{z_j}(x,z,p)|, |F_{z_jx_l}(x,z,p)|\le \mu\left(1+|z|^2+|p|^2\right)^{1/2},\\
&|F_{p^i_\alpha z_j}(x,z,p)|,\quad |F_{z_iz_j}(x,z,p)|\le \mu,\\
&M_1\sum^N_{i=1}\sum_{|\alpha|= 1}(\eta^i_\alpha)^2\le\sum^N_{i,j=1}\sum_{|\alpha|=|\beta|=1}F_{p^i_\alpha p^j_\beta}(x,z,p)\eta^i_\alpha\eta^j_\beta\le
M_2\sum^N_{i=1}\sum_{|\alpha|= 1}(\eta^i_\alpha)^2\\
&\forall\eta=(\eta^{i}_{\alpha})\in\R^{N\times n}.
\end{eqnarray*}
Moreover, if $F=F(x,p)$ does not depend explicitly on $z$, the first three lines are replaced by
\begin{eqnarray*}
&\nu\left(1+|p|^2\right)-\lambda\le F(x,p)
\le\mu\left(1+|p|^2\right)\quad\hbox{and}\\
&|F_{p^i_\alpha}(x,p)|,\quad |F_{p^i_\alpha x_l}(x,p)|\le \mu\left(1+|p|^2\right)^{1/2}.
\end{eqnarray*}

From these it is not hard to see

\begin{proposition}\label{prop:A.1}
 {\bf CGC}  implies {\textsf{Hypothesis} $\mathfrak{F}_{2,N,1,n}$}.
\end{proposition}

\medskip
\quad% The data information below will be filled by AIMS editorial staff
%Received August 2011; revised October 2012.
\medskip
\end{document}